\documentclass[11pt]{amsart}

\usepackage[OT1]{fontenc}
\usepackage{amsmath,amssymb}
\usepackage{bbm}
\usepackage{xcolor}
\usepackage{graphicx}
\usepackage{paralist}
\usepackage{tabularx,booktabs}
\usepackage[mathscr]{eucal}
\usepackage{units}
\usepackage{cite}
\usepackage[format=hang,width=\linewidth]{caption}
\usepackage{upgreek}
\usepackage{url}
\usepackage[colorinlistoftodos,
            bordercolor=orange,
            backgroundcolor=orange!20,
            linecolor=orange,
            textsize=scriptsize]{todonotes}

\usepackage{color}
\definecolor{darkblue}{rgb}{.8,.15,.15}
\definecolor{darkgreen}{rgb}{0.15,.4,.5}

\usepackage[pagebackref,
            naturalnames,
            colorlinks,linkcolor=darkblue,
            citecolor=darkblue,
            filecolor=darkblue,urlcolor=darkgreen,
            anchorcolor=darkblue,menucolor=darkblue,
            bookmarksopen=false,bookmarks=false
           ]{hyperref}

\renewcommand*{\backref}[1]{}
\renewcommand*{\backrefalt}[4]{%
\ifcase #1 %
 [Not cited]%
\or
 [Cited on page #2]%
\else
 [Cited on pages #2]%
\fi
}

\usepackage{tikz}
\usetikzlibrary{arrows}

\newtheorem{theorem}{Theorem}[section]
\newtheorem{corollary}[theorem]{Corollary}
\newtheorem{lemma}[theorem]{Lemma}
\newtheorem{proposition}[theorem]{Proposition}

\newtheorem*{question}{Question}
\newtheorem*{theorem*}{Theorem}

\theoremstyle{definition}
\newtheorem{definition}[theorem]{Definition}
\newtheorem{example}[theorem]{Example}
\newtheorem{remark}[theorem]{Remark}

\makeatletter
\newtheorem*{rep@theorem}{\rep@title}
\newenvironment{reptheorem}[1]{%
  \def\rep@title{Theorem #1}%
  \begin{rep@theorem}}%
  {\end{rep@theorem}}
\makeatother

\numberwithin{theorem}{section}
\numberwithin{figure}{section}
\numberwithin{equation}{section}
\numberwithin{table}{section}

\newcommand{\va}{{\boldsymbol{a}}}
\newcommand{\vb}{{\boldsymbol{b}}}
\newcommand{\vc}{{\boldsymbol{c}}}
\newcommand{\vd}{{\boldsymbol{d}}}
\newcommand{\ve}{{\boldsymbol{e}}}
\newcommand{\vu}{{\boldsymbol{u}}}
\newcommand{\vl}{{\boldsymbol{l}}}

\newcommand{\vx}{{\boldsymbol{x}}}
\newcommand{\vy}{{\boldsymbol{y}}}
\newcommand{\vn}{{\boldsymbol{n}}}
\newcommand{\vm}{{\boldsymbol{m}}}
\newcommand{\vp}{{\boldsymbol{p}}}
\newcommand{\vq}{{\boldsymbol{q}}}
\newcommand{\vt}{{\boldsymbol{t}}}
\newcommand{\vr}{{\boldsymbol{r}}}
\newcommand{\vchi}{{\boldsymbol{\chi}}}
\newcommand{\vv}{{\boldsymbol{v}}}
\newcommand{\vw}{{\boldsymbol{w}}}
\newcommand{\n}{{\boldsymbol{n}}}
\newcommand{\m}{{\boldsymbol{m}}}
\newcommand{\x}{\boldsymbol{x}}
\renewcommand{\t}{\boldsymbol{t}}
\newcommand{\y}{\boldsymbol{y}}
\newcommand{\e}{\boldsymbol{e}}
\newcommand{\0}{\boldsymbol{0}}
\newcommand{\1}{\boldsymbol{1}}
\newcommand{\kk}{\boldsymbol{k}}

\newcommand{\R}{\mathbb{R}}
\newcommand{\Z}{\mathbb{Z}}
\newcommand{\C}{\mathbb{C}}
\newcommand{\Q}{\mathbb{Q}}
\newcommand{\N}{\mathbb{N}}
\newcommand{\K}{\boldsymbol{K}}

\newcommand{\B}{\mathbf{B}}
\let\pS\S
\renewcommand{\S}{\mathcal{S}}
\newcommand{\T}{\mathcal{T}}

\newcommand{\rootA}{\mathbf A}
\newcommand{\rootB}{\mathbf B}
\newcommand{\rootC}{\mathbf C}
\newcommand{\rootD}{\mathbf D}

\newcommand{\rootF}{{\mathbf F}_4}
\newcommand{\rootG}{{\mathbf G}_2}
\newcommand{\rootE}{{\mathbf E}}

\newcommand{\configuration}{\mathscr A}

\newcommand{\calS}{\mathcal{S}}
\newcommand{\calA}{\mathcal{A}}

\newcommand{\calVC}{\mathcal{VC}}
\newcommand{\calC}{\mathcal{C}}

\newcommand{\ideal}{\mathscr{I}}
\DeclareMathOperator{\lead}{in}

\newcommand{\vomega}{{\boldsymbol{\omega}}}
\newcommand{\vvarepsilon}{{\boldsymbol{\varepsilon}}}
\newcommand{\veta}{{\boldsymbol{\eta}}}

\newcommand{\HS}{\Delta}
\newcommand{\zo}{\ensuremath{0/1}}
\newcommand{\cut}{\mathsf{Cut}}

\newcommand{\dddots}{\reflectbox{$\ddots$}}

\newcommand{\vertices}[1]{\mathcal V(#1)}
\newcommand{\hilb}{\mathcal{H}}
\newcommand{\conv}{\operatorname{conv}}

\newcommand{\aff}{\operatorname{aff}}
\newcommand{\vol}{\operatorname{vol}}
\newcommand{\spec}{\operatorname{Spec}}
\newcommand{\proj}{\operatorname{Proj}}
\newcommand{\pull}{\operatorname{pull}}
\DeclareMathOperator{\cone}{cone}
\DeclareMathOperator{\interior}{int}

\DeclareMathOperator{\id}{id}
\DeclareMathOperator{\Star}{star}

\DeclareMathOperator{\NP}{NP} %
\DeclareMathOperator{\supp}{supp}
\DeclareMathOperator{\std}{std}
\DeclareMathOperator{\lecturehallsimplexoperator}{LHS}
\newcommand{\lecturehallsimplex}[1]{\lecturehallsimplexoperator_{#1}}
\newcommand{\s}{\boldsymbol{s}}
\newcommand{\slecturehall}[2][\s]{\lecturehallsimplex{#2}(#1)}

\newcommand{\stdo}{\std_\vomega}

\newcommand{\polymake}{\texttt{polymake}}
\def\cocoa{{\hbox{\rm C\kern-.13em o\kern-.07em C\kern-.13em o\kern-.15em A}}}

\newcommand\aparhref[2]{\href{http://arxiv.org/abs/#1}{#2}}

\newcommand{\transport}{{\operatorname{T}_{\vr\vc}}}

\title[Unimodular triangulations]{\mbox{}\vspace{-3\baselineskip}
  \\
  Existence of unimodular triangulations \\ --- \\ positive results
}

\author[Haase]{Christian Haase}
\thanks{Work of the first author was supported in part by NSF grant
  DMS-0200740, and by Emmy Noether HA 4383/1 and Heisenberg HA
  4383/4 fellowships of the German Research Council (DFG)}
\address{Christian Haase, Freie Universit\"at Berlin, Germany}
\email{haase@math.fu-berlin.de}

\author[Paffenholz]{Andreas Paffenholz}
\thanks{Work of the second author was supported by a Priority Program (DFG/SPP
1489) of the German Research Council (DFG)}
\address{Andreas Paffenholz \\ Technische Universit\"at Darmstadt \\ Germany} 
\email{paffenholz@mathematik.tu-darmstadt.de}

\author[Piechnik]{Lindsay C. Piechnik}
\address{Lindsay C. Piechnik \\ High Point University, NC \\ USA}
\email{lpiechni@highpoint.edu}

\author[Santos]{Francisco Santos\vspace{-\baselineskip}}
\address{Francisco Santos \\ Universidad de Cantabria \\ Spain}
\email{santosf@unican.es} \thanks{Work of the fourth author has been
  supported in part by the Spanish Ministry of Science through grants
  MTM2011-22792 and MTM2014-54207-P, by  the Alexander von Humboldt Foundation, by the Einstein Foundation, and by the NSF while he was in residence at MSRI Berkeley in the fall of 2017.}

\usepackage[labeled]{multibib}
\newcites{todo}{Articles not cited}

\usepackage[pagewise]{lineno}

\date{\today}
\subjclass[2010]{Primary 52B20, 52B11;\\Secondary 13F20, 52C11}

\begin{document}

\begin{abstract}
  Unimodular triangulations of lattice polytopes 
  arise in algebraic geometry, commutative algebra, integer
  programming and, of course, combinatorics.

   In this article, we review several classes of polytopes 
  that do have unimodular triangulations and constructions that preserve 
  their existence.

  We include, in particular, the first effective proof of the
  classical result by Knudsen-Mumford-Waterman stating that every
  lattice polytope has a dilation that admits a unimodular
  triangulation.
  Our proof yields an explicit (although doubly exponential) bound for
  the dilation factor.
\end{abstract}
\maketitle

{\tableofcontents}

\section{Introduction}

Unimodular triangulations of lattice polytopes arise in algebraic geometry, commutative algebra,
integer programming and, of course, combinatorics.  Admitting a unimodular triangulation is a
property with nice implications, but presumably ``most'' lattice polytopes do not have such a
triangulation (see Question~\ref{question:most_do_not} for a precise statement).
In this paper we review methods for constructing unimodular triangulations and classes of polytopes
which are known to have unimodular triangulations.

After defining the basic objects and fixing notation in Section~\ref{sec:what}, in
Section~\ref{sec:why} we explain why unimodular triangulations are important. We then embed the
property of having a unimodular triangulation into a hierarchy of related properties in
Section~\ref{sec:hierarchy} and preview what does and does not appear in this paper in
Sections~\ref{sec:results} and~\ref{sec:whatnot}, before listing some open questions in
Section~\ref{sec:questions}.

The rest of the paper is divided into three parts. In Section~\ref{sec:methods} we present methods
of constructing unimodular triangulations, as well as how unimodular triangulations of certain
polytopes can be used in constructions of unimodular triangulations for more complicated cases.
This includes the analysis of joins, products, and projections. In
Section~\ref{sec:toric-groebner-bases} we review the relationship between unimodular triangulations
and Gr\"obner bases of toric ideals, a correspondence that can be exploited in both directions.

Section~\ref{sec:Examples} deals with particular classes of polytopes that can be shown to have
unimodular triangulations and classes for which the question of existence has been addressed or is
especially interesting. This includes polytopes related to root systems
(Sections~\ref{sec:cut-by-roots} and~\ref{sec:spanned-by-roots}), polytopes related to graphs
(Section~\ref{sec:graph-polytopes}) and smooth polytopes (Section ~\ref{sec:Smooth}). In particular,
we show that all polytopes with facet normals in the root systems $\rootB_n$ have regular unimodular
triangulations, a result that was not known before. (The same result for $\rootA_n$ easily follows
from total unimodularity.) We also include previously unpublished results about empty smooth polytopes
and smooth reflexive polytopes (e.g.~Thm~\ref{thm:empty-smooth} and Thm~\ref{thm:smoothreflexive}). 

Section~\ref{sec:kmw} addresses dilations of lattice polytopes: when dilations must admit unimodular
triangulations, and the process of constructing them.  In particular, we provide the first effective
version of the celebrated result of Knudsen, Mumford and Waterman saying that every lattice polytope
$P$ has a dilation $cP$ that admits unimodular triangulations. Previous proofs of the theorem, although they can be considered algorithmic, do not
mention an explicit bound for two reasons:
\begin{inparaenum}
\item it is not easy to derive a bound from the algorithm,
\item the bound obtained would contain arbitrarily long towers of exponentials.
\end{inparaenum}
Our bound, instead, is ``only'' doubly exponential:

\begin{reptheorem}{\ref{thm:KMW-effective}}
\label{thm:KMW-effective-intro}
If $P$ is a lattice polytope of dimension $d$ and (lattice) volume $\vol(P)$, then the dilation
\[
(d+1)!^{\vol(P)! {(d+1)^{(d+1)^2\vol(P)}}} P
\]
has a regular unimodular triangulation.

More precisely, if $P$ has a triangulation $\T$ into $N$ $d$-simplices, of volumes $V_1,\dots,V_N$, then the dilation 
\[
(d+1)!^{\sum_{i=1}^N  V_i! \left( (d+1)!^{d+1}\right)^{V_i-1}} \T
\]
has a regular unimodular refinement.
\end{reptheorem}

\subsection{What?} \label{sec:what} 

A \emph{lattice polytope} in $\R^d$ is the convex hull of finitely many points in the lattice
$\Z^d$. We identify two lattice polytopes if they are related by a lattice preserving affine map. Up
to this \emph{lattice equivalence}, we can always assume that our polytope is $d$-dimensional. (For
more on convex polytopes and lattices we refer to~\cite{BarvinokBook}.)

A \emph{unimodular simplex} is a lattice polytope which is lattice equivalent to the standard
simplex $\Delta^d$, the convex hull of the origin $\0$ together with the standard unit vectors
$\ve_i$ ($1 \le i \le d$). Equivalently, unimodular simplices are characterized as the
$d$-dimensional lattice polytopes of minimal possible Euclidean volume, $1/d\textit{!}$~.

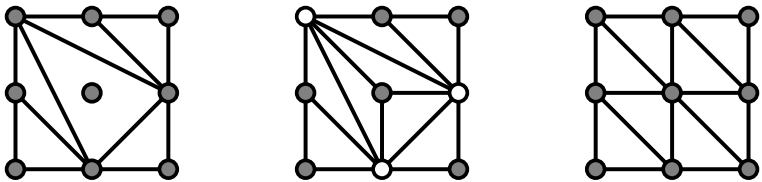
\begin{figure}[tb]
  \centering
\begin{tikzpicture}[scale=.8,y=-1cm]
\draw[line width=1.5pt,black] (11.43,3.81) -- (13.97,3.81) -- (13.97,5.08) -- (13.97,6.35) -- (12.7,6.35) -- (11.43,6.35) -- (11.43,5.08) -- cycle;
\draw[line width=1.5pt,black] (12.7,6.35) -- (11.43,3.81) -- (13.97,5.08) -- (12.7,6.35) -- (11.43,5.08);
\draw[line width=1.5pt,black] (13.97,5.08) -- (12.7,3.81);

\draw[line width=1.4pt,color=black,fill=black!50] (11.43,5.08) circle (4.2pt);
\draw[line width=1.4pt,color=black,fill=black!50] (11.43,3.81) circle (4.2pt);
\draw[line width=1.4pt,color=black,fill=black!50] (12.7,3.81) circle (4.2pt);
\draw[line width=1.4pt,color=black,fill=black!50] (12.7,5.08) circle (4.2pt);
\draw[line width=1.4pt,color=black,fill=black!50] (12.7,6.35) circle (4.2pt);
\draw[line width=1.4pt,color=black,fill=black!50] (13.97,5.08) circle (4.2pt);
\draw[line width=1.4pt,color=black,fill=black!50] (13.97,3.81) circle (4.2pt);
\draw[line width=1.4pt,color=black,fill=black!50] (11.43,6.35) circle (4.2pt);
\draw[line width=1.4pt,color=black,fill=black!50] (13.97,6.35) circle (4.2pt);
\end{tikzpicture}%
\qquad\qquad
\begin{tikzpicture}[scale=.8,y=-1cm]

\draw[line width=1.5pt,black] (11.43,3.81) -- (13.97,3.81) -- (13.97,5.08) -- (13.97,6.35) -- (12.7,6.35) -- (11.43,6.35) -- (11.43,5.08) -- cycle;
\draw[line width=1.5pt,black] (12.7,6.35) -- (11.43,3.81) -- (13.97,5.08) -- (12.7,6.35) -- (11.43,5.08);
\draw[line width=1.5pt,black] (13.97,5.08) -- (12.7,3.81);
\draw[line width=1.5pt,black] (11.43,3.81) -- (12.7,5.08) -- (12.7,6.35);
\draw[line width=1.5pt,black] (12.7,5.08) -- (13.97,5.08);

\draw[line width=1.4pt,color=black,fill=black!50] (11.43,5.08) circle (4.2pt);
\draw[line width=1.4pt,color=black,fill=black!50] (12.7,3.81) circle (4.2pt);
\draw[line width=1.4pt,color=black,fill=black!50] (12.7,5.08) circle (4.2pt);
\draw[line width=1.4pt,color=black,fill=black!50] (13.97,3.81) circle (4.2pt);
\draw[line width=1.4pt,color=black,fill=black!50] (11.43,6.35) circle (4.2pt);
\draw[line width=1.4pt,color=black,fill=black!50] (13.97,6.35) circle (4.2pt);
\path[line width=1.4pt,draw=black,fill=white] (11.43,3.81) circle (4.2pt);
\path[line width=1.4pt,draw=black,fill=white] (12.7,6.35) circle (4.2pt);
\path[line width=1.4pt,draw=black,fill=white] (13.97,5.08) circle (4.2pt);
\end{tikzpicture}%
\qquad\qquad
\begin{tikzpicture}[scale=.8,y=-1cm]
\draw[line width=1.5pt,black] (11.43,3.81) -- (13.97,3.81) -- (13.97,5.08) -- (13.97,6.35) -- (12.7,6.35) -- (11.43,6.35) -- (11.43,5.08) -- cycle;
\draw[line width=1.5pt,black] (13.97,5.08) -- (12.7,3.81);
\draw[line width=1.5pt,black] (11.43,3.81) -- (13.97,6.35);
\draw[line width=1.5pt,black] (12.7,3.81) -- (12.7,6.35) -- (11.43,5.08) -- (13.97,5.08);
\draw[line width=1.4pt,color=black,fill=black!50] (11.43,5.08) circle (4.2pt);
\draw[line width=1.4pt,color=black,fill=black!50] (11.43,3.81) circle (4.2pt);
\draw[line width=1.4pt,color=black,fill=black!50] (12.7,3.81) circle (4.2pt);
\draw[line width=1.4pt,color=black,fill=black!50] (12.7,5.08) circle (4.2pt);
\draw[line width=1.4pt,color=black,fill=black!50] (12.7,6.35) circle (4.2pt);
\draw[line width=1.4pt,color=black,fill=black!50] (13.97,5.08) circle (4.2pt);
\draw[line width=1.4pt,color=black,fill=black!50] (13.97,3.81) circle (4.2pt);
\draw[line width=1.4pt,color=black,fill=black!50] (11.43,6.35) circle (4.2pt);
\draw[line width=1.4pt,color=black,fill=black!50] (13.97,6.35) circle (4.2pt);
\end{tikzpicture}%

   \caption{A non-unimodular triangulation and two unimodular ones. They are all regular, but only
    the last is quadratic.  (See Section~\ref{sec:quadratic}).}
  \label{fig:triangulation}
\end{figure}

For the purposes of this paper, a (lattice) \emph{subdivision} of a $d$-di\-men\-sio\-nal lattice
polytope $P$ is a finite collection of (lattice) polytopes $\S$ such that
\begin{enumerate}
\item every face of a member of $\S$ is in $\S$, \label{subd:faces}
\item any two elements of $\S$ intersect in a common (possibly empty) face,
  and \label{subd:intersection}
\item the union of the polytopes in $\S$ is $P$. \label{subd:union}
\end{enumerate}
The maximal ($d$-dimensional) polytopes in $\S$ are called \emph{cells} of $\S$.

A \emph{triangulation} is a subdivision of a polytope for which each cell of the subdivision is a
simplex. The triangulation is \emph{unimodular} if every cell is.  Figure~\ref{fig:triangulation}
depicts three triangulations of the nine-point square. The first is not unimodular, while the other
two are.

A \emph{full triangulation} is a lattice triangulation which uses all the lattice points in $P$. The
triangulation on the left in Figure~\ref{fig:triangulation} is not full.  Every subdivision has a
refinement to a full triangulation, for example the one resulting from the strong pulling procedure
discussed in Section~\ref{sec:paco}.  Also, every unimodular triangulation is full, and in dimension
at most two the converse is true as well. Depending on one's perspective, this is a consequence of,
or the reason for, Pick's formula, which says that the area a polygon is one less than its number of
interior lattice points plus half the number of lattice points on its boundary~\cite{pick}. The
formula yields the following proposition.
\begin{proposition}
  Every lattice polygon has a unimodular triangulation.
\end{proposition}

However, there are three dimensional polytopes for which a unimodular triangulation does not exist.
\begin{example}[John Reeve~\cite{Reeve}]
  For $q \in \Z_{>0}$, the tetrahedron in Figure~\ref{fig:reeve}
  \begin{figure}[tb]
    \centering
    \raisebox{5mm}{
    $\displaystyle 
    \mathsf{reeve}(q)\ :=\ \conv \left[ \begin{smallmatrix}
        0&0&0&q\\
        0&0&1&1\\
        0&1&0&1
      \end{smallmatrix}
    \right]$}
    \qquad \qquad
\begin{tikzpicture}[y=-1cm,scale=1]
\draw[join=bevel,black] (3.175,6.6675) -- (7.46125,6.6675);
\draw[join=bevel] (3.175,6.6675) -- (3.175,6.50875);
\draw[join=bevel] (3.175,6.0325) -- (3.175,5.3975) -- (2.54,6.35);
\draw[join=bevel] (8.255,5.3975) -- (7.62,6.35);
\draw[join=bevel] (8.255,5.3975) -- (8.255,6.6675) -- (7.62,7.62);

\draw[line width=6.3bp,join=bevel,color=white] (4.1275,6.985) -- (6.985,5.87375);

\draw[very thick,join=bevel,black] (8.255,5.3975) -- (2.54,6.35);
\draw[very thick,join=bevel,color=black] (2.54,6.35) -- (2.54,7.62);
\draw[very thick,join=bevel,dashed,black] (2.54,7.62) -- (3.175,6.6675);
\draw[very thick,join=bevel,dashed,black] (3.175,6.6675) -- (2.54,6.35);
\draw[very thick,join=bevel,dashed,black] (3.175,6.6675) -- (8.255,5.3975);
\draw[very thick,join=bevel,black] (2.54,7.62) -- (8.255,5.3975);
\draw[join=bevel] (7.62,6.35) -- (7.62,7.62);
\draw[join=bevel] (7.77875,6.6675) -- (8.255,6.6675);

\draw[line width=6.3bp,join=bevel,color=white] (3.81,6.35) -- (6.6675,6.35);

\fill[draw=black] (2.54,6.35) circle (0.127cm);
\fill[draw=black] (8.255,5.3975) circle (0.127cm);
\fill[draw=black] (2.54,7.62) circle (0.127cm);
\fill[draw=black] (3.175,6.6675) circle (0.127cm);
\draw[join=bevel,black] (2.54,7.62) -- (7.62,7.62);
\draw[join=bevel,black] (3.175,5.3975) -- (8.255,5.3975);
\draw[join=bevel,black] (2.54,6.35) -- (7.62,6.35);
\end{tikzpicture}%
     \caption{Reeve's tetrahedra $\mathsf{reeve}(q)$ for integral
    non-negative $q$.}
    \label{fig:reeve}
  \end{figure}
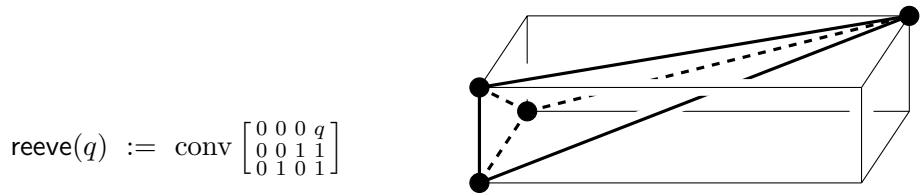
  contains only four lattice points --- the vertices.  Its only lattice triangulation is the trivial
  one.  As the Euclidean volume is equal to $q/6$, this simplex does not have a unimodular
  triangulation for $q>1$.
\end{example}

A subdivision is \emph{regular} if its cells are the domains of linearity of a convex piecewise linear
function. (Compare~\cite[Section~14.3]{LeeHandBook},~\cite{deLoeraRambauSantos}.) Less formally, a
regular subdivision can be thought of as a subdivision that can be realized as a ``convex
folding" of the polytope (Figure~\ref{fig:nonRegular} on the left).  All three triangulations in
Figure~\ref{fig:triangulation} are regular while the triangulation on the right in
Figure~\ref{fig:nonRegular} is not.

More formally, to define a regular subdivision \label{def:regular} of an arbitrary lattice polytope $P$ (or to certify that a given subdivision is regular) we need a set of  \emph{weights} (or \emph{heights}) $\vomega \in \R^\configuration$, where $\configuration = P \cap \Z^d$ is the set of lattice points in $P$. 
A lattice subpolytope $F$ of $P$ is a cell  in the regular subdivision 
$\S_\vomega$ of $P$ corresponding to those weights if there is a $\veta_F \in
\R^\configuration$ and a $\zeta_F \in \R$ such that
\begin{equation}
  \label{eq:face-of-regular}
  \omega_\va \ge \langle \veta_F, \va \rangle + \zeta_F \quad \text{for all $\va \in \configuration$}
\end{equation}
and 
\begin{equation}
  \label{eq:face-of-regular-equal}
  F = \conv \left\{ \va \in \configuration \ : \ \omega_\va = \langle \veta_F, \va \rangle + \zeta_F
  \right\} \,.
\end{equation}
This can also be viewed geometrically. For this, consider the polyhedron $\tilde P = \conv ( \va \times [\omega_\va, \infty) \ : \ \va
\in \configuration )$ in $\R^{d+1}$. The bounded faces of $\tilde P$ (which are also the lower faces if the last coordinate is considered as a height function) project to the faces of
$\S_\vomega$. The latter are the domains of linearity of the function 
$$\vx \mapsto \min \bigl\{ h \ : \ (\vx,h) \in \tilde P \bigr\} = \max \bigl\{ \langle \veta_F, \vx
  \rangle + \zeta_F \ : \ F \in \S_\vomega \bigr\} \,.$$

Conversely, we can of course prove that a given subdivision of a polytope is regular by finding an appropriate set of weights that generates the subdivision. 
A particular method for constructing regular full triangulations for an arbitrary lattice polytope
is given in Lemma~\ref{lemma:strongPulling}.

A \emph{non-face} of a triangulation is a set of points whose convex hull does not form a
face. Particularly important are the \emph{minimal non-faces}, which do not form faces but for which
every proper subset does. The list of minimal non-faces completely characterizes a triangulation, as
does the list of (sets of points that form) cells. A triangulation for which all minimal non-faces
contain only two elements is called \emph{flag}. Putting these properties together, a
\emph{quadratic triangulation} is defined as a regular unimodular flag triangulation.  The
triangulation on the right in Figure~\ref{fig:triangulation} is quadratic. However, the three white
vertices of the triangulation in the middle form a minimal non-face. So that triangulation is not
quadratic.
\begin{figure}[tb]
  \centering
\begin{tikzpicture}[y=-1cm,scale=.3]

\path[draw=black,fill=black!15] (12.7,1.905) -- (16.8275,8.89) -- (22.86,3.81) -- cycle;
\path[draw=black,fill=black!15] (12.7,1.905) -- (5.08,3.81) -- (11.1125,8.255);
\path[draw=black,fill=black!15,dashed] (5.08,16.51) -- (5.08,3.81);
\path[draw=black,fill=black!15,dashed] (11.1125,15.24) -- (11.1125,8.255);
\path[draw=black,fill=black!15,dashed] (22.86,12.7) -- (22.86,3.81);
\path[draw=black,fill=black!15,dashed] (12.7,11.43) -- (12.7,8.255);
\path[draw=black,fill=black!15,dashed] (13.335,13.335) -- (13.335,8.255);
\path[draw=black,fill=black!25] (12.7,1.905) -- (11.1125,8.255) -- (16.8275,8.89) -- cycle;

\path[draw=black,fill=black!15,dashed] (13.335,5.87375) -- (13.335,4.445);
\path[draw=black,fill=black!15,dashed] (14.605,8.255) -- (14.605,3.175);

\path[draw=black,fill=black!50] (15.24,7.62) -- (22.86,3.81) -- (16.8275,8.89);
\path[draw=black,fill=black!50] (11.1125,8.255) -- (5.08,3.81) -- (15.24,7.62);
\path[draw=black,fill=black!60] (15.24,7.62) -- (11.1125,8.255) -- (16.8275,8.89) -- cycle;

\draw[color=black,fill=black!50,line width=1.5pt] (5.08,16.51) circle (8pt);
\draw[color=black,fill=black!50,line width=1.5pt] (12.7,11.43) circle (8pt);
\draw[color=black,fill=black!50,line width=1.5pt] (15.24,17.78) circle (8pt);
\draw[color=black,fill=black!50,line width=1.5pt] (22.86,12.7) circle (8pt);
\draw[color=black,fill=black!50,line width=1.5pt] (11.1125,15.24) circle (8pt);
\draw[color=black,fill=black!50,line width=1.5pt] (16.8275,13.97) circle (8pt);
\draw[color=black,fill=black!50,line width=1.5pt] (13.335,13.335) circle (8pt);
\draw[color=black,fill=black!50,line width=1.5pt] (14.605,15.875) circle (8pt);
\draw[color=black,fill=black!50,line width=1.5pt] (11.1125,8.255) circle (8pt);
\draw[color=black,fill=black!50,line width=1.5pt] (16.8275,8.89) circle (8pt);
\draw[color=black,fill=black!50,line width=1.5pt] (22.86,3.81) circle (8pt);
\draw[color=black,fill=black!50,line width=1.5pt] (12.7,1.905) circle (8pt);
\draw[color=black,fill=black!50,line width=1.5pt] (13.335,4.445) circle (8pt);
\draw[color=black,fill=black!50,line width=1.5pt] (14.605,3.175) circle (8pt);
\draw[color=black,fill=black!50,line width=1.5pt] (5.08,3.81) circle (8pt);
\draw[color=black,fill=black!50,line width=1.5pt] (15.24,7.62) circle (8pt);
\draw[black] (5.08,16.51) -- (15.24,17.78);
\draw[black] (15.24,17.78) -- (22.86,12.7) -- (12.7,11.43) -- (5.08,16.51);
\fill[draw=black,dotted] (5.08,16.51) -- (22.86,12.7);
\draw[dotted,black] (12.7,11.43) -- (11.1125,15.24) -- (15.24,17.78) -- (16.8275,13.97) -- cycle;

\path[draw=black,fill=black!15,dashed] (16.8275,13.97) -- (16.8275,8.89);
\path[draw=black,fill=black!15,dashed] (15.24,17.78) -- (15.24,7.62);
\path[draw=black,fill=black!15,dashed] (14.605,15.875) -- (14.605,8.41375);

\end{tikzpicture}%
   \qquad \qquad
\begin{tikzpicture}[y=-1cm]

\draw[black,line width=1.4pt] (11.43,3.81) -- (11.43,7.62) -- (15.24,7.62) -- (15.24,3.81) -- (11.43,3.81) -- (15.24,7.62);
\draw[black,line width=1.4pt] (11.43,7.62) -- (12.7,6.35) -- (12.7,7.62);
\draw[black,line width=1.4pt] (12.7,6.35) -- (13.97,7.62);
\draw[black,line width=1.4pt] (12.7,6.35) -- (15.24,7.62);
\draw[black,line width=1.4pt] (11.43,5.08) -- (12.7,5.08) -- (11.43,6.35);
\draw[black,line width=1.4pt] (11.43,7.62) -- (12.7,5.08) -- (12.7,6.35) -- (13.97,6.35) -- (15.24,6.35);
\draw[black,line width=1.4pt] (15.24,5.08) -- (13.97,6.35) -- (15.24,3.81) -- (13.97,5.08) -- (13.97,3.81);
\draw[black,line width=1.4pt] (13.97,6.35) -- (13.97,5.08) -- (12.7,3.81);
\draw[black,line width=1.4pt] (12.7,5.08) -- (13.97,5.08) -- (11.43,3.81);

\draw[color=black,fill=black!50,line width=1.5pt] (11.43,5.08) circle (4pt);
\draw[color=black,fill=black!50,line width=1.5pt] (11.43,3.81) circle (4pt);
\draw[color=black,fill=black!50,line width=1.5pt] (12.7,3.81) circle (4pt);
\draw[color=black,fill=black!50,line width=1.5pt] (12.7,5.08) circle (4pt);
\draw[color=black,fill=black!50,line width=1.5pt] (12.7,6.35) circle (4pt);
\draw[color=black,fill=black!50,line width=1.5pt] (13.97,5.08) circle (4pt);
\draw[color=black,fill=black!50,line width=1.5pt] (13.97,3.81) circle (4pt);
\draw[color=black,fill=black!50,line width=1.5pt] (11.43,6.35) circle (4pt);
\draw[color=black,fill=black!50,line width=1.5pt] (13.97,6.35) circle (4pt);
\draw[color=black,fill=black!50,line width=1.5pt] (11.43,7.62) circle (4pt);
\draw[color=black,fill=black!50,line width=1.5pt] (12.7,7.62) circle (4pt);
\draw[color=black,fill=black!50,line width=1.5pt] (13.97,7.62) circle (4pt);
\draw[color=black,fill=black!50,line width=1.5pt] (15.24,7.62) circle (4pt);
\draw[color=black,fill=black!50,line width=1.5pt] (15.24,6.35) circle (4pt);
\draw[color=black,fill=black!50,line width=1.5pt] (15.24,5.08) circle (4pt);
\draw[color=black,fill=black!50,line width=1.5pt] (15.24,3.81) circle (4pt);
\end{tikzpicture}%
   \caption{Regular vs. non-regular subdivisions.}
  \label{fig:nonRegular}
\end{figure}
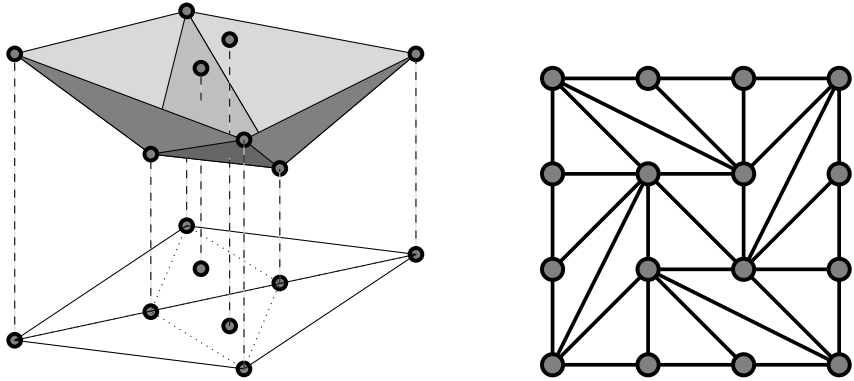

\subsection{Why? Who?} \label{sec:why} 

In this section, we present some applications of unimodular triangulations and closely related
objects (``Why?''), arranged by mathematical discipline (``Who?'').

\subsubsection{Enumerative combinatorics}
\label{sec:enum_combi}

Many counting problems can be phrased as counting lattice points in (dilates) of polytopes or
polyhedral complexes~\cite{deLoeraSurvey,CCD}.  By a fundamental result of Ehrhart, the number of
lattice points in positive dilates $kP$ of $P$ is a polynomial function of degree $d$ in $k \in
\Z_{>0}$~\cite{EhrhartII,CCD}. Consequently, the generating function has the special form:
\begin{align*}
  \sum_{k \ge 0} \#(kP \cap \Z^d) \ t^k \ = \ \frac{h^*(t)}{(1-t)^{d+1}}\,,
\end{align*}
where $h^*(t)$ is a polynomial of degree $\le d$.  If $P$ has a unimodular triangulation, $h^*$
equals the combinatorial $h$-polynomial of that triangulation~\cite{BetkeMcMullen}. This means if
different unimodular triangulations exist they have the same $f$-vector.

For example, the chromatic polynomial of a graph can be interpreted as the Ehrhart polynomial of the
complement of a hyperplane arrangement in the 0/1-cube~\cite{BeckZaslavsky}. This hyperplane
arrangement is compatible with the standard unimodular triangulation of the cube as the order
polytope of an anti-chain (see Section~\ref{sec:type-A-facets}). This can be used to compute the
chromatic polynomial in terms of Steingr\'{\i}msson's coloring complex~\cite{Steingrimsson}. Hersh
and Swartz used this to derive rather strong conditions for a polynomial to be a chromatic
polynomial of a graph \cite{HershSwartzColoring}.  A very similar argument applies for polynomials
counting nowhere-zero flows or nowhere-zero tensions~\cite{BreuerDall} and even for the Tutte
polynomial of a graph~\cite{BreuerSanyal}.

In combinatorial representation theory, Littlewood-Richardson coefficients in type $A$, and more
general Clebsch Gordan coefficients in other types, can be represented as the number of lattice
points in specific polytopes (\cite{BerensteinZelevinskyI,BerensteinZelevinskyII} and
\cite{KnutsonTaoSaturation,BuchSaturation}).  In type $A$, Knutson and Tao prove the Saturation
Theorem that implies Horn's Conjecture by showing that the hives polytope has an integral vertex
whenever it is non-empty.  This result would also follow from De Loera and McAllister's
conjecture~\cite[Conj.~4.5]{deLoeraMcAllisterDilation} that the so called homogenized hive matrix
has a unimodular triangulation, which they have validated up to $A_6$.

Other well known enumerative results involving unimodular triangulations include the proof of the
hook-length formula by Pak~\cite{PakHookLength} and Stanley's observation that Eulerian numbers are
volumes of hypersimplices~\cite{FoataEulerian,StanleyEulerian} (see also Section~\ref{sec:type-A-facets}).

\subsubsection{Integer programming} \label{sec:integer-programming} 

Embedding a polytope at height one as $P \times \{1\} \subset \R^{d+1}$ generates a pointed
polyhedral cone $\sigma_P \subset \R^{d+1}$. The \emph{Hilbert basis}, $\hilb$, of this cone is the
minimal generating set for the semigroup $\sigma_P \cap \Z^{d+1}$.  $\hilb$ clearly contains the set
of lattice points in $P \times \{1\}$. If the converse holds then $P$ is called \emph{integrally
  closed} or to have the \emph{integer decomposition property}.  A sufficient condition for this to happen is that $P$ has a unimodular triangulation or,
more generally, that $P$ can be covered by unimodular simplices.  In fact, there is a hierarchy of
covering properties interpolating between integrally closed polytopes and polytopes with unimodular
triangulations (see Section ~\ref{sec:hierarchy}).

This Hilbert basis hierarchy appears in integer programming in two guises: test sets and TDI
systems.  For the first, suppose we want to solve, for a fixed matrix $A \in \Z^{d \times n}$, the
family of integer programs $\operatorname{IP}(\vb; \vc; \vu)$ given by
$$\min \left\{ \, \langle \vc,\vx \rangle \ : \ \vx \in \Z^n,\
  A\vx=\vb,\ \0 \le \vx \le \vu \, \right\}$$ for varying $\vb \in \Z^d$, $\vu \in \Z^n$, $\vc \in
(\Z^n)^*$. For every sign pattern $\vvarepsilon \in \{\pm1\}^n$ consider the pointed polyhedral cone
\begin{align*}
  \sigma_\vvarepsilon = \{ \, \vx \in \R^n \ : \ A\vx=\0, \ \varepsilon_ix_i \ge \0 \ \text{ for } \
  i=1,\ldots,n \, \}
\end{align*}
together with its Hilbert basis $\hilb_\vvarepsilon$. Then the \emph{Graver basis} $\bigcup
\hilb_\vvarepsilon$ is a test set for our family of IPs in the sense that for every feasible
non-optimal point we find an improving vector in our finite test set
\cite{nonStdIP,SeboeHilbert,HenkWeismantel,MR3024570}.

TDI systems~\cite[\pS22]{Schrijver} are closely related.  A system of linear inequalities $A\vx \le \vb$ with integer coefficients is
called \emph{totally dual integral} (\!\emph{TDI}) if for every $\vc$ such that the dual linear program
(LP)
\begin{align*}
\min \{ \, \langle \vy, \vb \rangle \ : \ \vy A=\vc, \ \vy \ge \0\, \}
\end{align*}
 is bounded, there
exists an integral dual optimal solution.

This property is equivalent to the condition that the active constraints for every face of the
feasible polyhedron
\begin{align*}
P \ =\ \{ \, \vx \ : \ A\vx \le \vb \, \}
\end{align*}
form a Hilbert basis of the normal
cone~\cite[\pS22.3]{Schrijver}.  This is guaranteed, for example, if the normal fan of $P$ can be 
refined to a regular triangulation using no additional rays than the rows of $A$. If $A\vx \le \vb$ is TDI and $\vb$ is integral then $P$ is a lattice polytope.

A particularly nice special class of TDI systems are those with a totally unimodular constraint
matrix $A$. A matrix $A\in\R^{m\times n}$ is \emph{unimodular} if every maximal minor is in $\{0,
\pm1\}$, and it is \emph{totally unimodular} if this holds for all minors. Note that negating
columns or rows and adding unit vectors to the matrix does not change total unimodularity. There are various criteria to check whether a given matrix $A$ is totally unimodular, see e.g.~\cite[\pS19.2]{Schrijver} for a list and references. A system
$A\vx \le \vb$ is TDI for every integral right hand side $\vb$ if and only if $A$ is totally
unimodular. 

Surprisingly, many interesting families of polytopes can be defined by totally unimodular constraint matrices.  In fact, incidence matrices of
bipartite graphs, incidence matrices of directed graphs, and sub-configurations of the root system
$\rootA_{n-1}$ are totally unimodular. For the former see Section~\ref{sec:characteristic-polytopes}, and for the latter class, which includes order polytopes and hypersimplices see Section \ref{sec:type-A-facets}. Furthermore, we consider polytopes spanned by the columns of incidence matrices in Sections~\ref{subsec:anarcpolytope} and \ref{sec:edge-polytopes}.

Suppose we want to minimize a linear functional $\vc \in \R^n$ subject to constraints $A\vx=\vb$,
and $\vx \ge \0$.  We can interpret the columns of the $d \times n$ matrix $A$ as vectors in
$\R^d$. Clearly, the system is (LP-)feasible if and only if $\vb$ belongs to the cone generated by
these columns of $A$.  If we assume that this cone is pointed, then we can consider $c_i$ as a
weight of the $i$-th vector so that $\vc$ induces a regular subdivision of the cone. If this
subdivision happens to be a unimodular triangulation, then the LP-optimum equals the IP-optimum for
any $\vb$. This is because a feasible $\vb$ has a non-negative integral representation in ``its''
cone, and the PL function induced by $\vc$ is convex.

\subsubsection{Commutative algebra}  \label{sec:Commutative algebra}

Many properties of combinatorial objects have direct translations to algebraic objects like
semigroup algebras, monomial ideals, toric varieties, and singularities, via the correspondence
\begin{eqnarray*}
  \text{ lattice point } & & \text{ Laurent monomial } \\
  \va=(a_1,\ldots,a_d) \in \Z^d & \quad\longleftrightarrow\quad &
  \vx^\va:=x_1^{a_1} \cdot\ldots\cdot x_d^{a_d} \in
  \mathbbm{k}[x_1^{\pm1},\ldots,x_d^{\pm1}]
\end{eqnarray*} 

Commutative algebraists are interested in the properties of the graded semigroup ring $R_P =
\mathbbm{k}[\sigma_P \cap \Z^{d+1}]$. For example, $P$ is integrally closed if and only if the
domain $R_P$ is generated in degree one.

Normality is a closely related notion. Consider the subring $\tilde R_P \subseteq R_P$ generated by
the degree one piece. $P$ is called \emph{normal} if $\tilde R_P$ is normal (integrally closed in
its quotient field).  That is, $P$ is normal if $\mathbbm{k}[\sigma_P \cap \Lambda]$ is generated in
degree one, where $\Lambda \subseteq \Z^{d+1}$ is the sublattice generated by $(P \times \{1\}) \cap
\Z^{d+1}$~\cite[Def.~2.59]{bgBook}.

Further, there is a close connection between the Gr\"obner bases of the defining ideal $\ideal_P$ of
$\tilde R_P = \mathbbm{k}[x_1 \ldots x_n]/\ideal_P$ ($n = |P \cap \Z^d|$) and regular triangulations
of $P$.  In particular, if $P$ has a regular unimodular triangulation $\T$, then $R_P=\tilde R_P$
and, moreover, there is a Gr\"obner consisting of binomials and with leading terms given by the
minimal non-faces of $\T$ (this correspondence is demonstrated in
Section~\ref{sec:toric-groebner-bases}).  The degree of each binomial in the Gr\"obner basis is
equal to the size of the corresponding non-face.  So, triangulations provide degree bounds for
Gr\"obner bases.  This makes the search for quadratic triangulations of particular interest, as the
existence of such a triangulation guarantees the existence of a quadratic Gr\"obner basis which in
turn shows that the algebra $\tilde R_P= R_P$ is Koszul (i.e. $\mathbbm{k}$ has a linear free
resolution as an $R_P$-module)~\cite{MR1421076,EisenbudReevesTotaro}.

\subsubsection{Algebraic geometry}
\label{sec:why:algebraic-geometry}

Algebraic geometry associates two objects with~$P$ \cite{Ewald,FultonToric,Oda}.  First there is the
affine Gorenstein toric variety
\begin{equation*}
  U_P = \spec \mathbbm{k}[\sigma_P^\vee \cap \Z^{d+1}] ,
\end{equation*}
where $\sigma_P^\vee$ is the cone dual to $\sigma_P$.  Here a unimodular triangulation of $P$
corresponds to a crepant desingularization of $U_P$, which is projective if and only if the
triangulation is regular. These crepant birational morphisms have been used to reduce canonical
singularities to $\Q$-factorial terminal singularities, and to treat minimal models in high
dimensions. They appear in the high-dimensional McKay correspondence~\cite{IR,ReidMcKay} for
Gorenstein quotient singularities $\C^d/G$, proven by Batyrev~\cite{BatyrevMacKay}. Moreover, a
one-to-one correspondence of McKay-type also holds for triangulation induced resolutions of
$U_P$~\cite{BD}.

One of the earliest, and to this day one of the most striking, results involving unimodular
triangulations is the stable reduction theorem of Kempf, Knudsen, Mumford and
Saint-Donat~\cite{KKMS}. They showed that in characteristic zero, every one-dimensional family can
be resolved so that the exceptional locus is a normal crossing divisor.  They achieved this by
reducing the statement to the case of ``toroidal'' singularities. The combinatorial core of the
argument is the statement that every lattice polytope has some dilation admitting a regular 
unimodular triangulation~\cite{KKMS3}. This is the content of Theorem~\ref{thm:KMW-effective}, 
which we discuss in detail in Section~\ref{sec:kmw}.  

The other algebraic geometric object associated with $P$ is the projectively embedded toric variety
\begin{equation*}
  X_P = \proj \mathbbm{k}[\sigma_P \cap \Z^{d+1}] = \proj R_P .
\end{equation*}
Here the polytope $P$ specifies an ample line bundle $L_P$ on $X_P$. (see
\cite[Section~3.4]{FultonToric}.) If $X_P$ is smooth (i.e., the normal fan of $P$ is unimodular),
then $L_P$ is very ample, and provides an embedding $X_P \hookrightarrow \mathbb{P}^{n-1}$, where $n
= \# (P \cap \Z^d)$. The two following questions about the defining equations of a smooth $X_P
\subset \mathbb{P}^{n-1}$ have been around for quite a while, but the origins are hard to track.

First, if $P$ is a lattice polytope whose corresponding projective toric variety is smooth, is the
defining ideal $\ideal_P$ generated by quadrics?  And secondly, is the embedding $X_P
\hookrightarrow \mathbb{P}^{n-1}$ of a smooth $X_P$ \emph{projectively normal} (i.e. is $R_P$ generated
in degree one)?

Both questions have a positive answer for polytopes with a quadratic triangulation (see Section
~\ref{sec:quadratic}), both questions have a negative answer without the smoothness assumption.
(See Knutson's section in~\cite{cottonwood} for a historical discussion and partial results.)

\subsubsection{A hierarchy of covering properties} \label{sec:hierarchy}

The property of admitting a unimodular triangulation embeds into a large hierarchy of algebraic and
convex geometric properties.
We list some of the more combinatorial ones in decreasing strength.  Compare~\cite[p.~2097f]{MFO04}
and \cite[p.~2313ff]{MFO07}.
\begin{enumerate}
\item $P \cap \Z^d$ is totally unimodular
\item $P$ is compressed (see Section~\ref{sec:paco})
\item $P$ has a regular unimodular triangulation
\item $P$ has a unimodular triangulation
\item $P$ has a unimodular binary cover (a cycle generating $H_d(P,\partial P; \Z_2)$ formed by
  unimodular simplices)
\item $P$ has a unimodular cover
\item $\sigma_P$ is integrally closed and has a free Hilbert cover (every lattice point is a
  $\Z_{\ge 0}$-linear combination of linearly independent lattice points in $P \times
  \{1\}$) \label{item:FHC}
\item $P$ is integrally closed
\end{enumerate}
All the implications ($i$) $\Rightarrow$ ($i+1$) are strict~\cite{BG}.  See~\cite{firlaZieglerHilbertBases99} for other
covering properties.

Property~(\ref{item:FHC}) is known to be equivalent to $\sigma_P$ being integrally closed and satisfying the \emph{integral Caratheodory property}~\cite{BG}: every lattice point is a $\Z_{\ge 0}$-linear combination of $\dim C$ many lattice points in $P \times \{1\}$).  

Most of this hierarchy's properties have direct translations into an algebraic language about $R_P$
and $\ideal_P$. However, there are a few additional algebraic properties which do not quite fit and
have their own shorter hierarchy.
\begin{enumerate}
\item[(1')] $P$ has a quadratic triangulation (see \pS\ref{sec:quadratic})
\item[(2')] $\ideal_P$ has a quadratic Gr\"obner basis
\item[(3')] $R_P$ is a Koszul algebra ($\mathbbm{k}$ has a linear free resolution as an
  $R_P$-module)
\item[(4')] $\ideal_P$ is generated by quadrics
\end{enumerate}
The two hierarchies are linked by the fact that a quadratic triangulation is a regular unimodular
triangulation.

\subsection{What is new?}
\label{sec:results}

In this section we will give a brief summary of the results contained in this paper that have not
appeared elsewhere, or not appeared in this form.

\subsubsection{Section~\ref{sec:methods}}
In Section~\ref{sec:pulling} we clarify the relationship between two notions of pulling that have been
used in the literature, and sometimes confused with one another. We also show a general procedure
for obtaining regular, unimodular and/or flag triangulations of (some) lattice polytopes, by first
dicing by hyperplanes and then refining.  We show that all polytopes with unimodular facet matrices
have regular, unimodular triangulations. However, the $(3\times 3)$-Birkhoff polytope shows that not
all have quadratic ones.

Section~\ref{sec:chimney} introduces pull-back and push-forward
techniques for constructing unimodular triangulations. These methods
were announced in \cite{HPchimneys}. Here we give extended versions of
the constructions with full proofs. We use this to construct regular
unimodular triangulations for smooth reflexive polytopes. An update on
these results is given in Section~\ref{subsubsec:reflexive}.

We prove that (regular and unimodular) triangulations of polytopes can
be extended to give (regular and unimodular) triangulations of their
products, joins, and some other constructions
(Section~\ref{sec:joins-and-products}). 
In particular, we rework and extend 
to the case of non regular triangulations a
result of Ohsugi and Hibi about nested configurations, for which we
introduce the notion of a semidirect product of polytopes~\cite{0907.3253}.
We also extend a result of
Sullivant~\cite{SullivantToricFiber} on toric fiber products to the
case of positive codimension.

Section~\ref{sec:toric-groebner-bases} contains a  direct proof
of the correspondence between regular unimodular triangulations
and square-free initial ideals. Finally, in Section~\ref{sec:Hilbert}
we briefly recall the relation between toric Hilbert schemes,
Gr\"obner bases and unimodular triangulations. We briefly reproduce the
example of a disconnected toric Hilbert scheme by Santos, based in the
root configuration of type $\rootF$.

\subsubsection{Section~\ref{sec:Examples}}
Here we  give several new examples of families of polytopes that do (or
do not) have regular unimodular triangulations, and some with additional
properties.

We start by considering polytopes whose facet normals belong to a root
system (Section~\ref{sec:cut-by-roots}). Those of type $\rootA$ can be
triangulated via hyperplane dicing, which implies that they have
quadratic triangulations. For type $\rootB$ we show that the
hyperplane subdivision by short roots can be extended to produce a
regular unimodular triangulation (but not necessarily a quadratic
one). In contrast, we show examples of polytopes cut out by
hyperplanes of type $\rootE_8$ and $\rootF$ and which do not have
unimodular triangulations. In fact, these examples are not even integrally
closed.

Section~\ref{sec:spanned-by-roots} compiles several results about
polytopes whose vertices belong to a root system, giving a unified
approach to and sometimes shorter proofs of known results. 

Smooth $(3\times 3)$-transportation polytopes have been studied in
\cite{TransportPolys} using hyperplane subdivision. Here we extend
this with a theorem of Lenz proving that all simple $3\times 4$
transportation polytopes have  quadratic triangulations
(Proposition~\ref{prop:lenz}).  In Section~\ref{sec:empty} we show
that smooth empty lattice polytopes are products of unimodular
simplices.

\subsubsection{Section~\ref{sec:kmw}}
This section revolves around dilations of lattice polytopes, including the proof 
of our effective version of the Knudsen-Mumford-Waterman Theorem (Theorem~\ref{thm:KMW-effective}).
As the original one, our proof is based on
induction on the maximum volume of a simplex in an initial
triangulation of $P$ 

Before that, in Section~\ref{sec:kmwhypersimplex}  we
study refinements of dilations of unimodular triangulations. In
particular, we show a new procedure to quadratically refine
a dilation of a quadratic triangulation (Theorem~\ref{thm:unimodular-dilations}).

\subsection{What is not here}
\label{sec:whatnot}

Although we try to give a thorough overview of methods for obtaining unimodular triangulations, we
cannot cover everything in this paper. In particular, we restrict our attention to (geometric)
triangulations of lattice polytopes. We do not cover subdivisions of cones or combinatorial
abstractions via oriented matroids. We also do not consider methods for modifying a triangulation via
flips.

There are a number of computational tools available to test for and explore triangulations,
secondary fans, etc. Among the prominent tools are the software package \texttt{normaliz2} of Bruns,
Ichim, and S\"oger \cite{normaliz2,1211.5178} for the computation of normalizations of cones;
\texttt{4ti2} of Hemmecke, K\"oppe, and DeLoera~\cite{4ti2} for the computation of generating sets
of integer points in polyhedra, Hilbert and Gr\"obner bases; \texttt{TOPCOM} by
Rambau~\cite{Rambau:TOPCOM-ICMS:2002} for various types of triangulations; \texttt{LattE Integrale}
by De Loera et al.\ for computing Ehrhart series; The software
\texttt{polymake}~\cite{polymake-sw,GJ00} offers fast enumeration of lattice points and extensive
computations with polyhedra, fans, and triangulations. For the toric methods explained in
Section~\ref{sec:toric-groebner-bases} one should also consider the computer algebra packages
\texttt{Singular}~\cite{GPS}, \texttt{Macaulay2}~\cite{M2}, and \cocoa~\cite{CS}. The software
\texttt{polymake} offers many computations for toric ideals and varieties via
extensions~\cite{polymake-algebra,polymake-toric}.

Finally, this survey also does not comment on connections between the emerging field of tropical
geometry and lattice polytopes, their subdivisions, and the secondary fan. For this, see
for example~\cite{JoswigTropicalCombinatorics}.

\subsection{What is left?}
\label{sec:questions}
To close this introduction we here compile several  open questions on (regular, unimodular, or flag) triangulations of lattice polytopes.  
One general open question is: how special is it for a lattice polytope to have a unimodular
triangulation? One way to make this precise is as follows.

\begin{question}
\label{question:most_do_not}
For a given fixed dimension $d$ and
volume $V$ consider the set of lattice $d$-polytopes of normalized volume at most $V$, modulo
unimodular equivalence, which is a finite
set~\cite{ArnoldLatticePolytopeConjecture,BaranyYuan,BaranyVershik}.

Let $u(d,V)\in [0,1]$ denote the fraction of them that admit unimodular triangulations. We conjecture that
\[
\limsup_{V\to \infty} \sqrt[V]{u(d,V)} < 1,
\]
for every fixed $d\ge 3$.
\end{question}

Other questions we think are important are the following.

\subsubsection{Smoothness and normality, vs.~unimodular triangulations}
The following two questions have been mentioned already in Section~\ref{sec:why:algebraic-geometry}.
\begin{itemize}
\item  Does every smooth polytope have a unimodular triangulation, or at least a unimodular cover? A positive answer to this would imply a positive answer to the following very prominent open problem by Oda: Is every smooth projective toric variety 
\emph{projectively normal}? Both questions are open even in dimension three.
Various partial results have been proved in this direction
(see~\cite{TransportPolys,MFO07,0709.3570,0912.1068}, and computational approaches in
\cite{HaaseLP10,1206.4827}).
\item Is the ideal of every smooth projective toric variety generated by quadrics? Here, a positive
  answer would follow from the existence of a quadratic triangulation for every smooth polytope.
\end{itemize}
More generally, for every property of the hierarchies in Section~\ref{sec:hierarchy} above one can ask whether or
not all smooth polytopes satisfy that property. Every one of these questions is open in dimensions
three and above. In Section~\ref{subsubsec:reflexive} we investigate smooth reflexive polytopes
which have been classified up to dimension nine. Of the 80\,892 smooth reflexive polytopes in dimension at
most seven, there are 18 for which we could not decide the existence of a quadratic triangulation -- two
in dimension six and 16 in dimension seven.

\subsubsection{Polytopes cut out by roots}
Athough we here considerably extend the results on triangulations of polytopes cut out by
root systems (see Section~\ref{sec:cut-by-roots}), some cases are still open. In particular, it is not
known whether polytopes whose facet normals are contained in $\rootC_n$ or $\rootD_n$ have
unimodular triangulations (we expect that they do), and the same for those with facet normals in $\rootE_7$ or $\rootE_6$ do (we expect that they do not; in fact we do not expect them to be  integrally closed, in general). Further, although we show that $\rootB_n$-polytopes have a regular and unimodular
triangulation, we do not know whether they have quadratic ones.

Another natural question is as follows: the way we define polytopes cut out by roots is as those
whose normal vectors are in the root system. But one can also restrict further and look at polytopes whose normal fan
is refined by the cluster complex for that root system.
If the normal fan agrees with the cluster complex, this is a particular question of the
smoothness question above, since generalized associahedra have simplicial and unimodular normal
fans. But it is also a question that has more chances to be answered in the positive, since there is the additional machinery of root systems and cluster algebras available.
In type $\rootA$, this class contains the class of matroid polytopes and thus could shed some
light on an old conjecture of Neil White~\cite{WhiteConjecture} that the toric ideal of a matroid
polytope is generated in degree two. A partial positive answer is in~\cite{MR3197649}, but a quadratic triangulation of the matroid polytope would even yield a quadratic
Gr\"obner basis.

\subsubsection{Dilated polytopes}
There are several open questions concerning the dilation factors $c$ that make the dilation
$cP$ have a unimodular triangulation. Among them we can mention the following:
  \begin{enumerate}
  \item  Is it true that if $c_1P$ and $c_2P$ have unimodular triangulations, then $(c_1+c_2)P$ has one? That is,
  is the set of valid dilation factors a subsemigroup of $\N$? This is known to hold for empty
  simplices in dimension three~\cite{SantosZiegler}, but not in general. 
  \item Is it true that for every $P$ there is a $c_0$ such that $cP$ has a unimodular triangulation for all $c\ge c_0$? This is Problem 5 in \cite{BGTproblems}. Remember that it is not always true that  if $cP$ has a  (regular, flag) unimodular triangulation then $(c+1)P$ has one too
 (see Example~\ref{kmw:eg:c+1})
\item The bound in Theorem~\ref{thm:KMW-effective} is doubly exponential both in the dimension and
  the volume of the starting polytope. Is there a singly exponential one? Is there a bound depending
  on dimension but not on volume (a bound ``constant in fixed dimension'')? 
  Is the latter true at least in dimension four? 
  
  We do not believe a global bound (in fixed dimension) to 
exists for $d \ge 5$. However, given a triangulation of a $d$-polytope $P$ in which no
simplex has volume greater than $V$, we believe that a dilation factor $c$ of
about $ d^V$ should be sufficient for $cP$ to have a unimodular triangulation.

\end{enumerate}

All these questions have a positive answer if we only ask for unimodular
 \emph{covers} rather than triangulations. In this case for every lattice polytope $P$ there is 
a threshold $c_P$ such that $cP$ has a unimodular cover if and only if $c\ge c_P$. Different polytopes may have different thresholds, but for every dimension there is a common $c_d$ upper bounding the thresholds of all $d$-polytopes and depending polynomially on $d$~\cite[Theorem 3.23]{bgBook}. 
  
 We remark that for the IDP property there are polytopes where the valid dilations do not form a semigroup: it follows from~\cite{LasonMichalek}
  that for each $k$ there is a $(2k-1)$-dimensional lattice polytope $P$ such that $cP$ is IDP if, and only if, $c$ is not a proper divisor of  $k$. For example, there is a $49$-dimensional polytope such that $2P$ and $3P$ are IDP but $5P$ is not. See details in Example~\ref{kmw:eg:c1+c2}.
  
\subsubsection{Other questions}

Lattice parallelepipeds and, more generally, lattice zonotopes are known to be IDP. Is it true
  that they all have unimodular triangulations?

Lecture hall simplices (see Section~\ref{sec:lecture-hall}) arise very naturally in combinatorics
and have relations to partitions of integers. In their original form, $\lecturehallsimplex{d+1}$ has
recently been shown to have a quadratic triangulation. But for their generalizations
$\slecturehall{d+1}$ (see  \eqref{eq:sLectureHall}) it is not known what sequences $\s =
(s_i)_{i=1}^d$ yield polytopes admitting unimodular triangulations.
More generally, which $\s$-order polytopes $O(\preccurlyeq,\s)$ have
unimodular triangulations?

Do the homogenized hive matrices of De Loera and McAllister~\cite{deLoeraMcAllisterDilation} always
have unimodular triangulations?
(see Section~\ref{sec:enum_combi} for details).

\section{Methods} %
\label{sec:methods}

\subsection{Pulling Triangulations} 
\label{sec:pulling}
\label{sec:paco}
Pulling refinements are a useful tool for constructing regular
triangulations. Two distinct versions of pulling refinements appear in
the literature. We will call them weak pulling and strong pulling.

\subsubsection{Weak and Strong Pulling}
\label{sec:weakstrong}

We first discuss strong pulling. 
If $\S$ is a  subdivision of $P$ and $\vm \in  P \cap \Z^d$ is
a lattice point in $P$, the \emph{strong pulling 
  refinement} $\pull_{\vm}\S$ is obtained by replacing  every face $F
\in \S$  containing $\vm$ by the pyramids $\conv(\vm,F')$, for each
face $F'$ of $F$ which does not contain $\vm$.

Here are some facts about the structure of pulling subdivisions.
\begin{lemma} \label{lemma:strongPulling} Let $P$ be a lattice
  polytope with a subdivision $\S$.
  \begin{enumerate}
  \item Strong pulling preserves regularity.
  \item Strongly pulling all lattice points of $P$ in any order
    results in a full triangulation.
  \item If only vertices of $P$ are pulled, then every maximal cell of
    the refinement will be
    the join of the first pulled vertex $\vv_1$ with a maximal cell in
    the pulling subdivisions of the facets not containing $\vv_1$.
  \end{enumerate}
\end{lemma}

In particular, this Lemma guarantees that every (regular) lattice subdivision of a
lattice polytope has a (regular) refinement which is a full
triangulation. 

\begin{proof}
  {\em (1)\/}: Let $\S$ be the regular subdivision of $P$ given by a weight vector 
 $\vomega \in \R^\configuration$ ($\configuration = P \cap
  \Z^d$). By definition of regularity,
  the lower faces of the lifted polyhedron $\tilde P = \conv ( \va \times
  [\omega_\va, \infty) \ : \ \va \in \configuration )$ in $\R^{d+1}$ project to $\S$. 
  Let $\m \in \configuration$ and set $\omega'_\m = \min \{ h \ : \
  (\m,h) \in \tilde P \} - \epsilon$ and $\omega'_\va = \omega_\va$
  for all $\va \in \configuration \setminus \{\vm\}$. Then, for small
  enough $\epsilon > 0$, the pulling refinement $\pull_\vm(\S)$ is
  induced by the weights $\vomega'$.
  
  {\em (2)\/}: Every face of $\pull_\vm(\S)$ containing $\vm$ is a
  pyramid with apex $\vm$. If $Q \in \S$ has $\vn$ as an apex, then
  every face of $\pull_\vm(S)$ inside $Q$ and containing $\vn$ still has
  $\vn$ as an apex. After strongly pulling all lattice points, each
  lattice point is a vertex of the subdivision, 
  and every vertex of every cellos an apex. So, each cell is a simplex. 
  
  {\em (3)\/}: If we apply the previous argument to the trivial
  subdivision of $P$, $\vv_1$ must be an apex of every cell. 
\end{proof}

The other notion of pulling, weak pulling, arises
in the context of subdivisions of  point configurations~\cite{deLoeraRambauSantos}
(or, equivalently, subdivisions of marked polytopes~\cite{GKZ}).
Basically, in a subdivision of a point configuration 
each face $F \in \S$ is determined by the convex hull of some
(possibly full) subset of its lattice points, and weak pulling treats
lattice points in this set differently than those not in this set for
each $F$. 

For example, consider regular subdivisions $\S_1$ and $\S_2$ of
the one-dimensional configuration $\configuration=\{1,2,3,4\} \subset
\Z^1$ obtained from the weight vectors $\vomega_1=(0,0,0,1)$ and
$\vomega_2=(0,1,0,1)$ respectively.  They both consist of the segments $[1,3]$ and
$[3,4]$ but in $\S_1$ the segment $[1,3]$ appears as the convex hull
of $1$, $2$ and $3$ while in $\S_2$ it appears as the convex hull of
$1$ and $3$ alone. When looking at the set of all subdivisions of
$\configuration$ as a whole, there are good reasons to consider $\S_1$
and $\S_2$ as different, $\S_1$ not being a triangulation and $\S_2$
being a triangulation that refines $\S_1$.

In this context, when performing a pulling at $\vm$ it is natural to
remove  the
lattice points in each pyramid $\conv(\vm,F')$ from the list of points
that can be pulled later, but not those in its base
$F'$. This is what we call \emph{weak pulling\/}.
In the setting of toric algebra as seen in
Section ~\ref{sec:toric-groebner-bases}, the weak pulling triangulation
corresponds to the reverse lexicographic term order.
Figure~\ref{fig:pulling} illustrates the difference of the two
versions. 
\begin{figure}[tb]
  \centering
\begin{tikzpicture}[y=-1cm, scale=.3]

\draw[black] (6.35,21.59) -- (16.51,16.51);
\draw[black] (11.43,21.59) -- (16.51,16.51);
\draw[black] (6.35,21.59) -- (11.43,16.51);

\draw[black] (16.51,21.59) -- (11.43,21.59);
\draw[black] (6.35,21.59) -- (11.43,21.59);

\filldraw[black] (6.35,21.59) circle (0.5cm);
\filldraw[black] (11.43,21.59) circle (0.5cm);
\filldraw[black] (6.35,16.51) circle (0.5cm);
\filldraw[black] (11.43,16.51) circle (0.5cm);
\filldraw[black] (16.51,16.51) circle (0.5cm);
\filldraw[black] (16.51,21.59) circle (0.5cm);
\draw[black] (6.35,21.59) -- (6.35,16.51);
\draw[black] (16.51,21.59) -- (16.51,16.51) -- (6.35,16.51);
\path (11.5,15.24) node[text=black,anchor=base] {$\vm_3$};
\path (6.35,24) node[text=black,anchor=base] {$\vm_1$};
\path (11.5,24) node[text=black,anchor=base] {$\vm_2$};

\end{tikzpicture}%
   \qquad \qquad 
\begin{tikzpicture}[y=-1cm, scale=.3]

\draw[black] (6.35,21.59) -- (16.51,16.51);
\draw[black] (6.35,21.59) -- (11.43,16.51);

\draw[black] (16.51,21.59) -- (11.43,21.59);
\draw[black] (6.35,21.59) -- (11.43,21.59);

\filldraw[black] (6.35,21.59) circle (0.5cm);
\path[draw=black,fill=white] (11.43,21.59) circle (0.5cm);
\filldraw[black] (11.43,16.51) circle (0.5cm);
\filldraw[black] (16.51,16.51) circle (0.5cm);
\filldraw[black] (16.51,21.59) circle (0.5cm);
\draw[black] (6.35,21.59) -- (6.35,16.51);
\draw[black] (16.51,21.59) -- (16.51,16.51) -- (6.35,16.51);
\path (11.5,15.24) node[text=black,anchor=base] {$\vm_3$};
\path (6.35,24) node[text=black,anchor=base] {$\vm_1$};
\path (11.5,24) node[text=black,anchor=base] {$\vm_2$};

\end{tikzpicture}%
   \caption{Strongly and weakly pulling $\vm_1$, then $\vm_2$, then
    $\vm_3$.}
  \label{fig:pulling}
\end{figure}

The following is the analogue of Lemma~\ref{lemma:strongPulling} for
weak pullings. For a proof we refer to
\cite[Section 4.3.4]{deLoeraRambauSantos}:
\begin{lemma} \label{lemma:weakPulling} Let $P$ be a lattice polytope
  with a subdivision $\S$.
  \begin{compactenum}
  \item Weak pulling preserves regularity.
  \item 
    Weakly pulling all lattice points of $P$ in any order yields  a triangulation.
  \end{compactenum}
\end{lemma}

In this article, we will only use pulling refinements in cases where
all lattice points in the polytope $P$ are vertices of the subdivision
$\S$. With this assumption, strong and weak pulling subdivisions
agree, so there is no ambiguity when referring to \emph{pulling
  subdivisions}.

\subsubsection{Compressed Polytopes}
Stanley calls a polytope \emph{compressed} if all its weak pulling
triangulations are unimodular~\cite{StanleyDecompositions80}. 
This clearly implies that the only lattice points in $P$ are its vertices. Because of Theorem~\ref{thm:paco} below, compressed polytopes are sometimes called \emph{2-level} polytopes.

Surprisingly many well-known polytopes fall into this category.
Examples of compressed polytopes include
the Birkhoff polytope (Section ~\ref{sec:flow-polytopes}), order
polytopes and hypersimplices (Section ~\ref{sec:type-A-facets}),
stable set polytopes of perfect graphs
(Section~\ref{sec:characteristic-polytopes}), and the integer hulls of weighted Gelfand-Tsetlin polytopes $P_{\lambda/\mu,\1}$ for weight $\1$~\cite{MR3459049}.
Atha\-na\-sia\-dis was even able to use the fact that the Birkhoff polytope is
compressed to prove unimodality of its
$h^*$-vector~\cite{AthanasiadisCompressed} (see
also~\cite{BrunsRoemer}). 

There are several characterizations of compressed polytopes. Here we present one based on
width with respect to facets. If $P$ is a lattice
polytope and  $\langle \vy_i , \vx \rangle \ge  c_i$ for
$i=1,\ldots,n$ are  the  facet defining inequalities  with primitive
integral $\vy_i$, the \emph{width} of $P$ with respect to the $i$-th facet
(or with respect to the direction of $\vy_i$) is the difference
\begin{equation*}
    \max \langle \vy_i, P \rangle\ -\ \min \langle \vy_i, P \rangle  .
\end{equation*}
In particular, $P$ has width one with respect to a facet if it
lies between the hyperplane spanned by this facet and the next
parallel lattice hyperplane.

The main implication of the following theorem is due to the fourth
author. The proof we present here is the original one (MSRI 1997,
unpublished).  It was subsequently also proven by Ohsugi and
Hibi~\cite{OHConvexPolytopesRevLex01} and by
Sullivant~\cite{SullivantCompressed}.
\begin{theorem} \label{thm:paco}
  Let $P$ be a lattice polytope.
  Then the following are equivalent:
  \begin{compactenum}
    \item $P$ is compressed.
    \item $P$ has width one with respect to all its facets.
    \item\label{thm:item:third} $P$ is lattice equivalent 
      to the intersection of a unit cube with an affine space.
  \end{compactenum}
\end{theorem}

\begin{proof}
  $(2)\,\Longrightarrow\,(1)$:    By decreasing   induction  on    the
  dimension one sees that  every face of  $P$ has width one with respect to all facets. The
  restriction of  a pulling  triangulation to  any face is   a pulling
  triangulation  itself  and thus  unimodular  (by another induction).
  Hence, every maximal simplex in the triangulation of $P$ is the join
  of a unimodular  simplex in some facet  with the first lattice point
  that was pulled.

  The other implications are easy.
\end{proof}

\subsubsection{Hyperplane Arrangements} \label{sec:hyperplanes}
In this section, we apply the above characterization of compressed
polytopes to triangulate ``bigger'' polytopes using hyperplane
arrangements.

Let $\configuration:=\{\vn_1,  \ldots, \vn_r \} \subset \Z^d$ be a
collection of 
vectors that span $\R^d$ and form a unimodular matrix, (i.e. such that
all ($d \times d$)--minors of $\configuration$  are either zer0, one, or negative one).
Such a collection induces an infinite arrangement of hyperplanes
\[
 \{ \vx \in \R^d \ : \ \langle \vn_i, \vx \rangle = k \} \quad
\text{for} \quad i=1, \ldots, r \text{ and } k \in \Z\,,
\]
which subdivide $\R^d$ regularly into lattice polytopes. These
subdivisions are referred to in the literature as lattice dicings~\cite{ErdahlRyshkovDicing}.

A lattice polytope $P$ whose collection of primitive facet
normals forms a unimodular matrix is called \emph{facet unimodular\/}. Every
face of a facet unimodular polytope is also facet unimodular in its
own lattice. The lattice dicing hyperplane arrangement slices $P$ into dicing
cells. This is called the \emph{canonical subdivision\/} of
a facet unimodular polytope. The canonical subdivision subdivides
faces canonically.
\begin{theorem} \label{thm:hyperplanes}
  Suppose 
  that $P \subset \R^d$ is a facet unimodular lattice polytope.
  Then:
  \begin{enumerate}
  \item The canonical subdivision of $P$ is regular, and all the cells
    are compressed.
  \item $P$ has a regular unimodular triangulation.
  \end{enumerate}
\end{theorem}

\begin{proof}
  The dicing cells have width one with respect to all their facets by
  construction.  This proves part (1) except for regularity.  Regularity
  of the canonical subdivision follows from considering  weights given
  by the
  restriction of the following quadratic function to the lattice
  points in $P$:
  \[
  f(\vx)=\sum_{i=1}^r \langle \vn_i, \vx \rangle^2.
  \]
  Since every lattice point in each cell of the canonical subdivision is
   contained in two consecutive hyperplanes in each
  direction $\vn_i$, the corresponding summand of $f$ is equal, on
  those lattice points, to an affine map. Therefore, $f$ is affine
  on each cell. But any two adjacent cells contain lattice points in
  three hyperplanes for some direction, so $f$ is not affine
  on the union.
  
  For part (2) recall that any pulling refinement of a regular
  subdivision is regular.  Since the cells in the canonical
  subdivision are compressed, a triangulation obtained by pulling will be
  regular and unimodular
\end{proof}

As a direct application of Theorem~\ref{thm:hyperplanes}, flow
polytopes (Section ~\ref{sec:flow-polytopes}) as well as polytopes
with facets in the root system of type $\rootA$ have regular unimodular
triangulations (Section ~\ref{sec:type-A-facets}). 
This dicing method also shows that every dilation $cP$ of a polytope $P$ with
a (regular) unimodular triangulation itself admits a (regular) unimodular triangulation
(Theorem~\ref{thm:unimodular-dilations}).

This approach to finding (regular) unimodular triangulations can be
applied whenever there is a lattice dicing that cuts $P$ into lattice
polytopes. However in this more general case, the cells do not
automatically have width one with respect to the facets given by
$P$. This must be checked separately. Polytopes with facet normals in
the root system of type $\rootB$ are an example where this approach
was successful~(Section~\ref{sec:type-B-facets}).

\subsubsection{Circuits}

A \emph{circuit} of a point configuration $\configuration$ is a minimal affine dependent subset $C$. 
A circuit $C$ comes with a unique (up to a constant) affine dependence 
\[
\sum_{\va \in C} \lambda_\va \va =0, \quad \sum_{\va \in C} \lambda_\va=0.
\]

It is well-known that a configuration that is itself a circuit, or that has a unique circuit, has exactly two triangulations. 
Having a unique circuit is equivalent to the configuration having exactly two points more than its dimension.
The following is essentially Lemma~2.4.2 in~\cite{deLoeraRambauSantos}.

\begin{lemma}
\label{lemma:circuit-triangs}
Let $\configuration$ be a configuration of $d+2$ points spanning a $d$-dimensional affine space. Let $\lambda\in \R^\configuration$ be its unique (up to a constant) affine dependence. Call
\[
\configuration^+:=\{\va\in \configuration: \lambda_\va >0\},\quad
\configuration^0:=\{\va\in \configuration: \lambda_\va =0\},\quad
\configuration^-:=\{\va\in \configuration: \lambda_\va <0\}.
\]

Then, $\configuration$ has exactly two triangulations, namely:
\[
\T^+=\{F\subset \configuration : \configuration^+\not\subset F\}, \qquad 
\T^-=\{F\subset \configuration : \configuration^-\not\subset F\}, 
\]
Both triangulations are regular.
\end{lemma}

Put differently, $\T^+$ (resp. $\T^-$) has $\configuration^+$ (resp. $\configuration^-$) as its only minimal non-face.
Observe that the points of $\configuration^0$ lie in every maximal simplex of both $\T^+$ and $\T^-$. 
This reflects the fact that for a $\va\in\configuration^0$, $ \va$ is not in the affine span of $\configuration\setminus \va$.

It is easy to specify when these triangulations are flag and/or unimodular:

\begin{lemma}
\label{lemma:circuit}
Let $\configuration$ be a configuration of $d+2$ points spanning a $d$-dimensional affine space with its two triangulations $\T^+$ and $\T^-$.
\begin{enumerate}
\item $\T^+$ (resp. $\T^-$) is flag if, and only if, $|\configuration^+| \le 2$ (resp. $|\configuration^-| \le 2$).
\item Suppose $\configuration$ is a lattice point set and that $\lambda$ is normalized to have integer entries with no common factor. Let $\Lambda_\configuration$ be the affine lattice generated by $\configuration$.
Then, $\T^+$ (resp. $\T^-$) is unimodular \emph{in $\Lambda_\configuration$} if, and only if, all positive (resp. negative) coefficients in $\lambda$ are equal to $\pm 1$.
\end{enumerate}
\end{lemma}

\begin{proof}
For part (1), observe that $\configuration^+$ is the unique minimal non-face in $\T^+$. For part (2), observe that the coefficient $\lambda_\va$ of a point $\va\in \configuration$ equals $\pm 1$ if, and only if, $\va$ is an integer affine combination of the rest of the points. Since the maximal simplices of $\T^+$ are precisely $\{\configuration \setminus \va : \va \in \configuration^+\}$, the result follows.
\end{proof}

In particular, for $\T^+$ to be quadratic we need $\configuration^+$ to have at most two elements and those elements have a coefficient of one in the dependence. Since $\sum_{\va\in\configuration} \lambda_a=0$, $\configuration^-$ has also at most two elements and there are only the following two possibilities:
\begin{itemize}
\item $\configuration^+=\{\va, \vb\}$ and $\configuration^-=\{\vc, \vd\}$ with $\va + \vb=\vc+\vd$, or
\item $\configuration^+=\{\va, \vb\}$ and $\configuration^-=\{\vc\}$ with $\va + \vb=2\vc$.
\end{itemize}
The circuit consists of the four vertices of a parallelogram in the first case and of three collinear and equally spaced points in the second case.

Finally, let us mention that $\T^+$ (resp. $\T^-$) is the weak pulling of $\configuration$ from any $\va\in \configuration^-$ 
(resp. from any $\va\in \configuration^+$). It agrees with what strong pulling would give unless $\configuration^-$ (resp. $\configuration^+)$ has a single element $\va$. (In this case, $\va$ is not a vertex of $\configuration$). 

\subsection{Push-forward subdivisions and pull-back subdivisions}
\label{sec:chimney}

In some cases the search for triangulations can be simplified via
projection.  This is done via push-forward and pull-back subdivisions.

\subsubsection{Chimney polytopes and  pull-back subdivisions}
In this section we describe a method for recursively constructing
unimodular triangulations of certain lattice polytopes. The process yields
 (regular) unimodular triangulations of \emph{generalized prisms} over
polytopes with a (regular) unimodular triangulation. In particular, this section extends results
announced in \cite{HPchimneys}.

\newcommand{\chim}{\operatorname{Chim}}

We must first define chimney polytopes.  Given a lattice
polytope $Q\subset\R^d$, consider two integral linear functionals
$\vl$ and $\vu$, such that $\vl\le \vu$ along $Q$.  We define the the
\emph{chimney polytope} associated to $Q$, $\vl$ and $\vu$ as
  \begin{align*}
    \chim(Q,\vl,\vu)\ :=\ \{(\vx,y)\in  \R^d\times  \R\ \mid\ \vx\in Q,\ \vl(x)\le y\le
    \vu(\vx)\}\,.
  \end{align*}
  We call $Q$ the \emph{base} of the chimney.
\begin{figure}[tb]
  \centering
\begin{tikzpicture}[y=-1cm, scale=.2]

\draw[red] (3.81,3.81) -- (24.13,24.13);
\draw[blue] (24.13,21.59) -- (2.54,21.59);

\draw[very thick,black,fill=black!15] (6.35,6.35) -- (6.35,21.59) -- (16.51,21.59) -- (16.51,16.51) -- cycle;
\draw[arrows=-to,black] (11.43,23.495) -- (11.43,27.305);
\draw[very thick,black] (6.35,29.21) -- (16.51,29.21);

\filldraw[black] (6.35,21.59) circle (0.6cm);
\filldraw[black] (11.43,21.59) circle (0.6cm);
\filldraw[black] (6.35,16.51) circle (0.6cm);
\filldraw[black] (11.43,16.51) circle (0.6cm);
\filldraw[black] (11.43,11.43) circle (0.6cm);
\filldraw[black] (6.35,11.43) circle (0.6cm);
\filldraw[black] (16.51,16.51) circle (0.6cm);
\filldraw[black] (16.51,21.59) circle (0.6cm);
\filldraw[black] (6.35,6.35) circle (0.6cm);
\filldraw[black] (11.43,29.21) circle (0.6cm);
\filldraw[black] (16.51,29.21) circle (0.6cm);
\filldraw[black] (6.35,29.21) circle (0.6cm);
\path (5.08,20.32) node[text=blue,anchor=base east] {\large{}$l$};
\path (5.08,7.62) node[text=red,anchor=base east] {\large{}$u$};
\path (21.59,15.24) node[text=black,anchor=base west] {\large{}$P$};
\path (12.7,26.035) node[text=black,anchor=base west] {\large{}$\pi$};
\path (21.59,28.575) node[text=black,anchor=base west] {\large{}$Q$};

\end{tikzpicture}%
  \caption{An  example of  the chimney construction. Here $l\equiv 0$
    and $u(x)=3-x$ on $Q=[0,2]$.}
  \label{fig:nakajima2d}
\end{figure}
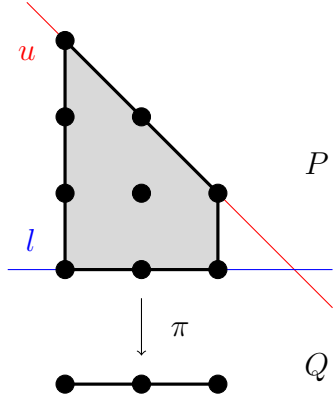
$\chim(Q,\vl,\vu)$ is itself a lattice polytope (see Figure~\ref{fig:nakajima2d}).

We will show that a chimney polytope has a unimodular triangulation
if its  base has one.  For this we introduce the  general concept of a
pull-back subdivision.

Given a lattice polytope $P$ in $\R^d$ and a projection $\pi \colon
\R^d\rightarrow  \R^{d'}$, let $Q:=\pi(P)$ and let $\S'$ be a
subdivision of $Q$.
The \emph{pull-back subdivision} $\pi^*\S'$ of $P$ is obtained from
$\S'$ by intersecting $P$ with the infinite prisms
$\pi^{-1}(F)$ for each cell $F \in \S'$. 

Observe that the cells in the pull-back subdivision
may in principle not be lattice polytopes. A simple example is 
the projection of the triangle $\{(x,y) \colon 0\le 2y \le x \le 2\}$ 
to the segment $[0,2]$. 
The pull-back of $\{[0,1],[1,2]\}$ produces cells with a non-integer vertex, $(1,1/2)$. 
(Observe that this triangle is not a chimney polytope over the segment,
because the functional $\vu(x)=x/2$ is not integral).

We want to show that in the case when
the pull-back is integral and the projection drops only one dimension, it can be refined to a triangulation with nice properties.
To show that the construction preserves degree 
of the triangulation we need the following property:

\begin{lemma} \label{lemma:dualTree}
  Let $\T$ be a simplicial ball whose dual graph is a
  tree. Then $\T$ is flag.
\end{lemma}

\begin{proof}
We use induction on the number of maximal simplices in $\T$.

Let $F$ be a maximal simplex that corresponds to a leaf in the tree. $F$ has a common facet with some other simplex in the triangulation, and a single vertex $\va$ not in that facet. Then, $\T'=\T\setminus F$ is also a triangulation whose dual graph is a tree, so we assume it is flag. 

Let now $N$ be a non-face of $\T$. If $\va\not\in N$ then $N$ is also a non-face in $T'$ hence it has size two. If $\va\in N$ then pick any vertex $\vb\in N \setminus F$ (which exists since $N\not\subset F$) and observe that $N'=\{\va,\vb\}$ is a non-face.
\end{proof}

\begin{theorem} \label{thm:chimney}
  Let $P \subset \R^d$ be a lattice polytope and let $\pi \colon
  \R^d\rightarrow \R^{d-1}$ be a projection such that
  $\pi(\Z^d)=\Z^{d-1}$.
  Let $\T$ be a unimodular triangulation of $Q:=\pi(P)$ and suppose
  $\pi^*\T$ is integral, then 
  \begin{itemize}
    \item Any full refinement $\T'$ of $\pi^*\T$ is a
  unimodular triangulation of $P$. 
    \item $\T'$ does not have minimal non-faces of cardinality larger than those of $\T$.
    \item If $\T$ is regular, $\T'$ is
  regular as well. 
  \end{itemize}
\end{theorem}

\begin{proof}
  To show that regularity can be preserved, recall that if $\T$ is
  regular, a full pulling refinement of $\pi^*\T$ will be regular as
  well.

  For the unimodularity, it is enough to consider the chimneys
  $\pi^{-1}(G) \cap P$ for each simplex $G \in \T$ individually. They
  are equivalent to some $\chim(\Delta^{d-1},\0,\vu)$. Any $d$-simplex
  $F$ in a full triangulation of $\chim(\Delta^{d-1},\0,\vu)$ has two
  vertices above one vertex of $\Delta^{d-1}$ and one vertex above
  every other vertex of $\Delta^{d-1}$. Since $F$ is part of a full
  triangulation $\T'$, the heights of the two vertices with the same
  projection differ by one. Hence, $F$ is unimodular.

  Let $N \subseteq P \cap \Z^d$ be a non-face of $\T'$. Then either $\pi(N)$
  spans a face of $\T$ or not. If $\pi(N)$ is a non-face then $N$ contains
  a non-face $N'$ with $\pi(N')=\pi(N)$ on which $\pi$ is injective.

  If $\pi(N)$ is a face we can, again, restrict our attention to a
  single chimney of the form $\chim(\Delta^{d-1},\0,\vu)$. 
  The dual graph in a triangulation of such a chimney is a path
  (cf.\ Figure~\ref{fig:chimneyOrdering}) and
  Lemma~\ref{lemma:dualTree} shows that $N$ contains a non-face of
  cardinality two.
   \begin{figure}[tb]
     \centering
\begin{tikzpicture}[scale=.25, y=-1cm]

\draw[black] (6.35,16.51) -- (11.43,21.59);
\draw[black] (6.35,11.43) -- (11.43,21.59);
\draw[black] (11.43,21.59) -- (6.35,21.59);

\filldraw[black] (6.35,21.59) circle (0.5cm);
\filldraw[black] (11.43,21.59) circle (0.5cm);
\filldraw[black] (6.35,16.51) circle (0.5cm);
\filldraw[black] (11.43,16.51) circle (0.5cm);
\filldraw[black] (11.43,11.43) circle (0.5cm);
\filldraw[black] (11.43,6.35) circle (0.5cm);
\filldraw[black] (6.35,11.43) circle (0.5cm);
\draw[black] (11.43,16.51) -- (6.35,11.43) -- (11.43,11.43);
\draw[black] (6.35,21.59) -- (6.35,11.43);
\draw[black] (11.43,21.59) -- (11.43,6.35) -- (6.35,11.43);
\path (9.2,17.62125) node[text=black,anchor=base west] {\large{}$3$};
\path (6.8,20.47875) node[text=black,anchor=base west] {\large{}$1$};
\path (8.5,13.6525) node[text=black,anchor=base west] {\large{}$4$};
\path (8.5,10.31875) node[text=black,anchor=base west] {\large{}$5$};
\path (6.5,16.51) node[text=black,anchor=base west] {\large{}$2$};
\end{tikzpicture}%
      \caption{Ordering of maximal simplices in a simplex chimney.}
     \label{fig:chimneyOrdering}
   \end{figure}
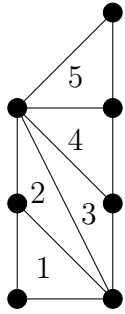
\end{proof}

This method of pull-back subdivisions and induction on dimension works
nicely on the class of recursively defined polytopes known as
Nakajima. A lattice polytope is a \emph{Nakajima polytope} if it is a single
lattice point or it is of the form $\chim(Q,\0,\vu)$ for a Nakajima
polytope $Q$. These are precisely those polytopes $P$ for which the
singularity $U_P$ is a local complete intersection (see
Section~\ref{sec:why:algebraic-geometry}).

\begin{corollary}
  Every Nakajima polytope has a quadratic triangulation. (A
  triangulation which is regular, unimodular and flag.)
\end{corollary}
\begin{proof}
  For a polytope $\chim(Q,\vl,\vu)$, the pull back of every lattice
  subdivision of $Q$ is lattice. Hence, we can apply
  Theorem~\ref{thm:chimney} recursively.
\end{proof}

\subsubsection{Push-forward subdivision}

To apply the chimney Theorem~\ref{thm:chimney} in a case
where $P$ has more than one functional bounding $P$ from below or from
above, we need the subdivision of the projected polytope $Q$ to
respect the intersections of the multiple upper and lower facets. To
this end we define the \emph{push-forward} of a subdivision.

Given a subdivision $\S$ of a lattice polytope $P$ in $\R^d$ and
a projection $\pi \colon \R^d \rightarrow \R^{d'}$, the
\emph{push-forward subdivision} $\pi_*\S$ of $Q:= \pi(P)$ is the
common refinement of the projections of all faces of $\S$
(including low-dimensional faces).

The following theorem tells us under what conditions we can still
apply Theorem~\ref{thm:chimney}.

\begin{theorem}
  Let $P \subset \R^d$ be a lattice polytope and $\pi \colon
  \R^d \to \R^{d-1}$  be the projection which forgets the last coordinate. 
  If $\S$ is a (regular) subdivision of $P$ such that every
  cell $F$ of $\S$ has a description
  $$ F = \{(\vx,y)\in \pi(F) \times \R \ : \ \vl_i(\vx) \le y \le
  \vu_j(\vx), \text{ for } 1 \le i \le r,\ 1 \le j \le s \}$$
  with integral linear functionals $\vl_1, \ldots, \vl_r$ and $\vu_1,
  \ldots, \vu_s$ such that $\vl_i \le \vu_j$ for $1 \le i \le r$, $1
  \le j \le s$ along the lattice polytope $Q$, and that the
  push-forward $\pi_*\S$ of $\S$ to $Q$ has a (regular) unimodular
  refinement, then $\S$ has a (regular) unimodular refinement. The
  degree of minimal non-faces will be preserved.
\end{theorem}

This theorem provides a heuristic for finding  regular unimodular
triangulations of a lattice polytopes. Namely, given a lattice polytope
$P$: search for unimodular transformations $\Phi$ of $P$ such
that $\Phi(P)$ has the above form; project to $Q$ and check whether
$Q$ has a regular unimodular refinement of the push-forward
subdivision; iterate. The push-forward and pull-back methods are
implemented in an extension to \texttt{polymake}~\cite{subdivProj-sw}
and have been used for triangulating smooth reflexive polytopes (see
Section~\ref{subsubsec:reflexive}).
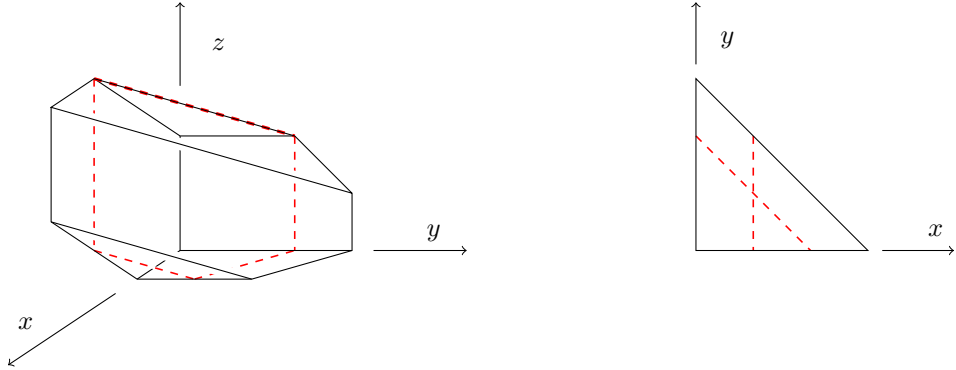
\begin{figure}[tb]
\begin{tikzpicture}[y=-1cm,scale=.6]

\draw[black] (4.1275,6.985) -- (5.08,6.35);
\draw[semithick,dashed,red] (7.62,3.81) -- (7.62,6.35);
\draw[semithick,dashed,red] (3.175,2.54) -- (3.175,6.35);
\draw[semithick,dashed,red] (17.78,3.81) -- (17.78,6.35);
\draw[semithick,dashed,red] (5.3975,6.985) -- (7.62,6.35);
\draw[semithick,dashed,red] (16.51,3.81) -- (19.05,6.35);

\draw[very thick,white] (3.175,6.35) -- (5.3975,6.985);

\draw[black] (3.175,2.54) -- (5.08,3.81) -- (7.62,3.81);
\draw[black] (5.08,3.81) -- (5.08,6.35) -- (8.89,6.35);
\draw[semithick,dashed,red] (3.175,6.35) -- (5.3975,6.985);

\draw[line width=4.5bp,white] (2.2225,5.715) -- (6.6675,6.985);
\draw[line width=4.5bp,white] (2.2225,3.175) -- (8.89,5.08);
\draw[very thick,dashed,red] (3.175,2.54) -- (7.62,3.81);

\draw[black] (2.2225,3.175) -- (2.2225,5.715) -- (4.1275,6.985) -- (6.6675,6.985) -- (8.89,6.35) -- (8.89,5.08) -- (7.62,3.81) -- (3.175,2.54) -- (2.2225,3.175) -- (8.89,5.08);
\draw[black] (6.6675,6.985) -- (2.2225,5.715);
\draw[arrows=-to,black] (3.65125,7.3025) -- (1.27,8.89);
\draw[arrows=-to] (5.08,2.69875) -- (5.08,0.84667);
\draw[arrows=-to] (9.36625,6.35) -- (11.43,6.35);
\draw[arrows=-to] (16.51,2.2225) -- (16.51,0.84667);
\draw[arrows=-to] (20.6375,6.35) -- (22.225,6.35);
\draw[black] (16.51,2.54) -- (16.51,6.35) -- (20.32,6.35) -- cycle;
\path (1.27,8.09625) node[anchor=base west] {\fontsize{9.6}{11.52}\selectfont{}$x$};
\path (10.31875,6.0325) node[anchor=base west] {\fontsize{9.6}{11.52}\selectfont{}$y$};
\path (5.55625,1.905) node[anchor=base west] {\fontsize{9.6}{11.52}\selectfont{}$z$};
\path (21.43125,6.0325) node[anchor=base west] {\fontsize{9.6}{11.52}\selectfont{}$x$};
\path (16.8275,1.74625) node[anchor=base west] {\fontsize{9.6}{11.52}\selectfont{}$y$};

\end{tikzpicture}%
   \caption{The projections $P_{xyz}$ and $P_{xy}$}
  \label{fig:Z1xyz}
\end{figure}

Here  is an example.  Consider  the  following polytope given by eight
inequalities in variables $x,y,z,w$.
\begin{equation}\label{ap_eq1}
  \begin{array}{rcccl}
    0 & \le & x & \\
    0 & \le & y & \le & 3-x \\ 
    0 & \le & z & \\
    x-1 & \le & z \\
    0 & \le & w & \le & 2+x-z  \\
    & & w & \le & 4-y-z 
  \end{array}
\end{equation}
We have ordered the inequalities so that each variable is bounded
above or below by integral linear functionals in the previous
variables.  We want to project $P$ to $x$-$y$-$z$-space.  This
projection $P_{xyz}$ has the representation (see
Figure~\ref{fig:Z1xyz} on the left)
\begin{equation*}
  \begin{array}{rcccl}
    0 &\le& x \\
    0 &\le& y &\le& 3-x \\
    0 &\le& z &\le& 2+x \\
    x-1 &\le& z &\le& 4-y\;.
  \end{array}
\end{equation*}
Observe that $P_{xyz}$ has facets $z \le 2+x$ and $z \le 4-y$ whose
pull-backs are not facets of $P$. They are implied by the inequalities
$0 \le w$ and $w \le 2+x-z$, respectively $w \le 4-y-z$.

The push-forward of the trivial subdivision of $P$ divides $P_{xyz}$
along the plane $x+y=2$, the projection of the ridge formed by the two
upper bounds on $w$ in~(\ref{ap_eq1}),
\begin{equation*}
  \begin{array}{rcccl}
    0 &\le& w &\le& 2+x-z  \\
    && w &\le& 4-y-z \ \raisebox{2.5mm}[0pt][0pt]{$\Big\} \ x+y=2$\;.}
  \end{array}
\end{equation*}
This is a lattice subdivision, as the intersection of this hyperplane
with $P_{xyz}$ is the convex hull of the lattice points $(1,1,0)$,
$(0,2,0)$, $(0,2,2)$, $(2,0,4)$, and $(2,0,1)$.

We can project this again to obtain a subdivided polytope $P_{xy}$ in
the $x$-$y$-plane given by the inequalities $0 \le x$ and $0 \le y \le
3-x$ (see Figure~\ref{fig:Z1xyz} on the right).  Any (regular and
unimodular) triangulation of this subdivision can be used to construct
a triangulation of $P$.

\subsection{Joins and (Fiber) Products}
\label{sec:joins-and-products}

\subsubsection{Products}
Let $\T$ and $\T'$ be subdivisions of $P$ and $P'$ respectively. Then
$$ \T \times \T' := \left\{ F \times F' \colon F \in \T, F' \in \T'
\right\}$$ 
is the \emph{product subdivision} of $P \times
P'$. Un\-for\-tu\-na\-te\-ly, the product of two
triangulations is not a triangulation. It is a subdivision into
products of simplices.  The product of two unimodular simplices is 
totally unimodular (see~\cite[p.72, Ex.(9)]{SturmfelsGBCP},~\cite[p.~282]{LeeHandBook},
or~\cite[Section 6.2.2]{deLoeraRambauSantos}; alternatively, think of it as the
undirected edge polytope of a complete bipartite graph and apply Lemma~\ref{lemma:edge-poly}(\ref{lemma:bipartite-edge-poly:bipartite})). 
In particular, all its triangulations are unimodular.

There is a particularly nice triangulation of a product of simplices $\Delta^d \times \Delta^{d'}$
(compare~\cite[p.~282]{LeeHandBook},~\cite[Section
6.2.3]{deLoeraRambauSantos}).  To define it, order the vertices of the factors
$\va_0 \prec \ldots \prec \va_d$ and $\va'_0 \prec \ldots \prec
\va'_{d'}$. This induces a componentwise partial order on the
vertices of $\Delta^d \times \Delta^{d'}$. The family of totally
ordered subsets yields a quadratic triangulation, called the
\emph{staircase triangulation}. One geometric way to construct this
triangulation is to pull the vertices of $\Delta^d \times \Delta^{d'}$
in the lexicographic order.

\begin{proposition} \label{prop:products}
  Let $P$ and $P'$ be lattice polytopes. If both admit regular
  unimodular triangulations $\T$ and $\T'$, then so does $P \times P'$.

  The set of minimal non--faces consists of lifts of minimal
  non--faces from $P$ and $P'$ together with non--faces of cardinality
  two.
\end{proposition}
\begin{proof}
  If $\T$ and $\T'$ are regular, then the subdivision $\T\times \T'$ of $P \times P'$ into products of unimodular
  simplices is regular, and
  any triangulation that refines it is unimodular. 
  In order to control the non--faces, order the lattice
  points $\vp_1 \prec \ldots \prec \vp_r$ in $P$ and $\vp_1' \prec
  \ldots \prec \vp_s'$ in $P'$. Then pull the lattice points
  $(\vp_i,\vp_j')$ in $P \times P'$ lexicographically.

  Consider a non--face $N$. If both its projections to $P$ or $P'$
  are faces, $N$ is a non--face in a staircase triangulation.
\end{proof}

Proposition~\ref{prop:products} can be extended to
non-regular triangulations.
This extension is at the heart of the counter-examples constructed by
Santos in~\cite{SantosNoFlips, SantosNonconnectedHilbertScheme}
(see also~\cite[Ch.~7]{deLoeraRambauSantos}) which we will return
 to in Section~\ref{sec:Hilbert}. The
main idea is that if we do not care about the regularity
of the unimodular triangulation of $P\times P'$ we do not need to
refine $\T\times \T'$ by pulling vertices. Any refinement of the
individual cells of $\T\times \T'$ is unimodular, and the only concern is that the different refinements agree on common
faces. 
Using staircase refinements of each product of simplices will still accomplish this. 
It does not require a globally
defined ordering of all the vertices of $\T$ and of $\T'$, but only a
local ordering of the vertices in each individual simplex. These local
orderings can be represented via an acyclic orientation of the
1-skeleton of each simplex, as follows.
Let $\T$ be a triangulation of a point configuration. A \emph{locally
  acyclic orientation} of the one-skeleton of $\T$ (or a locally
acyclic orientation of $\T$) is an assignment of a direction to each
edge such that no simplex contains a directed cycle (equivalently, no
triangle in $\T$ is a directed three-cycle).

\begin{proposition}
\label{prop:lao}
Let $\T_1$ and $\T_2$ be triangulations of $P_1$ and $P_2$ with
locally acyclic orientations.
Refining each product of simplices in $\T_1\times \T_2$ in the
staircase manner indicated by the orientations produces a
triangulation $\T$ of $P_1\times P_2$ with the following properties:
\begin{compactenum}
\item $\T$ is unimodular if and only if  $\T_1$ and $\T_2$ are both
  unimodular. 
\item $\T$ is regular if and only if  $\T_1$ and $\T_2$ are both
  regular and the orientations are globally acyclic. 
\item $\T$ is flag if and only if  $\T_1$ and $\T_2$ are both flag.
\end{compactenum}
\end{proposition}

The triangulation of Proposition~\ref{prop:lao} is called the
\emph{staircase refinement} of $\T_1\times \T_2$ with respect to the
corresponding locally acyclic orientations.

\begin{proof}
We made it clear above that the triangulation is well-defined.
For (1)  the refinement is unimodular if the factors are unimodular
  since every triangulation of the product of two simplices is
  unimodular. Conversely, if a simplex in one of the factors is not
  unimodular, then the staircase refinement does not give product
  cells using that simplex unimodular refinements. 
 For (2), observe that the restriction of $\T$ to every
  $\{\va_1\}\times P_2$ or $P_1\times\{\va_2\}$ for vertices $\va_1$
  of $P_1$ or $\va_2$ of $P_2$ is affinely isomorphic to $\T_2$ or
  $\T_1$, respectively. So, both must be regular for $\T$ to be
  regular. If the locally acyclic orientation of, say, $\T_1$, has a
  cycle $\va_0,\va_1,\dots,\va_k=\va_0$ then for every (oriented) edge
  $\vb_1\vb_2$ in $\T_2$  the edges $(\va_i,\vb_1)(\va_{i+1},\vb_2)$ ,
  $i=0,\dots,k$ are in $\T$, which implies $\T$ is not
  regular. Indeed, the circuit
\[
(\va_i,\vb_1) + (\va_{i+1},\vb_2) = (\va_i,\vb_2)(\va_{i+1},\vb_1)
\]
implies that if the edge $(\va_i,\vb_1)(\va_{i+1},\vb_2)$ appears
the following inequality must hold for the weight vector $\vomega$:
\[
\vomega(\va_i,\vb_1) + \vomega(\va_{i+1},\vb_2) < \vomega(\va_i,\vb_2)
\vomega(\va_{i+1},\vb_1).
\]
Summing this over all $i$ yields the impossible equation
\[
\sum_i \vomega(\va_i,\vb_1) + \sum_i\vomega(\va_{i},\vb_2) <
\sum_i\vomega(\va_i,\vb_2) \sum_i\vomega(\va_{i},\vb_1).
\]
This proves it is necessary for $\T_1$ and $\T_2$ to be regular. 
For sufficiency, note that if
$\T_1$ and $\T_2$ are regular,  $\T_1\times \T_2$ is a regular
subdivision. Now if the orientations of $\T_1$ and $\T_2$'s 1-skeletons are globally
acyclic, they can be extended to give total orderings on the vertices of
$\T_1$ and $\T_2$, and our staircase refinement can be obtained by
pulling the vertices of $\T_1\times \T_2$ with respect to the
lexicographic ordering, which yields a regular triangulation.

For (3), flagness  follows from the characterization of minimal non-faces
  of the staircase refinement stated in
  Proposition~\ref{prop:products}.
\end{proof}

\subsubsection{Joins}

Let $P$ and $P'$ be polytopes of dimension $d$ and $d'$, and $\0_k$
the origin in $\R^k$. The join $P\star P'$ of $P$ and $P'$ is the
convex hull of
\begin{align*}
  P\times \{\0_{d'}\}\times\{0\}\quad\cup\quad\{\0_d\}\times P'\times
  \{1\}.
\end{align*}
This gives a $(d+d'+1)$-dimensional polytope.  The join of two
simplices, $\Delta^r \star \Delta^{r'}$ is a simplex
$\Delta^{r+r'+1}$. Any subdivisions $\S$ and $\S'$ of $P$  and $P'$
lift to a subdivision $\T$ of $P\star P'$ by taking all joins of cells
in $\S$ and $\S'$, and every subdivision of $P\star P'$ can be
obtained this way. In particular, if $\S$ and $\S'$ are (regular,
unimodular, flag) triangulations of $P$ and $P'$, then $\T$ is a
(regular, unimodular, flag) triangulation of $P\star P'$.

The toric ring of the join is the tensor product of the components:
$R_{P\star P'} = R_P \otimes R_{P'}$ (compare
Section~\ref{sec:Commutative algebra}). 

The facets of $P\star P'$ are joins of $P$ with facets of $P'$ and
joins of $P'$ with  facets of $P$.
Hence $P\star P'$ is compressed if and only if both $P$ and $P'$ are
compressed.

The join can be defined for more than two factors in a similar way,
and is associative. Just as $P \star P'$ has a canonical projection to
$\Delta^1$ --- the last coordinate, $P_0 \star \ldots \star P_r$ has a
canonical projection to $\Delta^r$.

\subsubsection{Fiber Products}
\label{sec:fiberproducts}
Suppose two lattice polytopes project linearly,
respecting their lattices, to the same lattice polytope:  $P
\overset{\pi}{\rightarrow} Q \overset{\ \pi'}{\leftarrow} P'$. Then
the \emph{polyhedral fiber product}, also known as the
\emph{multigraded Segre product} $P \times_Q P'$ is the polytope
$$\left\{ (\vp,\vp') \in P \times P' \ : \ \pi(\vp)=\pi'(\vp')
\right\}\ .$$
This construction was first used by Buczy\'nska and
 Wisniewski in the study of statistical models of binary
symmetric phylogenetic trees~\cite{WisniewskiBuczynska}. 
A closely related \emph{toric fiber product} 
\[
\conv \left\{ (\vp,\vp') \in (P\cap\Z^d) \times (P'\cap\Z^{d'}) \ :
  \ \pi(\vp)=\pi'(\vp') \right\}
  \]
was defined by Sullivant~\cite{SullivantToricFiber}.

Under the assumptions of the following theorem (which includes
the phylogenetic case) $P \times_Q P'$ is a lattice polytope, so
the two notions agree.

\begin{theorem} \label{thm:fiber-product}
  Let $P \overset{\pi}{\rightarrow} Q
  \overset{\ \pi'}{\leftarrow} P'$ be lattice preserving projections.
 If $Q$ admits a unimodular triangulation $\T$, and  $P$
  and $P'$ have unimodular triangulations $\S$ and $\S'$, which refine the
  pull-back subdivisions $\pi^*\T$ and $\pi'^*\T$ respectively,
  then $P \times_Q P'$ admits a unimodular triangulation.

  Further, regularity and the degree of minimal non-faces can be
  preserved.
\end{theorem}

Before proving this theorem, let us state as a corollary 
a slight generalization of a result of Sullivant~\cite[Cor.~15]{SullivantToricFiber} for the case where $Q$ is a
unimodular simplex. In this case, the pull-back
subdivision is trivial.

\begin{corollary} \label{cor:fiber-product}
If $P \overset{\pi}{\rightarrow} \Delta^r \overset{\ \pi'}{\leftarrow}
P'$ are lattice preserving projections such that both $P$ and $P'$
admit unimodular triangulations, then $P \times_{\Delta^r} P'$ admits
a unimodular triangulation.

Regularity and degree of minimal non-faces can be preserved.
\end{corollary}

A lattice polytope projecting to a unimodular simplex is known in the
literature as a Cayley sum of the fibers of the simplex
vertices~\cite[Ch.9,~eq.(1.2)]{GKZ}. The fiber product in the above
corollary is the Cayley sum of the products of the fibers.

The proof of Theorem~\ref{thm:fiber-product} requires the following
lemma.

\begin{lemma} \label{lemma:fiber-products}
  Given lattice preserving projections  $\Delta^d
  \overset{\pi}{\rightarrow} \Delta^r \overset{\ \pi'}{\leftarrow}
  \Delta^{d'}$, the fiber product $\Delta^d \times_{\Delta^r}
  \Delta^{d'}$ is a lattice polytope.
\end{lemma}

\begin{proof}
  For $\nu = 1, \ldots, r+1$ let $I_\nu := \{ \ i \ : \ \pi(\ve_i) =
  \ve_\nu \}$ and $I'_\nu := \{ \ j \ : \ \pi'(\ve_j) = \ve_\nu \}$.
  With this notation, the fiber product has the inequality description
  \begin{align*}
    (\vp,\vp') &\ge \0 \ , \\
    \sum_{i=1}^{d+1} p_i &= \sum_{j=1}^{d'+1} p'_j = 1 \ , \\
    \sum_{i \in I_\nu} p_i &= \sum_{j \in I'_\nu} p'_j
    \quad\text{ for }\quad 1 \le \nu \le r+1\ .
  \end{align*}
  The equation matrix is (after omission of repeated columns) of the
  form
  $$ \begin{pmatrix}
    1 \cdots 1 & 0 \cdots 0 \\
    0 \cdots 0 & 1 \cdots 1 \\
    \id_{r+1} & -\id_{r+1}
  \end{pmatrix}$$
  which is a totally unimodular matrix.  
\end{proof}

\begin{proof}[Proof of Theorem~\ref{thm:fiber-product}]
Let $\T$, $\S$ and $\S'$ be as in the statement. $\S$ and $\S'$ give a subdivision
$\S\times \S'$ of $P \times P'$ into products of unimodular
simplices. We claim that intersecting $\S \times \S'$ with $P \times_Q P'$ 
gives a lattice subdivision of $P \times_Q P'$. Indeed, consider a cell in this subdivision
 $(F\times F') \cap (P \times_Q P')$ 
a cell in this subdivision, 
for some unimodular
simplices $F$ and $F'$. 
Since $\pi$ and $\pi'$ map simplices of $\S$ and $\S'$ to simplices in $\T$, 
$\pi F$ and $\pi F'$ are simplices in $\T$ and, in fact, (assuming $F\times F'$ is the minimal 
product of simplices containing our cell) they are the same simplex $G$. Then
\[
(F\times F') \cap (P \times_Q P') = F\times_G F'
\]
which, by Lemma~\ref{lemma:fiber-products}, is a lattice polytope. 

Thus, we have a lattice regular subdivision of $P \times_Q P'$ into totally unimodular cells.
(For regularity, observe that the intersection of a regular subdivision with an affine subspace is regular).
Any refinement of it into a triangulation of $P \times_Q P'$ is unimodular. If the refinement
is done, for example, pulling all the lattice points in $P \times_Q P'$ lexicographically as
in the proof of Proposition~\ref{prop:products}, it will also be regular. The 
non--face statement follows as in Proposition~\ref{prop:products}.
\end{proof}

Note that Theorem~\ref{thm:fiber-product} is true for more than two factors, by induction, since the triangulation obtained refines the pull-back of
$P\times_Q P' \to Q$.

\subsubsection{Semidirect products}

Motivated by a construction in algebraic statistics, Aoki, Hibi,
Ohsugi and Takemura introduced \emph{nested configurations}.
Given lattice polytopes $Q \subseteq k\Delta^d$ and $P_i \subset
\R^{d_i}$ for $i = 1, \ldots, d+1$, the \emph{nested polytope} $\NP(Q;P_1,
\ldots, P_{d+1})$ is the convex hull, in $\R^{d+1} \times \prod \R^{d_i}$ of
the polytopes $\{\va\} \times \prod a_i P_i$, where $\va$ runs over
the vertices of $Q$. 
Here and in what follows, $\Delta^d=\{(x_1,\dots,x_{d+1}): \vx\ge 0, \sum x_i=1\}$ denotes
the homogeneous unimodular $d$-simplex in $\R^{d+1}$.

The following is an equivalent definition in terms of joins.
\[
\NP(Q;P_1,\ldots, P_{d+1}) = k \cdot (P_1 \star \dots \star P_{d+1}) \cap
\pi^{-1}(Q),
\]
where $\pi: \R^{d+1} \times \prod \R^{d_i} \to \R^{d+1}$ is the natural
projection.

In~\cite{0907.3253}, Hibi and Ohsugi show 
that many properties of the input polytopes can be inherited by nested polytopes.
These include
normality and the existence of regular
unimodular triangulations as well as degrees of Gr\"obner bases (or
generators).
Their proof uses the algebraic-geometric machinery from
section~\ref{sec:toric-groebner-bases}, so no statements about
non-regular triangulations can be concluded from it.
Here we offer a purely combinatorial proof.
(The definition
in~\cite{AokiHibiOhsugiTakemuraNested} takes more general
configurations as input, but if one is only interested in the normal
case, no generality is lost by taking our definition.)

We first introduce the following alternative way of looking at nested configurations.
Given lattice polytopes $Q\subset \R^d$ and $P_i \subset \R^{d_i}$ for $i = 1, \ldots, n$
and an integer affine map, $\phi:\Z^d \to \Z^n$, that is nonnegative on $Q$,
the \emph{semidirect product} of $Q$ and the tuple $(P_1,\dots,P_n)$ along the map $\phi$ is defined as
\[
Q \ltimes_\phi (P_1,\dots, P_n):= \conv_{\va \in Q} \left(\{\va\} \times \prod \phi_i(\va) P_i\right),
\]
where $(\phi_1,\dots,\phi_n)$ are the coordinates of $\phi$.

If $n=d+1$, $Q\subset k \Delta^d$,  and $\phi$ is the identity map, we recover the definition of nested configuration. Conversely, every semidirect product can be 
rewritten as a nested configuration as follows. If $\phi$ is injective, unimodular (meaning that $\phi(\Z^d)=\aff( \phi(\Z^d) ) \cap \Z^n)$) and homogeneous (meaning that 
$\phi(\Z^d)\subset \{ \sum x_i=k\}$ for some $k\in \N$) then
\[
Q \ltimes_\phi (P_1,\dots, P_n) \cong \NP(\phi(Q) ; P_1,\dots, P_n).
\]
If $\phi$ is not injective, not unimodular, or not homogeneous, consider the modified map $\widetilde \phi=(\phi, \operatorname{Id}, f): \R^d \to \R^n\times \R^d\times \R$, where $f(x)=k- \sum_i \phi(x)_i -\sum_j x_j$ for a sufficiently large $k$, 
take $P_i$ to be a single point for all $i>n$ and observe that
\[
Q \ltimes_\phi (P_1,\dots, P_n) \cong \NP(\widetilde\phi(Q) ; P_1,\dots, P_n, \mathrm{pt}, \dots, \mathrm{pt}).
\]

So, semidirect products are  really an equivalent  operation to nested configurations. But we find them conceptually easier to handle.
They generalize the following constructions:
\begin{itemize}
\item $\Delta^d \ltimes_\mathrm{Id} (P_0,\ldots, P_d)$ is the join of $P_0,\dots,P_d$,
\item $\{\mathrm{pt}\}\ltimes_{\mathrm{1}} (P_0,\ldots, P_d)$ is the product of $P_0,\dots,P_d$,
as is $P_0\ltimes_{\mathrm{1}} (P_1,\ldots, P_d)$. In both cases $1$ denotes the constant map with image $(1,\dots,1)$.
\item $\{\mathrm{pt}\}\ltimes_{k} (P)$ is the $k$-th dilation of $P$ and $\{\mathrm{pt}\}\ltimes_{(k_1,\dots, k_d)} (P_1,\dots, P_d)$
is the product $\prod k_i P_i$.
\item The chimney $\chim(Q,\vl,\vu)$ associated to two integer functionals $\vl\le \vu$ on $Q$ is equivalent to the semidirect product
$Q \ltimes_{\vu-\vl} I$, where $I$ is a unimodular segment. In particular, a Nakajima polytope is one that can be obtained as
\[
(\dots (\{\mathrm{pt} \ltimes_{\phi_1} I)\dots )\ltimes_{\phi_d} I,
\]
for certain choice of functionals $\phi_i$.
\end{itemize}

There are two ways of relating any semidirect product with several factors to semidirect products with only two factors at a time.
One is as a special case of fiber products, taking into account  that every semidirect product comes with a canonical projection $Q \ltimes_\phi (P_1,\dots, P_n)\to Q$.
The other is a special associativity property:

\begin{lemma}
\label{lemma:semidirect_associative}
As in the definition of the semidirect product,
 let $\phi_1,\dots,\phi_n$ denote the coordinates of $\phi: \R^d\to \R^n$.
Then:
\begin{eqnarray*}
Q \ltimes_\phi (P_1,\dots, P_d) 
  &=& (Q \ltimes_{\phi_1} P_1) \times_Q \dots \times_Q (Q \ltimes_{\phi_n} P_n)\\
  &=& (\dots (Q \ltimes_{\widetilde\phi_1} P_1) \ltimes_{\widetilde\phi_2} \dots \ltimes_{\widetilde\phi_d} P_d),
\end{eqnarray*}
where $\widetilde \phi_i$ denotes the composition of 
the natural projection $\Z^d \times \Z^{d_1+\dots+d_{i-1}}\to \Z^d$ with $\phi_i : Z^d\to \Z$.
\qed
\end{lemma}

We are going to prove that

\begin{theorem} \label{thm:nested}
  Suppose $Q$ and  $P$ admit unimodular triangulations $\S$ and $\T$. Then, the semidirect product
  $Q\ltimes_\phi P$ admits a unimodular triangulation that refines the pull-back $\pi^*\S$ of $\S$ by the projection
   $\pi: Q\ltimes_\phi P\to Q$. 
  \end{theorem}

\begin{corollary} \label{coro:nested}
  If $Q,P_1,\ldots, $and  $ P_n$ admit unimodular triangulations, then every
  semidirect product  $Q\ltimes_\phi (P_1,\ldots, P_n)$ admits one too. 
\end{corollary}

\begin{proof}
Lemma~\ref{lemma:semidirect_associative} gives two different ways to derive the corollary from Theorem~\ref{thm:nested}.
One is by associativity, using induction on $n$. The other is via the relation to fiber products, using Theorem~\ref{thm:fiber-product}.
\end{proof}

The key step for the proof of Theorem~\ref{thm:nested} is to look at the case where both $Q$ and $P$ are unimodular simplices. So, let $Q\subset\R^d$ be a unimodular $d$-simplex and $P=\conv\{\vp_0,\dots,\vp_e\} \subset\R^e$ a unimodular $e$-simplex. As implied by our notation, the vertices of $P$ are considered with a given specific order, which will be important both for the construction on this particular simplex and for gluing the constructions between simplices. Let $f_j: P \to \R$ be the affine functional that takes the value zero in $\vp_0,\dots, \vp_{e-i}$ and one in $\vp_{e-j+1},\dots, \vp_e$. Put differently:
\[
 P =
  \left\{ {\vx} \in \R^{e}
  \ : \ 
      0 \le f_1(\vx) \le \dots \le f_e(\vx)
      \le
      1
\right\}.
\]
(If $P$ is the \emph{standard ordered $e$-simplex} $\{\vx: 0\le x_1,\dots , x_e\le 1\}$ then $f_j$ is the $j$-th coordinate). 

Let $\phi:Q\to \R$ be the affine functional in the definition of semidirect product. Then,
\[
Q\ltimes_\phi P =
  \left\{ \binom{\vy}{\vx} \in Q \times \R^{e}
  \ : \ 
      0 \le f_1(\vx) \le \dots \le f_e(\vx)
      \le
      \phi(\vy)
\right\}.
\]
In this setting we can define the canonical slicing of $Q\ltimes_\phi P$.
 For each $b \in \R$ let
$\phi_{\le b}: Q\to \R$ be the (unique) affine functional with $\phi_{\le b}(\vq)=\min(\phi(\vq),b)$ on each vertex $\vq$ of $Q$, 
and for each $b \in \N$ and each $j\in[e]$ consider the hyperplane
\[
H(j,b) := \left\{ \binom{\vy}{\vx} \in \R^{d} \times \R^{e}
  \ : \ f_j(\vx) = \phi_{\le b} (\vy)
    \right\} .
\]
The \emph{canonical slicing} of $Q\ltimes_\phi P$ is the polyhedral subdivision obtained slicing $Q\ltimes_\phi P$ by all these hyperplanes.

Figure~\ref{fig:nested} shows the canonical slicing in the case
$d=1$, $e=2$, with $\phi$ taking the values two and five on the vertices of the segment $Q$.
\begin{figure}[tb]
  \centering
\begin{tikzpicture}[y=-1cm,scale=.3]

\draw[very thick,join=round,black] (8,5) -- (15.5,10) -- cycle;
\draw[semithick,join=round,red] (4.5,9) -- (4.5,11) -- (9.5,6) -- (9.5,4) -- cycle;
\draw[semithick,join=round,red] (6,8) -- (11,3) -- (11,7) -- (6,12);
\draw[semithick,join=round,red] (12.5,2) -- (12.5,8) -- (6,12);
\draw[semithick,join=round,red] (14,1) -- (14,9) -- (6,12);
\draw[join=round,black] (3,10) -- (8,5);
\draw[join=round,black] (6,12) -- (15.5,10);

\draw[line width=4.5bp,join=round,white] (4.5,9) -- (6,10) -- (15.5,8) -- (9.5,4) -- cycle;

\draw[semithick,join=round,blue] (4.5,9) -- (6,10) -- (15.5,8) -- (9.5,4) -- cycle;

\draw[line width=4.5bp,join=round,white] (6,8) -- (15.5,6) -- (11,3) -- cycle;

\draw[semithick,join=round,blue] (6,8) -- (15.5,6) -- (11,3) -- cycle;

\draw[line width=4.5bp,join=round,white] (6,8) -- (15.5,4) -- (12.5,2) -- cycle;

\draw[semithick,join=round,blue] (6,8) -- (15.5,4) -- (12.5,2) -- cycle;

\draw[line width=4.5bp,join=round,white] (6,8) -- (15.5,2) -- (14,1) -- cycle;

\draw[join=round,arrows=-to,black] (3,10) -- (3,4);
\draw[join=round,arrows=-to,black] (3,10) -- (9,14);
\draw[join=round,arrows=-to,black] (8,5) -- (8,-1);
\draw[join=round,arrows=-to,black] (8,5) -- (20,13);
\draw[semithick,join=round,blue] (6,8) -- (15.5,2) -- (14,1) -- cycle;
\draw[very thick,join=round,black] (6,8) -- (15.5,0);
\draw[very thick,join=round,black] (6,12) -- (15.5,10);
\draw[very thick,join=round,black] (3,10) -- (8,5);

\draw[very thick,join=round,black] (6,8) -- (6,12) -- (3,10) -- cycle;
\draw[very thick,join=round,black] (15.5,10) -- (15.5,0) -- (8,5);

\path (1,4.09625) node[anchor=base west] {\fontsize{9.6}{11.52}\selectfont{}$x_1$};
\path (9,13.5) node[anchor=base west] {\fontsize{9.6}{11.52}\selectfont{}$x_2$};

\end{tikzpicture}%
   \caption{The hyperplanes $\color{red}H(1,b)$ and
    $\color{blue}H(2,b)$ for $d=1$, $e=2$, and $\phi(\vq_1)=2$, $\phi(\vq_2)=5$}
  \label{fig:nested}
\end{figure}

\begin{lemma}
\label{lemma:nested_canonical}
Let $Q$ and $P=\conv\{\vp_0,\dots,\vp_e\}$ be unimodular simplices, with the vertices of $P$ given in a specified ordering, and let $\phi:Q\to \R$ be a nonnegative integer affine function on $Q$.
Then:
\begin{enumerate}
\item The canonical slicing of every face of $Q\ltimes_\phi P$ coincides with the restriction to that face of the canonical slicing of $Q\ltimes_\phi P$.
\item The canonical slicing is a lattice subdivision (all  vertices are integer).
\item All cells in the canonical slicing are compressed.
\end{enumerate}
\end{lemma}

\begin{proof}
Part (1) is trivial.

For part (2), let $\binom{\bar\vy}{\bar\vx}\in Q\ltimes_\phi P$ be a vertex of the canonical slicing. By part (1) there is no loss of generality in assuming that
$\binom{\bar\vy}{\bar\vx}$ lies in the interior of $Q\ltimes_\phi P$, in particular $\bar\vy$ lies in the interior of $Q$. For such an $\bar\vy$ the function $b \mapsto
\phi_{\le b}(\bar\vy)$ is (continuous and) strictly increasing in the
range $b \le \max_{\vq}(\phi(\vq))$. Hence $\binom{\bar\vy}{\bar\vx}$ lies in at most one hyperplane $H(j,b)$ for each $j\in [e]$. In order for these hyperplanes to define a vertex, we need to have (at least) $d+e$ of them, so that the only possibility is $d=0$. But when $d=0$, $Q\ltimes_\phi P$ is just the $k$-th dilation of $P$, where $k$ is the value taken by $\phi$ in the (unique) point of $Q$. The hyperplanes $H(j,b)$ are of the form ``$f_j$ equals a constant'', and the facet-defining hyperplanes of $kP$ are of the form $f_{j+1}-f_j=0$. Together they form a totally unimodular system of possible facet normals (in the basis consisting of the $f_j$'s, which is itself unimodular), so the slicing they produce can only have integer vertices.

The inductive argument above implies that every vertex of the canonical slicing lies in one of the fibers $\{\vq\} \times \phi(\vq) P$ where $\vq$ is a vertex of $Q$,  which will be useful in the last part of the proof.

For part (3), we consider the three possible types of facets separately. 
\begin{itemize}
\item Those contained in facets $F\ltimes_\phi P$, where $F$ is a facet of $Q$, have width one since $Q$ is unimodular and $Q\ltimes_\phi P$ projects to it. 

\item For, those defined by a hyperplane $H(j,b)$, observe that cells incident to that hyperplane are contained between the hyperplanes $H(j,b-1)$ and $H(j,b+1)$. For every point $\binom{\bar\vy}{\bar\vx}$ in those cells we have
\[
\phi_{\le b}(\bar\vy)-1 \le \phi_{\le b-1}(\bar\vy) \le f_j(\bar\vx) \le \phi_{\le b+1}(\bar\vy) \le \phi_{\le b}(\bar\vy)+1,
\]
which proves they have width one.

\item 
For, those contained in facets $Q\ltimes_\phi F_j$ of $Q\ltimes_\phi P$, where $F_j$ is a facet of $P$, 
recall that $F_j$ is defined by the equation $f_j(\vx)=f_{j-1}(\vx)$, $j=1,\dots, e+1$ (with the convention $f_0\equiv 0$ and $f_{e+1}\equiv 1$). The facet $Q\ltimes_\phi F_j$ lies in the hyperplane (in $\R^{d} \times \R^{e}$ defined by the same equation (in the $\R^e$ variables), except for the  $j=e+1$ case, where the equation defining $Q\ltimes_\phi F_{e+1}$ is
$f_{e+1}(\vx) = \phi(\vy)$.
All vertices of the canonical slicing are in fibers over the vertices of $Q$. Since the canonical slicing restricted to these fibers is a dicing by a system of totally unimodular vectors (the same system of vectors as in the proof of part (2)), the cells in each individual fiber have width one. This implies that the cells in the whole slicing also have width one with respect to these functionals.
\end{itemize}
\end{proof}

\begin{proof}[Proof of Theorem~\ref{thm:nested}]
Suppose $Q,P$ admit unimodular triangulations $\S,\T$ respectively.
We then have that the cells
\[
B \ltimes_\phi C,
\]
for all the simplices $B\in \S$ and $C\in \T$ form a subdivision
of $Q\ltimes_\phi P$. (To show this, if $\phi$ is strictly positive on
$Q$ observe that $Q\ltimes_\phi P$ is projectively equivalent to $Q\times P$.
If $\phi$ is zero on some face of $Q$ then the corresponding face of $Q\times P$
collapses to lower dimension in $Q\ltimes_\phi P$, but the result is still true).

Now, Lemma~\ref{lemma:nested_canonical} tells us how to unimodularly subdivide
each $B \ltimes_\phi C$ into compressed lattice polytopes, and the canonical nature of these subdivisions guarantees
that they agree on common faces. Any triangulation that refines the subdivision obtained this way (e.g., by pulling all vertices)
is unimodular.
\end{proof}

\subsection{Toric Gr\"obner Bases} \label{sec:toric-groebner-bases}

The toric dictionary translates between the discrete geometry of
lattice points in polytopes and the algebraic geometry of toric
varieties. We explore the translations involving unimodular
triangulations.

\subsubsection{Unimodular triangulations and Gr\"obner bases}
\label{sec:UnimodularTriangulations}

Let $\mathbbm{k}$ be a field, and  $\configuration := (P \times
\{1\}) \cap \Z^{d+1}$ denote the homogenized set of lattice points in
$P$. Consider the polynomial ring
$S := \mathbbm{k}[x_{\va} \! : \! \va \in \configuration]$
with one variable for each lattice point.
There is a canonical ring homomorphism $\phi_P$ to the Laurent
polynomial ring $\mathbbm{k}[t_1^{\pm 1}, \ldots, t_d^{\pm 1}, t_{d+1}]$
mapping each variable to the corresponding (homogenized)
$\t$-monomial: $\phi_P(x_{\va}) = \vt^{\va} = t_1^{a_1} \cdot \ldots
\cdot t_d^{a_d} \cdot t_{d+1}$.
The \emph{toric ideal} $\ideal_P := \ker \phi_P$ is spanned, as a
$\mathbbm{k}$-vector space, by the set
\begin{equation} \label{eq:groebner-ideal-as-vector-space}
    \left\{ \ \vx^\vu-\vx^\vv \ : \ \vu, \vv \in \Z_{\ge 0}^\configuration
    \ , \  \sum_{\va \in \configuration} u_{\va} \va
    = \sum_{\va \in \configuration} v_{\va} \va \ \right\}  ,
\end{equation}
where $\vx^\vu = \prod_{\va \in \configuration}
x_{\va}^{u_{\va}}$. It thus encodes the affine dependencies
among the lattice points in $P$~\cite[Lemma~4.1]{SturmfelsGBCP}.

The tie between Gr\"obner bases of $\ideal_P$ and regular
triangulations of $P$ is established via two different interpretations
of a generic\footnote{To be on the safe side, assume that the numbers
  $\omega_{\va}$ are linearly independent over $\Q$.}
weight vector $\vomega \in \R^\configuration$. On the lattice polytope
side, $\vomega$ induces a regular triangulation $\T_\vomega$ of $P$
as explained at the end of Section~\ref{sec:what}. On the algebraic
side, such an $\vomega$ induces an ordering of the monomials in $S$ via
$\displaystyle \vx^\vm \prec \vx^\vn \ :\iff \langle \vomega, \vm \rangle <
\langle \vomega, \vn \rangle$.
For a polynomial $f 
\in S$ the \emph{leading term} $\lead_\vomega f$ is the biggest monomial
in this ordering which has a non-zero coefficient in $f$,
and the \emph{initial ideal} $\lead_\vomega I := \langle \lead_\vomega f \ :
\ f \in I \rangle$ of an ideal $I$ collects all the leading terms of
polynomials in $I$. A collection of polynomials in $I$ whose leading
terms generate the initial ideal is called a \emph{Gr\"obner basis} of $I$
(with respect to $\vomega$).
If $\vomega$ is not generic, we only get a partial ordering of the
monomials. A small enough generic perturbation of $\vomega$ refine the partial order
to a term order.

We will now investigate how (regular) subdivisions can help to find
generating sets and Gr\"obner bases of toric ideals.
For a vector $\vv \in \Z_{\ge0}^\configuration$ we call the set $\supp
\vv := \{\va \in \configuration \mid v_\va \neq 0 \}$ the
\emph{support} of $\vv$ and say that $\vx^\vv$ is supported on 
$F \subseteq P$ whenever $\supp \vv \subseteq F$. As a preliminary
step, suppose $P$ has a covering $\calC$ by integrally closed
polytopes (cf.~Section \ref{sec:integer-programming}). Then we can restrict
the generating set~(Section \ref{eq:groebner-ideal-as-vector-space}) of the
vector space $\ideal_P$ to those binomials $\vx^\vu-\vx^\vv$ which
have at least one monomial supported in a cell of $\calC$.

For a subdivision $\S$
of $P \subset \R^d$ into lattice polytopes, we call $N \subseteq
\configuration$ a non-face if $N \not\subset Q$ for all $Q \in
\S$. (Our two notions of a non-face --- for a triangulation and for a
subdivision --- agree for full triangulations.) 
If, for example, $\S$ comes from a lattice dicing as in
Section \ref{sec:hyperplanes}, then all minimal non-faces have size two.

Given $\vomega \in \R^\configuration$ with induced regular subdivision
$\S_\vomega$, and given a monomial $\vx^\vu \in S$, we call a monomial
$\vx^\vv$ \emph{standard}, written $\vx^\vv \in \stdo(\vx^\vu)$, if
$\vx^\vu-\vx^\vv \in \ideal_P$, and $\vv$ minimizes $\langle \vomega,
\vv \rangle$ subject to this condition. Again, we can restrict
the generating set~(Section \ref{eq:groebner-ideal-as-vector-space}) of the
vector space $\ideal_P$ to binomials $\vx^\vu-\vx^\vv$ for which
$\vx^\vv \in \stdo(\vx^\vu)$.
If the cells of $\S_\vomega$ are integrally closed 
then the definition of regular subdivision
says that $\vx^\vv$ is standard if and only if $\vx^\vv$ is supported
on a cell of $\S_\vomega$.
In particular, for every non-face $N \subseteq \configuration$ there
is a monomial $\vx^\vv$ supported on a cell of $\S_\vomega$ with
$f_N:=\vx^N-\vx^\vv \in \ideal_P$, where $\vx^N$ denotes the
squarefree monomial $\prod_{\va \in \configuration} x_\va$ with
support $N$.

For the following lemma, we consider the polynomial rings
$\mathbbm{k}[x_\va \! : \! \va \in \configuration \cap F]$ and their
ideals $\ideal_F$  as subsets of $S$.

A lifting function $\vomega$ producing the regular subdivision $\S_\vomega$ of $\configuration$ is \emph{tight}
if $(\va,\omega_\va)$ lies in the boundary of  $\tilde P:= \conv ( \va \times [\omega_\va, \infty) \ : \ \va
\in \configuration )$ for every $\va\in \configuration$. 
Observe that if $\omega_\va$ is not tight then there is a canonical way of making it tight: decrease the entries of $\omega_\va$ that are not tight until they are (without changing $\tilde P$). 
If all $\va \in \configuration$ are
vertices of $\S_\vomega$, then $(\S_\vomega,\vomega)$ is automatically tight.

\begin{lemma} \label{lemma:normal-covering}
  Suppose $\S_\vomega$ is a regular subdivision of the lattice
  polytope $P$ into integrally closed lattice polytopes.
  \begin{enumerate}
  \item The toric ideal $\ideal_P$ is generated by $\ideal_F$ for $F
    \in \S_\vomega$ together with $f_N$ for minimal non-faces $N$.
    \label{lemma:normal-covering:generation}
  \item If $\vomega$ is tight for $\S_\vomega$ then for any small enough
    generic perturbation $\vomega'$ of $\vomega$, 
         combining Gr\"obner
    bases for the $\ideal_F$ with respect to $\vomega'|_F$ with the
    $f_N$ for minimal non-faces
    $N$
    yields an $\vomega'$-Gr\"obner basis of $\ideal_P$.
     \label{lemma:normal-covering:Groebner}
  \end{enumerate}
\end{lemma}

Part (\ref{lemma:normal-covering:generation}) of this lemma will be
used 
in
Corollary~\ref{cor:rootB-quadratic}
to show type $\rootB$ polytopes are quadratically generated.
 It follows readily from
part (\ref{lemma:normal-covering:Groebner}).

\begin{proof}
  Take a binomial $f = \vx^\vu-\vx^\vv \in \ideal_P$ with $\vx^\vu =
  \lead_{\vomega'} f$.

  If $\supp \vu$ is a non-face of $\S_\vomega$,
  then there is a minimal non-face $N \subseteq \supp \vu$, and 
  $\lead_{\vomega'} f_N = \vx^N | \vx^\vu = \lead_{\vomega'} f$.
  
  If $\supp \vu$ is contained in a face $F$ of $\S_\vomega$, we claim
  that $\supp \vv$ must also be contained in $F$: we have
  $\langle \vomega, \vu \rangle \ge \langle \vomega, \vv \rangle$ (as
  $\vomega'$ is a small perturbation of $\vomega$ and we have strict
  inequality for $\vomega'$), and $\sum_{\va \in \configuration}
  u_{\va} \va = \sum_{\va \in \configuration} v_{\va} \va$ is an
  affine dependence. So the tightness condition
    yields $\omega_\va = \langle
  \veta_F, \va \rangle + \zeta_F$ for all $\va \in \supp \vv$.
\end{proof}

If we apply the preceding lemma to a regular unimodular triangulation,
we obtain the following corollary which is contained in
\cite[Corollaries~8.4, 8.8]{SturmfelsGBCP}.
\begin{corollary}
  \label{lem:Groebner}
  If $\T_\vomega$  is a regular unimodular triangulation $\T_\vomega$
  of $P$, then
  \[
  \lead_\vomega \ideal_P = \left\langle \ \prod_{\va\in N}x_\va \ :  N
    \text{ is a minimal non-face of  } \T_\vomega \ \right\rangle.
  \]
\end{corollary} 
This ideal is known as the Stanley-Reisner ideal of the simplicial
complex $\T_\vomega$.
The formula allows us to recover $\T_\vomega$ from $\lead_\vomega
\ideal_P$ (its faces correspond to the 
monomials not in
$\lead_\omega \ideal_P$, and vice versa). In fact, by
Lemma~\ref{lemma:normal-covering}
even a Gr\"obner basis can be read off of $\T_\vomega$.

If $\T_\vomega$ is not unimodular then the following modified formula
is still true:
\[
\operatorname{Rad} (\lead_\vomega \ideal_P) = \left\langle \ \
  \prod_{\va\in N}\x_\va \ :  N \text{ is a minimal non-face of  }
  \T_\vomega \ \right\rangle.
\]
In particular, we can still recover $\T_\vomega$ from $\lead_\vomega
\ideal_P$, but not the other way around.

\begin{theorem}
  \label{thm:Groebner}
  Given that $\configuration$ generates the lattice $\Z^{d+1}$,
  the initial ideal $\lead_\vomega \ideal_P$ is squarefree if and only
  if the regular triangulation $\T_\vomega$ of $P$ is unimodular. 
\end{theorem}

This theorem is Corollary~8.9 in~\cite{SturmfelsGBCP}; it follows
from~\cite[Thm.~5.3]{KapranovSturmfelsZelevinsky}, and
is one of the primary motivations for studying regular unimodular triangulations, from the perspective of algebraic geometry.

\begin{proof}
If $\T_\omega$ is unimodular, the previous lemma shows that
$\lead_\vomega \ideal_P$ is squarefree.

So, assume $\T_\vomega$ is not unimodular, and let $F = \conv ( \va_0,
\ldots, \va_d ) \in \T_\vomega$ be a simplex of determinant $D > 1$.
Let  $\Lambda$ denote the (strict) sublattice of $\Z^{d+1}$ generated by
the vertices of $F$. Observe that $D \vm \in \Lambda$ for all $\vm
\in \Z^{d+1}$.

We will construct a vector $\vb \in \cone F \cap \Z^{d+1} \setminus
\Lambda$ which is a non-negative integral linear combination of
$\configuration$. First, choose $\vb' \in \Z^{d+1} \setminus
\Lambda$. By assumption, $\vb'$ is an integral linear combination of
$\configuration$. Adding a sufficiently large multiple of $D \sum_{\va \in
  \configuration} \va$ will make the coefficients non-negative.
Then adding a sufficiently multiple of $\sum_{\va \in F} \va$ will yield  a point
in $\cone F$.

Among all $\vn \in \Z^\configuration_{\ge 0}$ satisfying $\sum_{\va \in
  \configuration} n_\va \va = \vb$ choose the one with minimal
$\vomega$-weight. Since $\vb \not\in \Lambda$, $\vx^\vn$ is not
supported on $F$. Still, $\vx^\vn$ is never a leading term: $\vx^\vn
\not\in \lead_\vomega \ideal_P$. 

Yet, $D \vb \in \cone F \cap \Lambda$, so $D \vb = \sum_{i=0}^d
m_i \va_i$ for some $\vm$, and $\vx^{D\n} - \vx^\vm \in \ideal_P$. As
$\vx^\vm$ is supported on the face $F$, it cannot be the leading term
and $(\vx^\vn)^D \in \lead_\vomega \ideal_P$. So $\lead_\vomega
\ideal_P$ is not squarefree.
\end{proof}

Theorem~\ref{thm:Groebner} provides a method for constructing regular
unimodular triangulations. Conversely, all regular unimodular
triangulations constructed in the present article yield Gr\"obner
bases of the corresponding toric ideals. Both directions of the
theorem have been used -- compare  sections
~\ref{sec:flow-polytopes} and \ref{sec:edge-polytopes}.

\begin{example}
  Consider the the twisted cubic curve. Let $P=[1,4]$ be a
  1-dimensional polytope whose lattice point set and toric ideal are
  $\configuration=\{(1,1), (2,1), (3,1), (4,1)\}$ and
  $\ideal_P=\left\langle x_1x_3-x_2^2, x_2x_4-x_3^2, x_1x_4 -
    x_2x_3\right\rangle$.  Let $\vomega=(\omega_1, \omega_2, \omega_3,
  \omega_4)$.  There are eight monomial initial ideals and four
  triangulations, depending on the values of
  $\lambda_1:=\omega_1-2\omega_2+\omega_3$ and
  $\lambda_4:=\omega_4-2\omega_3+\omega_2$:
  \[
  \begin{array}{ccc|c|c}
    & \vomega & &  \lead_\vomega \ideal_P  &  \T_\vomega \cr 
    \hline
    \lambda_1 > 0 &\text{ and }& \lambda_4 > 0 &
    \langle x_1x_3, x_2x_4,x_1x_4 \rangle &
    [1,2],[2,3],[3,4]  
    \cr
    \lambda_1<0 &\text{ and }& 2\lambda_1 + \lambda_4 > 0 & 
    \langle x_1x_4, x_2^2, x_2x_4 \rangle &
    [1,3],[3,4] 
    \cr
    \lambda_1 + 2\lambda_4 > 0 &\text{ and }& 2\lambda_1 + \lambda_4 < 0
    &\langle x_2^2, x_2x_3, x_2x_4, x_1x_4^2 \rangle &
    [1,3],[3,4] 
    \cr
    \lambda_1 + 2\lambda_4 < 0 &\text{ and }& \lambda_4 > 0 &
    \langle x_2^2, x_2x_3, x_2x_4, x_3^3 \rangle &  
    [1,4]
    \cr
    \lambda_1 < 0 &\text{ and }& \lambda_4 < 0 &
    \langle x_2^2, x_2x_3, x_3^2 \rangle &
    [1,4]  
    \cr
    \lambda_1 > 0 &\text{ and }& 2\lambda_1 + \lambda_4 < 0 & 
    \langle x_1x_3, x_2x_3, x_2^3, x_3^2 \rangle &
    [1,4] 
    \cr
    \lambda_1 + 2\lambda_4 < 0 &\text{ and }& 2\lambda_1 + \lambda_4 > 0
    &\langle x_1x_3, x_2x_3, x_1^2x_4, x_3^2 \rangle &
    [1,2],[2,4] 
    \cr
    \lambda_1 + 2\lambda_4 > 0 &\text{ and }& \lambda_4 < 0 &\langle
    x_1x_3, x_1x_4, x_3^2 \rangle &
    [1,2],[2,4]
  \end{array}
  \]
  Observe how each triangulation corresponds to the (radical of the)
  initial ideal. The converse works precisely in the first case where
  the initial ideal is squarefree and the triangulation is unimodular.
\end{example}

\subsubsection{Quadratic triangulations} \label{sec:quadratic}
If a polytope has a regular unimodular triangulation, we have seen
that the size of the minimal non-faces controls the degree of the
corresponding Gr\"obner basis.
Of particular interest is the case of degree two -- quadratic
triangulations -- especially given the connection to Koszul
algebras. Recall that an algebra $R$ over a field $\mathbbm{k}$ is
Koszul if $\mathbbm{k}$ has a linear free resolution as an
$R$--module.
In this context, the second hierarchy of properties in
section~\ref{sec:hierarchy} is expressed in the following proposition
(see~\cite[Cor.~2.1.3]{BGT}).
\begin{proposition}
  $\ideal_A$ has a quadratic initial ideal $\Rightarrow$
  $\kk[\x]/\ideal_A$ is Koszul $\Rightarrow$ $\ideal_A$ is generated
  by quadratic binomials.
\end{proposition}
For polygons we have the following nice characterization by Bruns,
Gubeladze, and Trung.
\begin{proposition}[{\cite[Cor.~3.2.5]{BGT}}] \label{prop:polygonKoszul}
  A lattice polygon with at least four boundary lattice points has a
  quadratic triangulation. 
\end{proposition}
A non-unimodular polygon with exactly three boundary lattice points
cannot have a quadratic triangulation because one cannot get rid of
the cubic generator of $\ideal_A$ coming from the product of the
corresponding three variables.

Lemma~\ref{lem:Groebner} also implies the following correspondence
between quadratic triangulations and quadratic Gr\"obner bases
referenced in Section~\ref{sec:Commutative algebra}. 

\begin{theorem} \label{thm:GBtriang} If $P$ has a quadratic
  triangulation $\T$, then the defining ideal $\ideal_P$ of the
  projective toric variety $X_P \subset \mathbb{P}^{r-1}$ has a
  quadratic Gr\"obner basis.  In this case, $\lead (\ideal_P) =
  \langle x_\va x_\vb \ \vert \ \va\vb \text{ is not an edge in } \T
  \rangle$. In particular, $R_P$ is Koszul.
\end{theorem}

See~\cite{BGTproblems} for a collection of unsolved problems in the
field.

\section{Examples}  \label{sec:Examples}

Here we present what is known (and unknown) for some particular families of polytopes.
Most of them are connected to one of the classical crystallographic root systems. 
We  will examine two  distinct ways  of associating  polytopes  to  a root
system   $\Gamma$.    In  Section~\ref{sec:cut-by-roots} we
consider polytopes with facet normals in $\Gamma$ (\emph{polytopes of
  type  $\Gamma$}),  and in  Section~\ref{sec:spanned-by-roots}   polytopes  with vertices  in
$\Gamma$ (\emph{$\Gamma$-root polytopes}).
Then, in Section~\ref{sec:graph-polytopes} we look at polytopes defined
from graphs, including flow polytopes of directed graphs
and characteristic polytopes of undirected ones. Finally,
Section~\ref{sec:Smooth} examines smooth polytopes.

A
(real, finite) \emph{root system} is a family of vectors $\Gamma$ that is
invariant under reflection with respect to the
hyperplanes orthogonal to each of the elements in $\Gamma$.
It is \emph{crystallographic} if $\Gamma$ generates a lattice, which we denote $\Lambda_\Gamma$.
Each system studied here satisfies that property, so when we refer to root systems, 
they are assumed to be crystallographic.

The direct sum
\[
\Gamma\oplus\Gamma':=\Gamma\times \{0\} \cup \{0\}\times \Gamma'
\]
of two root systems is a root system. Root systems that cannot be
decomposed in this fashion are called \emph{irreducible} and are classified as follows. 
(Here we are only considering crystallographic ones).

There are the four infinite families, that exist in all dimensions:
\begin{align*}
  \rootA_{n-1}&:=\{\ve_i-\ve_j\mid 1\le i, j \le n\}\\
  \rootB_n&:=\{\pm \ve_i\mid 1\le i\le n\}\cup \{\pm\ve_i\pm \ve_j\mid
  1\le i< j \le n\}\\
  \rootC_n&:=\{\pm 2\ve_i\mid 1\le i\le n\}\cup \{\pm\ve_i\pm \ve_j\mid
  1\le i< j \le n\}\\
  \rootD_n&:=\{\pm\ve_i\pm \ve_j\mid  1\le i< j \le n\},\\
  \end{align*}
where, as usual,  $\ve_i$ is the  $i$-th standard  unit vector  in $\R^n$. 

In addition to  these, there are five
other root systems:
\begin{itemize}
\item The system $\rootF$ is spanned by the roots in $\rootB_4$
  together with all vectors in
  $\{(\pm\nicefrac12,\pm\nicefrac12,\pm\nicefrac12,\pm\nicefrac12)\}$. 
\item The
  root system $\rootE_8$ is the union of $\rootD_8$ and
  \begin{align*}
  \{\vx\mid x_i=\pm\nicefrac12\text{ for } 1\le i\le 8\text{ and }
  \textstyle\sum x_i\text{ even}\}\,.
  \end{align*}
\item 
  The root system $\rootE_7$ is the subset of $\rootE_8$ of all
  vectors whose entries sum up to zero. 
\item The root system $\rootE_6$ is
  the set of roots of $\rootE_8$ that are spanned by the vector
  \begin{align*}
    \begin{array}{rrrrrrrr}            (\nicefrac12,&-\nicefrac12,&-\nicefrac12,&-\nicefrac12,&-\nicefrac12,&-\nicefrac12,&-\nicefrac12,&\nicefrac12)
    \end{array}
  \end{align*}
  together with $\ve_1+\ve_2$ and $-\ve_i+\ve_{i+1}$ for
  $1\le i\le 4$. 
\item Finally, the root system $\rootG$ is the set of
  vectors
  \begin{align*}
    \begin{array}[t]{rrrrrrrrr}
      \pm(\,1&-1&0\,),&
                        \pm(\,1&0&-1\,),&
                                          \pm(\,0&1&-1\,),\\
      \pm(\,2&-1&-1\,),&
                         \pm(\,1&-2&1\,),&
                                           \pm(\,1&1&-2\,)\,.
    \end{array}
  \end{align*}
\end{itemize}

For  each  root   system  $\Gamma\in\{\rootA_{n-1},  \rootB_n,  \rootC_n,
\rootD_n,  \rootF\}$  we  define  the special choice
of \emph{positive} roots in the root system as the
 subset $\Gamma^+ \subset \Gamma$   consisting of
all  roots whose  first  nonzero  entry  is positive,  and
$\tilde\Gamma$ and $\tilde\Gamma^+$ are defined as $\Gamma\cup\{\0\}$
and $\Gamma^+\cup\{\0\}$, respectively.

\subsection{Polytopes cut out by roots}
\label{sec:cut-by-roots}

Given a root system $\Gamma$, root hyperplanes are hyperplanes of
the form $H_{\vv,b}:=\{\vv \vx \le w\}$ where $\vv\in\Gamma$ and
$w\in\Z$. If $\Gamma$ is a crystallographic root system with lattice
$\Lambda_\Gamma$ then the root hyperplanes are lattice hyperplanes
for the dual lattice
\[
\Lambda_\Gamma^*:=\{\vx : \vv\vx\in \Z, \forall \vv\in\Gamma\}.
\]

A polytope is \emph{of type $\Gamma$} if it is a lattice polytope
(for the lattice $\Lambda_\Gamma^*$) and all its facet-defining
hyperplanes are root hyperplanes.  Note that polytopes cut out by a non-irreducible root
system are products of their irreducible components.

The following is a general result about polytopes cut out by roots
(cf.~\cite[p.~90]{Humphreys}).
\begin{lemma}
\label{lemma:roothyperplanes}
 Given an irreducible root system $\Gamma$ spanning $\R^d$,
 every cell in the hyperplane arrangement given by the
  (infinite) family of hyperplanes of the form $\{\vv x =
z\}$ for $\vv\in \Gamma$ and $z\in \Z$ is a simplex.
\end{lemma}

That is, the arrangement of all root hyperplanes is an (infinite,
periodic) triangulation of $\R^d$. Its cells are called
\emph{alcoves}. 
It may, however, have vertices which do not belong to
$\Lambda_\Gamma^*$. For example, for the root system $\rootB_2$,
$\Lambda=\Lambda^*=\Z^2$, but the arrangement of root hyperplanes
contains vertices in $(1/2,1/2)+\Z^2$.

Payne proved another important general result about polytopes cut out by roots
 using Frobenius splittings:

\begin{theorem}~\cite[Thm.~1.1]{PayneFrobenius}
\label{thm:rootfacets_integralKoszul}
Every type $\Gamma$
polytope for each of the classical root systems $\rootA_{n-1}$,
$\rootB_n$, $\rootC_n$ and $\rootD_n$ is integrally closed and
Koszul. 
\end{theorem}

\subsubsection{Type $\rootA$ Polytopes} \label{sec:type-A-facets} The
 type $\rootA$ root system is very special in that the matrix whose
columns are its roots is totally unimodular. In particular, all
polytopes of type $\rootA$ are totally unimodular. Therefore, they have
quadratic triangulations which can be realized via
 the construction in Theorem~\ref{thm:hyperplanes}. The following statement
(without the proof of flagness) appeared already in~\cite[Lemma 2.4]{KKMS3}.

\begin{theorem}
\label{thm:rootA}
Let $P$ be a polytope of type $\rootA$. The lattice dicing subdivision
$\T$ obtained from slicing $P$ by all the lattice hyperplanes with
normal in $\rootA_n$ is a quadratic triangulation of $P$.
\end{theorem}

\begin{proof}
Theorem~\ref{thm:hyperplanes} established that $\T$ is a regular lattice
subdivision in which all faces are compressed, and
Lemma~\ref{lemma:roothyperplanes} shows that this is in fact a
triangulation. So, it only remains to be shown that it is flag. 

For this, suppose $N$ is a minimal non-face with more than two
elements. Since $N$ is not a face, the relative interior of $\conv(N)$
is cut by some hyperplane $H$ spanned by a face of $\T$, which  must
be a root hyperplane.  In particular, $N$ contains points $\vn^+$ and
$\vn^-$ on both sides of that hyperplane, which contradicts the fact
that $\vn^+\vn^-$ is an edge.
\end{proof}

We now discuss two particularly interesting cases of type~$\rootA$
polytopes, order polytopes and hypersimplices.

Let $(X,\preccurlyeq)$ be a partial order on $X=\{1, \ldots, n\}$.
A vector $\vv\in[0,1]^n$ is said to \emph{respect the order} if
$v_i\le v_j$ whenever $i\preccurlyeq j$. A \emph{linear extension}
of $\preccurlyeq$ is a total order on $X$ that refines the partial
order. The \emph{order polytope} associated $(X,\preccurlyeq)$ 
is the polytope 
\begin{align*}
  O(\preccurlyeq):=\{\vx\in[0,1]^n \mid \vx \text{ respects the
    order}\}.
\end{align*}

Vertices of $O(\preccurlyeq)$ are the characteristic vectors of \emph{up-close subsets} or 
\emph{filters} of the poset (that is, sets $S\subset X$ with the property that $i\in S$ and $i \preccurlyeq j$ implies $j\in S$).
Facets of $O(\preccurlyeq)$ are in bijection with \emph{covering relations}, \emph{minimal elements} and \emph{maximal elements}. If $i \preccurlyeq j$ is a covering relation (meaning that  there is no $k$ with $i \preccurlyeq k \preccurlyeq j$) then $x_i\le x_j$ defines a facet, and if $i$ is a minimal (resp. maximal) element then
$x_i\ge 0$ (resp. $x_i\le 1$) defines a facet. All facets are of one of these forms.
In particular, the facet vectors of an order polytope are contained in the set of vectors $\{e_i: i\in [n]\} \cup \{ e_i -e_j : i,j\in [n]\}$, which is mapped to $\rootA_n$ by the linear isomorphism $e_i \mapsto e_i-e_{n+1}$. This shows that $O(\preccurlyeq)$ is of type $\rootA$, and also that it is compressed.

The quadratic triangulation of an order polytope guaranteed by Theorem~\ref{thm:rootA} was first studied by Stanley~\cite{StanleyOrderPolytope}.
Its maximal simplices are in bijection to the linear extensions of the partial order. In particular, the normalized volume of $O(\preccurlyeq)$
equals the number of distinct linear extensions of $\preccurlyeq$. 

The $d$-dimensional  hypersimplex $\HS(d,k)$ for  $1\le k\le d$
is defined as
\begin{align*}
  \HS(d,k):=\{\vx\in\R^d\mid  k-1\le \sum_i  x_i\le k,\,  \0\le \vx\le
  \1\}\,.
\end{align*}
Alternatively, each $\HS(d,k)$  can also be realized as the intersection of the $(d+1)$ dimensional
$0/1$-cube with the hyperplane $\{\vx\mid \sum_ix_i=k\}$.  The
hypersimplices for $k=1$ and $k=d$ are simplices.  The hypersimplex
$\HS(3,2)$ is the octahedron.  To see that these polytopes are root
system polytopes of type $\rootA_{d-1}$,  apply the unimodular
transformation given by $y_j:=\sum_{i=1}^jx_i$.  The facet
inequalities are then given by $y_j-y_{j-1}\ge 0$ and $k-1\le y_d\le k$.
The fact that these hypersimplices are also compressed  follows directly from
condition (\ref{thm:item:third}) of Theorem \ref{thm:paco}.

\subsubsection{Type $\rootB$ Polytopes} \label{sec:type-B-facets}

The root system $\B_n \subset \R^n$ has two types of roots. Their distinction appears
in our proofs of Proposition \ref{prop:type-B} and Lemma \ref{lemma:prop:type-B}.
 The ``short roots'' are given by the
vectors $\pm\ve_i$, and the ``long roots'' are given by the vectors
$\pm\ve_i\pm\ve_j$ for $1 \le i,j \le n$.

\begin{proposition} \label{prop:type-B} If $P$ is a lattice
  polytope with facet normals in $\rootB_n$, then $P$ admits a regular
  unimodular triangulation.
\end{proposition}

The proof
 requires the following lemma.
 
\begin{lemma} \label{lemma:prop:type-B}
  Let $M$ be a matrix with rows in $\rootB_n$, $N$ be a matrix with rows in $\{e_i : i\in [n]\}$ (the set of short roots), 
  and $\vw\in \Z^k$, where $k$ is the number of rows in $N$.
  If the system
  \[
  M\vx = 0, \qquad N\vx = \vw
  \] 
  has a unique solution, then this solution is integral.
\end{lemma}

\begin{proof}
First, observe that there is no loss of generality in assuming that all rows of $M$ are long roots (as all rows that are short roots can be put in $N$).
Also, if a long root $\e_i \pm \e_j$ in $M$ shares a coordinate with a short root $\e_i$ in $N$, then the other coordinate of that long root is fixed
to the value $x_j=\pm w_i$, so  removing  that row from $M$ and putting a new row in $N$, yields equivalent system. 

After this process has been applied as many times as possible, $M$ and $N$ operate in disjoint sets of coordinates. At which point, setting all coordinates in $N$ to their value $w_i$ and all coordinates not in $N$ to zero gives an integral solution of the system (unless $N$ is inconsistent, in which case the original system had no solution).
\end{proof}

\begin{proof}[Proof of Proposition~\ref{prop:type-B}]
  Slice $P$ by all short roots. This gives a regular
  subdivision of $P$. The proposition follows from showing that the cells in this subdivision
  have integral vertices and are compressed.
  
  Let $\va$ be a vertex of this subdivision and let $\vb$ be a vertex of the 
  carrier face of $\va$ in $P$. Then, Lemma~\ref{lemma:prop:type-B}
  implies that $\va- \vb$ is integral, since it is 
  determined by setting some coordinates to a fixed value and 
  homogenized versions of some of the
  facet-defining inequalities of $P$. Thus $\va=\vb+(\va-\vb)$ is also integral.

  It remains to show that all cells in the subdivision are compressed.
  By construction, the cells have width one in the direction of the
  short roots. In the direction of a long root $\pm\ve_i\pm\ve_j$
  consider the projection onto the $x_i$-$x_j$-plane. In each case the
  cell has width one.
\end{proof}

\begin{example}
The quadratic triangulations obtained for type $\rootA$ polytopes in Theorem~\ref{thm:rootA}
are in fact ``type $\rootA$ triangulations'', in the sense that they consist of type $\rootA$ simplices. The same does not in general occur in type $\rootB$.

For example, consider  the following type $\rootB$ polytope:
\[
 P = \conv \left[ 
  \begin{smallmatrix}
    0&1&0&1&1&1 \\
    0&0&1&1&1&0 \\
    0&0&0&0&1&-1
  \end{smallmatrix}
\right]
=
\left\{ \vx \in [0,1]^3 \ : \ z \le x,y; x+z \ge 0; y-z \ge 1
\right\}
\]
The only type $\rootB$ hyperplane cutting through its interior but not creating any non-integer vertices is $z=0$. This is the hyperplane used in the proof of Proposition~\ref{prop:type-B}, and cuts $P$ into two compressed type $\rootB$ polytopes. Each of them is a square pyramid and they both have type $\rootB$ unimodular triangulations, but their type $\rootB$ triangulations use different diagonals of the square. 
Meaning, in the last step of the proof we cannot triangulate the cells of the short-root dicing of $P$ using
transformed $\rootA$-triangulations.
\end{example}

We do not know whether Proposition~\ref{prop:type-B} can be extended
to yield quadratic triangulations of type $\rootB$ polytopes.
We do know arbitrary pulling refinements
will not work, because there are
non-flag pulling triangulations of the order polytope $[0,1]^3$.
Even pulling according to the vertex order of a linear functional 
may produce non-flag triangulations in higher dimensional order polytopes. 
(The smallest example we know of is six-dimensional: the order
polytope of the Boolean lattice on three elements with the minimum and
maximum elements removed.)

However, the following lemma, combined with Lemma~\ref{lemma:normal-covering}
does guarantee  that the toric ideals of type $\rootB$
polytopes are quadratically generated, since $\rootB$-cells
are equivalent to $\rootA$-cells (which have quadratic
triangulations).

\begin{lemma}
  If $P$ is a lattice polytope with facet normals in $\rootB_n$ and
   $Q$ is a cell in the lattice dicing subdivision of $P$ obtained
  from slicing $P$ by all the lattice hyperplanes with normals among
  the short roots in $\rootB_n$, then $Q$ is equivalent to an order
  polytope.
\end{lemma}

\begin{proof}
We can assume $Q$ is full-dimensional. Up to translation, $Q$ is given
by the intersection of $[0,1]^n$ with some constraints $x_i \lessgtr
x_j$ and some $x_i + x_j \lessgtr 1$. 
In either case, the cell must contain the point $\vb := \frac12 \1$. Since
$\vb$ cannot be a vertex, there must be a vertex
$\va \in \{0,1\}^d$ such that for small $\epsilon>0$, the points
$\vb \pm \epsilon (\va-\vb)$ both belong to our cell. Due to the
nature of its inequalities, our cell contains the long diagonal from
$\va$ to $\1-\va$. So, mapping $x_i$ to $1-x_i$ whenever $a_i=1$
identifies the cell unimodularly with an order polytope.
\end{proof}

\begin{corollary} \label{cor:rootB-quadratic}
  Toric ideals of type $\rootB$ polytopes
  are quadratically generated.
\end{corollary}

\subsubsection{Type $\rootC_n$ and $\rootD_n$ Polytopes}
\label{sec:type-C-facets}

Payne's theorem~(Theorem~\ref{thm:rootfacets_integralKoszul}) guarantees that
all polytopes of type $\rootC$ and $\rootD$ are integrally closed.
However, we do not know whether they all have unimodular triangulations
for cases other than  $\rootD_n\cong \rootA_n$ ($n\le 3$) 
and $\rootC_n\cong \rootB_n$ ($n\le 2$) .

\subsubsection{Type $\rootF$ Polytopes} \label{sec:type-F-facets}

Polytopes with facets in $\rootF$ do not in general have regular
unimodular triangulations.  For an example, consider the polytope
$P_{\rootF}$ defined by the following linear inequalities:
\begin{align*}
  x_4 + x_2&\le 0& x_4 - x_1 &\ge 0& x_4 - x_2&
  \le 0\\
  x_3 + x_1 &\ge 0& x_3 - x_1&\le 2
\end{align*}
The dual lattice $(\rootF)^*$ is the sublattice of $\Z^4$ containing
all points with even coordinate sum. Hence, the lattice points in $P$
are
\begin{align*}
  (-1,\phantom{-}0,\phantom{-}1,\phantom{-}0)&&  (-1,-1,\phantom{-}1,-1)\\
  (\phantom{-}0,\phantom{-}0,\phantom{-}0,\phantom{-}0)&&   (-1,\phantom{-}1,\phantom{-}1,-1)\\
  (\phantom{-}0,\phantom{-}0,\phantom{-}2,\phantom{-}0)
\end{align*}
which are also the vertices of $P_{\rootF}$.  In particular, since 
$P_{\rootF}$ is a non-unimodular empty simplex, it does not
have a unimodular triangulation. In fact, it is not even
integrally closed. For example, $(-1,0,2,-1)\in
(\rootF)^*\cap2P$ is not a sum of two lattice points in
$P_{(\rootF,\vc)}$.  
Therefore,  polytopes cut out by $\rootF$ are not in general
integrally closed, which answers a question left open in~\cite{PayneFrobenius}.

\subsubsection{Type $\rootE_6$, $\rootE_7$, $\rootE_8$ Polytopes} \label{sec:type-E-facets}

Here is an example of a type $\rootE_8$ polytope that does not have
a regular unimodular triangulation. $P_{\rootE_8}$ is the product
of the polytope $P_{\rootF}$ defined in
Section~\ref{sec:type-F-facets} with itself. The facet normals of this
polytope are roots in $\rootE_8$. The $25$ vertices of
$P_{\rootE_8}$ are the products of the five vertices of
$P_{\rootF}$. Additionally, we have the following $12$ lattice
points:
\begin{align*}
  \begin{array}{rrrrrrrr}
  (-1& 0& 1& -1& -1& 0& 1& -1)\\
  (-1& 0& 1& -1& 0& 0& 1& 0)\\
  (\phantom{-}0& 0& 1& 0& -1& 0& 1& -1)\\
  (\phantom{-}0& 0& 1& 0& 0& 0& 1& 0)\\
  (-\nicefrac12& -\nicefrac12& \phantom{-}\nicefrac12& -\nicefrac12& -\nicefrac12& -\nicefrac12& \phantom{-}\nicefrac12& -\nicefrac12)\\
  (-\nicefrac12& -\nicefrac12& \nicefrac12& -\nicefrac12& -\nicefrac12&  \nicefrac12& \nicefrac32& -\nicefrac12)\\
  (-\nicefrac12&  \nicefrac12& \nicefrac12& -\nicefrac12& -\nicefrac12& -\nicefrac12& \nicefrac32& -\nicefrac12)\\
  (-\nicefrac12&  \nicefrac12& \nicefrac12& -\nicefrac12& -\nicefrac12&  \nicefrac12& \nicefrac12& -\nicefrac12)\\
  (-\nicefrac12& -\nicefrac12& \nicefrac32& -\nicefrac12& -\nicefrac12& -\nicefrac12& \nicefrac32& -\nicefrac12)\\
  (-\nicefrac12& -\nicefrac12& \nicefrac32& -\nicefrac12& -\nicefrac12&  \nicefrac12& \nicefrac12& -\nicefrac12)\\
  (-\nicefrac12&  \nicefrac12& \nicefrac32& -\nicefrac12& -\nicefrac12& -\nicefrac12& \nicefrac12& -\nicefrac12)\\
  (-\nicefrac12&  \nicefrac12& \nicefrac32& -\nicefrac12& -\nicefrac12&  \nicefrac12& \nicefrac32& -\nicefrac12)
\end{array}
\end{align*}
The point $(-2,-1,2,-1,-1,0,2,-1)$ is in twice the
polytope $P_{\rootE_8}$, but it is not the sum of two lattice points
in $P_{\rootE_8}$. So $P_{\rootE_8}$ is not integrally closed, and cannot have a regular unimodular triangulation.

The $\rootE_7$ and $\rootE_6$ cases  remain open.

\subsubsection{Type $\rootG$ Polytopes} \label{sec:type-G-facets}

Every type $\rootG$ polytope $P_{\rootG}$ is a polygon. As such, they have
regular unimodular triangulations. By
Proposition~\ref{prop:polygonKoszul} they even have quadratic
triangulations whenever there are at least four boundary lattice points. Up to lattice
equivalence, the only type $\rootG$ polygons with three boundary lattice
points are a unimodular triangle and the triangle of Figure~\ref{fig:non-quad-triangle}.
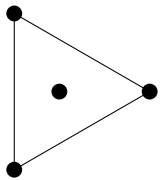
\begin{figure}[tb]
  \centering
  \begin{tikzpicture}
    \foreach \x/\y in {0/p,120/q,240/r} {
      \coordinate (\y) at (\x:1.2);%
    }
    \coordinate (a) at (barycentric cs:p=1,q=1,r=1);
    \draw (p) -- (q) -- (r) -- (p);
    \fill (a) [black] circle (3pt);%
    \fill (p) [black] circle (3pt);%
    \fill (q) [black] circle (3pt);%
    \fill (r) [black] circle (3pt);%
  \end{tikzpicture}
  \caption{A triangle without a quadratic triangulation}
  \label{fig:non-quad-triangle}
\end{figure}

\subsection{Polytopes spanned by roots}
\label{sec:spanned-by-roots}

Now we consider polytopes defined as the convex hull of a subset of
one of the root systems.  Polytopes of this type are widely studied,
see e.g.\ Gelfand, Graev, and Postnikov~\cite{0876.33011}, Ohsugi and
Hibi \cite{1001.05125, 1012.13012}, Meszaros
\cite{Meszaros2009,Meszaros2009a}, and Cho \cite{0939.05085}.  Ohsugi
and Hibi have found regular unimodular triangulations for several
classes of these polytopes.

Before presenting results, we introduce some notation.  Let $\configuration$ be a set of vectors
in $\R^n$ and $\widetilde\configuration:=\configuration\cup\{\0\}$.
The polytope associated to this configuration is
$P_{\configuration}:=\conv(\configuration)$.

\subsubsection{Sub-configurations of $\rootA_n$: arc polytopes.}
\label{subsec:anarcpolytope}

Arc polytopes are the easiest type of a root polytope to find
unimodular triangulations for. Since
$\rootA_n$ is totally unimodular as a vector configuration,
for every $\configuration\subseteq \rootA_n$ each
simplex containing the origin as a vertex is unimodular. As a result:

\begin{theorem}\label{thm:An-1}The following polytopes have regular
  unimodular triangulations.
  \begin{compactenum}
  \item\label{thm:An:4} \label{thm:An:5}
    $P_{\widetilde\configuration}$ for any
    $\configuration\subseteq\rootA_{n-1}$.  In particular,
    $P_{\configuration}$ for any $\configuration\subseteq\rootA_{n-1}$
    with $\0\in P_{\configuration}$.
  \item\label{thm:An:6} $P_{\configuration}$ for any
    $\configuration\subseteq\rootA_{n-1}$ with
    $\0\not\in\aff(\configuration)$.  In this case \emph{every}
    triangulation is unimodular.
  \end{compactenum}
\end{theorem}

\begin{proof}
Part (\ref{thm:An:4}) follows from
Stanley~{\cite[Ex.\ 2.4a]{StanleyDecompositions80}}. For any
pulling triangulation in which $\0$ is the first point pulled,
each simplex in the triangulation will contain $\0$.  By unimodularity
of the vertex matrix, this means every resulting simplex has volume one.

The argument for Part (\ref{thm:An:6}) is essentially the same.  $Q:=\conv(\{\0\}\cup\configuration^+)$ is a pyramid over
$P_{\configuration^+}$ with apex $\0$, so every simplex of every triangulation of it
contains $\0$, and is hence unimodular.
 \end{proof}

The only configurations $\configuration \subseteq
\rootA_{n-1}$  not addressed by Theorem~\ref{thm:An-1} are those that
have $\0$ in their affine hull, but not in their convex hull. Not all
of these have a regular unimodular triangulation.
In fact, the polytope of such a configuration can even fail to be
integrally closed.
\begin{example}  \label{eg:weirdA-non-normal}
Consider the subconfiguration of $\rootA^+_{4}$ consisting of
$\ve_1-\ve_2$, 
$\ve_2-\ve_3$,
$\ve_3-\ve_4$,
$\ve_4-\ve_5$,
$\ve_1-\ve_4$ and $\ve_3-\ve_5$.
It is 4-dimensional with six vertices, so it has a unique affine dependence. Namely:
\[
(\ve_1-\ve_2) + (\ve_2-\ve_3) + 2  (\ve_3-\ve_5) = (\ve_1-\ve_4) + (\ve_3-\ve_4) + 2(\ve_4-\ve_5).
\]
Since both sides of the dependence have some coefficient different from one, neither triangulation
of this circuit is unimodular (part (1) of Lemma~\ref{lemma:circuit}).
\end{example}
Even when there is a regular unimodular triangulation, it will not necessarily be flag.

\begin{example}  \label{eg:weirdA-non-quadratic}
Consider the subconfiguration of $\rootA^+_{5}$ consisting of
  $\ve_1-\ve_2$, $\ve_2-\ve_3$, $\ve_1-\ve_3$, $\ve_3-\ve_5$,
$\ve_3-\ve_4$, and $\ve_4-\ve_5$. 
This is again a circuit and its only affine dependence is 
\[
(\ve_1-\ve_2)+(\ve_2-\ve_3)+(\ve_3-\ve_5)=
(\ve_1-\ve_3)+(\ve_3-\ve_4)+(\ve_4-\ve_5).
\]
Since the circuit has three points on each side, none of its triangulations is flag 
 (part (2) of Lemma~\ref{lemma:circuit}).
 Let us mention that this configuration is lattice equivalent to the set of vertices of the third Birkhoff polytope.
 (See Theorem~\ref{thm:quadTrans} and the discussion before it.) 
\end{example}
Notice that Theorem~\ref{thm:An-1}(\ref{thm:An:4}) implies
$P_{\rootA_{n-1}}$ has a regular unimodular triangulation.  Even more is
known for the
subconfiguration of all positive roots.

\begin{theorem}\label{thm:An}
  \begin{compactenum}
  \item\label{thm:An:2} $P_{\widetilde\rootA^+_{n-1}}$ has a quadratic
    triangulation.
  \item\label{thm:An:3} $P_{\rootA^+_{n-1}}$ has a regular unimodular
    triangulation in which all non-faces have size two or three.
  \end{compactenum}
\end{theorem}

Part (\ref{thm:An:2}) was proved by Gelfand, Graev, and
Postnikov~\cite{0876.33011}, and 
part (\ref{thm:An:3}) is due to Kitamura~\cite{1094.13043}. 
This statement is the best possible since $\rootA_{n-1}^+$   cannot have a quadratic triangulation
for $n\ge 6$, as demonstrated in the following example.

\begin{example}[Example~\ref{eg:weirdA-non-quadratic} continued]
 \label{eg:Aplus-non-quadratic}
The vectors 
$\ve_1-\ve_2$, $\ve_1-\ve_3$, $\ve_2-\ve_3$, $\ve_4-\ve_5$,
$\ve_4-\ve_6$, $\ve_5-\ve_6$   
are the vertex set of the face of  $\rootA_{5}^+$ defined by $x_1+x_2+x_3=0$.
Any flag triangulation of $\rootA_{5}^+$ would in particular give a flag triangulation
of that face, but that face does not have any flag triangulation.
\end{example}

\medskip
$\rootA$-root configurations have a natural
graph-theoretic interpretation. Let $G=(V,A)$ be a directed graph,
which we  assume to be connected in the undirected sense. Its
(directed) incidence matrix is the $(|V|\times |A|)$-matrix $D_G$ with
a $1$ at position $(v,a)$ if $v$ is the head of the directed edge
(arc) $a$, with a $-1$ if $v$ is the tail of $a$, and with a $0$
otherwise. The \emph{arc polytope} of $G$ is the convex hull of the
columns of $D_G$. By construction, the columns are roots of type
$\rootA_{n-1}$. 
Other than possibly the origin, the only 
lattice points in the arc polytope are its vertices, which correspond to arcs in $G$.

The following easy properties were noticed by Hibi and
Ohsugi~\cite{1012.13012}. 

\begin{lemma}
Let $G$ be a connected graph with $n$ vertices, and $\configuration$ be
the set of lattice points in its arc polytope.
\begin{compactenum}
\item $\configuration $ contains the origin if and only if $G$ has a
  directed cycle. 
\item $\configuration$ has dimension $d-1$ (that is, it affinely spans the hyperplane $\sum x_i=0$) if and only if 
$G$ has an \emph{unbalanced cycle}. 
(Here, an undirected cycle in a directed graph is called \emph{balanced}  if it has the same number of edges oriented in both directions.) 
\item
  $\configuration $ is totally unimodular (all its full-dimensional
  simplices, and hence all its triangulations, are unimodular) if and
  only if all unbalanced cycles contain exactly one more edge oriented
  in one direction than in the other.~\cite[Lemma~3.6]{1001.05125}
\end{compactenum}
\end{lemma}

In part (3) we mean totally unimodular with respect to the root
lattice. This is different from the convention in~\cite{1012.13012}
where unimodular is meant with respect to the lattice generated by
$\configuration$.

In part (2) observe that balanced cycles must be even, so
that the graphs having only balanced cycles must be bipartite.  An
important subclass are those
where all edges are directed from one part to the other.  In this case
the arc
polytopes are subpolytopes of the product of two simplices, so all
their triangulations are unimodular.  (This follows also from Theorem
\ref{thm:An-1}(\ref{thm:An:6}). See also~\cite{1162.52007}.)

\begin{proof}
We first prove part (1).
  If $G$ has a directed cycle then the sum of the corresponding
  columns of $D_G$ is zero. Conversely, if $G$ does not have a directed cycle
  (that is, it is \emph{acyclic}) then its vertices can be ordered so
  that every directed edge is directed towards the greater vertex.
  So, the following strict linear inequality is satisfied in
  $\configuration$:  $\sum i x_i >0$.

  For parts (2) and (3) we give a full description of the subsets of
  $\configuration$ that span $n-1$ dimensional simplices. The
  corresponding subgraphs are necessarily spanning (all
  coordinates need to be used) and connected (otherwise a proper subset of coordinates has sum equal to zero)
  and have $n$ elements.  This makes each
  the union of a spanning tree with an extra edge.  The converse is
  almost true: If a subgraph is a spanning tree plus one edge, then
  its corresponding subconfiguration is affinely independent (hence it
  spans an $(n-1)$-dimensional configuration) unless the cycle
  contained in the subgraph is balanced. So part (2) is clear. 

   If an unbalanced cycle does
  not exist, then there can not be an $(n-1)$-dimensional simplex. If
  there is an unbalanced cycle, it yields an $(n-1)$-dimensional simplex.
  Part (3) follows from the fact that the  volume (with respect to the $\rootA$ lattice) of any such
  simplex is the difference in number of edges in each direction in
  its cycle.
\end{proof}

Hibi and Ohsugi also proved that if all induced cycles (cycles without
a chord) satisfy part (3) (they are balanced or have one
more edge in one direction) then $P_\configuration$ has a unimodular
cover.

\begin{corollary}
\begin{compactenum}
\item If $G$ is not acyclic then its arc polytope has regular
  unimodular triangulations.
\item If $G$ is bipartite with parts $X$ and $Y$, and the
  edges are all directed from $X$ to $Y$ then all triangulations
  are unimodular.
\end{compactenum}
\end{corollary}

\begin{proof}
For (1), a regular unimodular triangulation of $P_\configuration$ can be constructed
by pulling the origin first and then the remaining points in any
order. In such a triangulation, the origin is a vertex of every cell,
so their lattice volumes are the determinants of the other
vertices. Since the vertices are roots of $\rootA_{n-1}$, these will
all be one.

In part (2), $\0$ is not  in the affine hull of $P_\configuration$. So letting $X$ and $Y$ denote the two parts in $G$,
we have $\configuration \subset \{\sum_{v\in X} x_a=-1\} \cap
\{\sum_{v\in Y} x_a=1\}$, and theorem~\ref{thm:An-1}(\ref{thm:An:6})
implies the statement.
\end{proof}

\subsubsection{Sub-configurations of\/ $\rootB_n$, $\rootC_n$, $\rootD_n$ and $\rootF$}
Much less is  known  about  sub-configurations  of other  root
systems. 

\begin{theorem}
 \label{thm:rootBC}
 (Ohsugi and Hibi, ~\cite{1012.13012} for part (1),~\cite{1001.05125} for part (2))
  \begin{compactenum}
  \item    $P_{\widetilde\rootB_n^+},        P_{\widetilde\rootC_n^+}$     and
    $P_{\widetilde\rootD_n^+}$ have  \emph{quadratic} triangulations.
  \item If  $n\ge 2$, and  $\configuration^+$  is a set of vectors
    satisfying the following conditions
    \begin{enumerate}
    \item $\{\ve_i+\ve_j\mid 1\le i < j \le n\} \subseteq
      \configuration^+ \subseteq \rootB_n^+\cup\rootC_n^+$.
    \item $\configuration^+\cap \rootA_{n-1}^+$ is transitively
      closed (that is, for  all  $1\le i<j<k\le n$ with  $\ve_i-\ve_j,
      \ve_j-\ve_k\in \configuration^+$ also $\ve_i-\ve_k\in
      \configuration^+$).
    \item either all $\ve_i\in\configuration^+$ or no $\ve_i$ is in
      $\configuration^+$,
    \end{enumerate}
    then $P_{\widetilde\configuration^+}$ has a regular unimodular
    triangulation.
  \end{compactenum}
\end{theorem}

Producing triangulations   for sub-configurations of  (unions of) root
systems is much more  difficult if $\0$ is  not a vertex. For $n\ge 6$
the toric ideal corresponding to $P_{\configuration}$ for
$\configuration=\rootA^+_{n-1}$, $\rootB_n^+$, $\rootC_n^+$,
$\rootD_n^+$ has no \emph{quadratic} initial ideal, since it
contains Example~\ref{eg:Aplus-non-quadratic} as a face.
So no regular unimodular triangulation can be \emph{flag}. However,
the ideal can still have a \emph{square-free} initial ideal.

For $\rootF$ we know that some convex hulls of subsets of $\rootF$ are not integrally 
closed. For example, consider the polytope given by the convex hull
$P$ of the unit vectors together with $\ve_1+\ve_2$. The point
$\frac12 (1,1,1,1)$ is contained in $2P$, but is not a sum of two
lattice points in $P$.

\subsubsection{Edge Polytopes of Undirected Graphs}
\label{sec:edge-polytopes}

Sub-configurations  of $\{\ve_i +  \ve_j\mid  1\le i\le j \le n\} \subset
\rootC_n^+$ have attracted some attention. They can be interpreted
as  \emph{edge polytopes} of  undirected graphs, perhaps with loops but without
multiple edges.   For this,  let
$G=(V,E)$ be a finite graph on $n=|V|$  vertices and $m=|E|$ edges. The \emph{incidence matrix}   $D_G=(d_{ve})$   of $G$  is  the
$(n\times m)$ matrix with  entries in $\{0,1,2\}$, where $d_{ve}=1$ if
$e$ is  incident to $v$, but not  a loop, and  $d_{ve}=2$ if $e$  is a
loop at $v$, and all other entries of $D_G$ are zero.

Letting $\configuration$ be the set of columns of $D_G$
the  \emph{edge polytope} $P_G$ of $G$ is defined to be the convex  hull of
$\configuration$. Different graphs
may define the same polytope.   Namely, if there are loops attached to
two vertices $i$  and $j$  of $G$,  then adding or removing the edge
between $i$ and $j$ does not change $P_G$.  To avoid this ambiguity we
define  the graph  $\widetilde G$  to be  obtained from  $G$ by
adding  all edges  between vertices  incident  to a  loop. 
Then the only lattice points in $P_G$  correspond to
edges of $\widetilde G$.

Observe that Theorem~\ref{thm:rootBC}  implies that
the edge polytope $P_{K_n}$ of the complete  graph $K_n$ has a regular
unimodular triangulation. 

\begin{lemma} \label{lemma:edge-poly}
Let $G$ be connected graph on $n$ vertices.
\begin{compactenum}
\item $P_G$ is contained in the hyperplane $\sum x_i=2$, and 
it affinely spans that hyperplane if and only if $G$ has an odd cycle (that
  is, if it is not bipartite).
  \label{lemma:bipartite-edge-poly:non-bipartite}
\item If $G$ is bipartite then $P_G$ is totally unimodular.
  \label{lemma:bipartite-edge-poly:bipartite}
\item If $G$ is not bipartite then a subgraph $N$ with $c$ connected
  components corresponds to the vertices of an $(n-1)$-dimensional
  simplex in $P_G$ if and only if $N$ spans all vertices and it has a
  unique cycle in each component, each of which is an odd cycle.
  In this case, the (lattice) volume of $N$ is $2^{c-1}$.
  \label{lemma:bipartite-edge-poly:simplex}
\end{compactenum}
\end{lemma}

\begin{proof}
Part (2) follows from the fact that if $G$ is bipartite with parts $X$ and $Y$ then orienting all
edges from $X$ to $Y$ makes $P_G$ (modulo a sign change on the coordinates corresponding to $Y$)
the arc polytope of a directed bipartite graph satisfying the conditions of \ref{thm:An-1}(\ref{thm:An:6}).
Also, in this case $P_A$ lies in the codimension-two affine subspace $\sum_{i\in X} x_i = \sum_{i\in Y} x_i = 1$,
which is part of statement (1). For the rest of part (1), 
suppose that $G$ is not bipartite. Then, it has an odd cycle $C$, and this cycle can be extended to a spanning
subgraph $H$ containing no other cycle. $H$ has $n$ vertices and $n$ edges, and we only need to check that the 
determinant of the corresponding matrix is non-zero. In this determinant the rows and columns of vertices and edges 
that are not in the cycle can be neglected, and the determinant of the odd cycle itself is positive or negative one.

For part (3) note that,
as in the directed case, the subgraph $N$, corresponding to a full
dimensional simplex, needs to use all coordinates. However, it does
not need to be connected. It cannot contain even cycles, since an even cycle produces an affine dependence (the alternating sum of the points corresponding to the edges in the cycle is zero). But it can contain odd cycles, as we saw in the proof of part (1). However, no  two odd cycles can be in the same connected component. If
two cycles have more than one vertex in common then there is an even
cycle; If they have one or no vertices in common then, joining them by
a path, if needed, yields an even \emph{walk} in $N$, which also
induces an affine dependence. So, we cannot have more than one cycle
per component which means we cannot have more then $n$ edges in
total. Since we need $n$ edges as the vertices of an $(n-1)$ simplex, we 
need all components to contain exactly one cycle, which proves the ``only if'' direction of (3). 

For the converse, assume $N$ is spanning and has a unique cycle in each component. Since $N$ lies in an affine hyperplane not containing the origin, the (lattice) volume of its convex hull is the determinant of the corresponding matrix, divided by the (lattice) distance from that hyperplane to the origin. The latter is two. To show that the former is $2^c$ observe that the matrix has a block for each connected component, and the determinant of each block is positive or negative two, by the same argument as in part (1).
\end{proof}

The following result characterizes the existence
of a unimodular cover of the edge polytope of a graph. 

\begin{theorem}[Ohsugi \& Hibi '07 \cite{Ohsugi1998,Ohsugi1999}, Simis,
  Vasconcelos, and Villarreal~\cite{SVV98}]
  \label{thm:OH:graphpolytopes} 
  Given a connected graph $G$, the polytope $P_G$ has a unimodular cover if and only if no induced subgraph 
  consists of two disjoint odd cycles.
  \end{theorem}
This   characterization   was   originally   conjectured   by   Simis,
Vasconcelos, and  Villarreal~\cite{SVV94}. The condition  of the graph
$G$ in the theorem is sometimes called the \emph{odd cycle condition} 
(and usually stated in a different but equivalent way).
In particular, this theorem implies that $P_G$ is  integrally closed if
and only if $\widetilde G$ satisfies the odd cycle condition and that no simple
graph violating  the condition can have  a regular unimodular  triangulation.

\begin{proof}
Every induced subgraph of $G$ corresponds to the points in a particular face of $G$ (the converse is not true).
If an induced subgraph $N$ consists of two disjoint odd cycles, then the corresponding face is a non-unimodular
simplex, so $P_G$ cannot have a unimodular cover.

Conversely, assuming the odd cycle condition, every disjoint pair of odd
cycles in $G$ can be connected to one another by a single edge, otherwise the minimal (induced) cycles contained in them violates the odd cycle condition.

To show that $G_P$ has a unimodular cover
it suffices to show that every non-unimodular simplex can be covered by simplices of smaller volume.
For this,  we show that if $N$ is (the subgraph of $G$ corresponding to) a non-unimodular simplex
then we can add a single edge $e$ to $N$ so that all the simplices contained in $N\cup\{e\}$ have smaller volume than
$N$.

So, let $N$ be the subgraph corresponding to a non-unimodular simplex of full dimension. By the previous lemma,
$N$ has at least two connected components, each with an odd cycle. Let $N_1$ and $N_2$ be these cycles, and let $e$
be an edge of $G$ connecting them, which exists by the odd cycle condition. Then, $N\cup\{e\}$ has a unique affine dependence and it has, by Lemma~\ref{lemma:circuit-triangs}, two triangulations, one using $N$ as a simplex and the other not. In the affine dependence, $e$ is the only point having coefficient positive or negative two. All others have coefficients of positive or negative one. This implies that $N$ has twice the volume of every other simplex in $N \cup\{e\}$.
\end{proof}

\begin{figure}[t]
  \centering
  \begin{tikzpicture}
    \foreach \x/\y in {0/p,72/q,144/r,216/s,288/t} {
      \coordinate (\y) at (\x:2);%
    }
    \foreach \x/\y in {180/a,252/b,324/c,36/d,108/e} {
      \coordinate (\y) at (\x:2.45);%
    }
    \foreach \x in {a,b,c,d,e,p,q,r,s,t} {
      \fill (\x) [black] circle (3pt);%
    }
    \draw (p) -- (c) -- (b) -- (s) -- (t) -- (p);
    \draw (q) -- (e) -- (r);
    \draw [line width=2pt] (p) -- (q) -- (d);
    \draw [line width=2pt] (a) -- (r) -- (s);
    \draw [line width=2pt, dashed] (p) -- (d);
    \draw [line width=2pt, dashed] (a) -- (s);
    \draw [line width=2pt, dashed] (q) -- (r);
    \node[right=5pt] at (p) {$v_0$};
    \node[left=5pt] at (e) {$v_3$};
    \node[above=5pt] at (q) {$v_2$};
    \node[right=5pt] at (d) {$v_1$};
    \node[left=5pt] at (r) {$v_4$};
    \node[right=5pt] at (c) {$v_9$};
    \node[left=5pt] at (s) {$v_6$};
    \node[left=5pt] at (b) {$v_7$};
    \node[right=3pt,below=3pt] at (t) {$v_8$};
    \node[left=5pt] at (a) {$v_5$};
    \node[left=3pt] at (barycentric cs:p=.5,q=.5) {$p_{02}$};
    \node[right=3pt] at (barycentric cs:p=.5,d=.5) {$p_{01}$};
    \node[above=3pt] at (barycentric cs:q=.5,d=.5) {$p_{12}$};
    \node[below=3pt] at (barycentric cs:q=.5,r=.5) {$p_{24}$};
    \node[left=3pt] at (barycentric cs:r=.5,a=.5) {$p_{45}$};
    \node[right=3pt] at (barycentric cs:r=.5,s=.5) {$p_{46}$};
    \node[left=3pt] at (barycentric cs:s=.5,a=.5) {$p_{56}$};
  \end{tikzpicture}
  \caption{The graph $G_{\rm NRUT}$ of Example~\ref{ex:NRUT}}
  \label{fig:gnrut}
\end{figure}
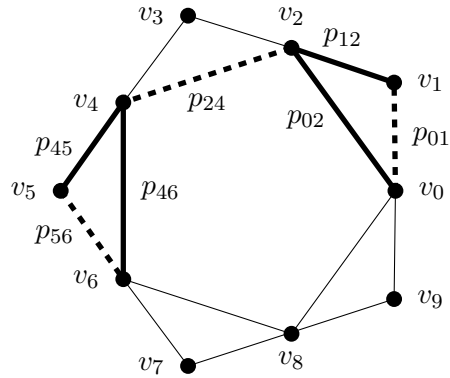

However, not every polytope satisfying  the odd-cycle condition
has a regular unimodular triangulation. This was observed by Ohsugi
and Hibi in \cite{1027.52008}.

\begin{example}[\cite{1027.52008}]
\label{ex:NRUT}
  Let $G_{\rm NRUT}$ be the graph obtained from a $10$-cycle with
  nodes numbered consecutively by $v_0, \ldots, v_9$ by adding the
  edges $(v_0,v_2)$, $(v_2,v_4)$, $(v_4,v_6)$, $(v_6,v_8)$ and
  $(v_0,v_8)$.  See Figure~\ref{fig:gnrut}.  This graph satisfies the
  odd cycle condition, so its edge
  polytope $P_{\rm NRUT}$ is integrally closed. It has a vertex
  $p_{ab}$ corresponding to every   edge $(v_a,v_b)$ in the graph.
 
\begin{lemma}
$P_{\rm NRUT}$ does not have a regular unimodular triangulation.
\end{lemma}

\begin{proof}
  The seven thick edges of Figure~\ref{fig:gnrut} form a five dimensional face $F$ of $P_{\rm NRUT}$  (they are a non-bipartite induced subgraph with six vertices) with the unique affine dependence
    \begin{align*}
    p_{03}+p_{23}+p_{45}+p_{46} =  p_{01}+2p_{24}+p_{56}\,.
  \end{align*} 
  Since one of the sides has all coefficients equal to one but the other does not, only one of the two triangulations of $F$ is unimodular  (Lemma~\ref{lemma:circuit}). If it is also regular, the weight vector defining it must choose the first term of the  binomial  $x_{03}x_{23}x_{45}x_{46}-x_{01}x_{24}^2x_{56}$ derived from the circuit (Theorem~\ref{thm:Groebner}). That is, the weight vector $\vomega\in\Z^{15}$ must satisfy
  \begin{align*}
    \omega_{02}+\omega_{23}+\omega_{45}+\omega_{46} >     \omega_{01}+2\omega_{24}+\omega_{56}\,.
  \end{align*}
  But there are four other faces obtained from this one by the pentagonal symmetry in $G_{\rm NRUT}$, leading to the inequalities
  \begin{align*}
    \omega_{24}+\omega_{34}+\omega_{67}+\omega_{68} &>     \omega_{23}+2\omega_{46}+\omega_{78}\,.\\
    \omega_{46}+\omega_{56}+\omega_{08}+\omega_{89} &>     \omega_{45}+2\omega_{68}+\omega_{01}\,.\\
    \omega_{68}+\omega_{78}+\omega_{01}+\omega_{02} &>     \omega_{67}+2\omega_{08}+\omega_{12}\,.\\
    \omega_{08}+\omega_{09}+\omega_{23}+\omega_{24} &>     \omega_{89}+2\omega_{02}+\omega_{34}\,.\\
  \end{align*}
  Since the sum of the five left-hand sides equals the sum of the five right-hand sides, these five inequalities cannot be simultaneously satisfied.
  Therefore, $P_{\rm NRUT}$ cannot have a regular   unimodular triangulation.
  \end{proof}
It was shown by Firla and Ziegler \cite{firlaZieglerHilbertBases99} (and further
studied by De Loera using his program PUNTOS) that the polytope
$P_{\rm NRUT}$ of the previous example does have a non-regular
unimodular triangulation.  Later, Ohsugi~\cite{1013.52010} showed that
$G_{\rm   NRUT}$ can be generalized to an infinite family of
integrally closed edge polytopes without regular unimodular
triangulations. Each graph in this family is obtained by successively
replacing a node of degree two with a path of length two.
\end{example}

The next theorem collects classes of graphs whose edge polytopes do have
regular unimodular triangulations. 

\begin{theorem}\label{thm:edgepolys}
  Let $G=(V,E)$ be a  finite simple graph,  possibly with loops,  and
  $P_G$ the associated edge polytope.
  \begin{compactenum}
  \item\label{thm:edgepolys:1} If $G$ is bipartite, then all
     triangulations of $P_G$  are unimodular. Further,  $P_G$  has a
    quadratic  triangulation if and only  if  every minimal cycle in $G$
    has length four.
  \item\label{thm:edgepolys:2} If $P_G$  is simple, but not a simplex,
    then $P_G$ has a quadratic triangulation. 
  \item\label{thm:edgepolys:3}    If $G$ does  not   contain a pair of
    disjoint  odd  cycles, then any  regular triangulation  of $P_G$ is
    unimodular.
  \end{compactenum}
\end{theorem}

\begin{proof}
The first part of (\ref{thm:edgepolys:1}) follows from the fact that
the  incidence  matrix of a  bipartite   graph is unimodular,  and the
convex hull of the rows  does not  contain  $\0$. Necessity in the second 
part comes from the fact that an induced cycle of length $2k$ in $G$ 
produces a face in $P_G$ containing a unique circuit, and the circuit
has $k$ positive and $k$ negative elements. Neither of the two triangulations
of such a circuit is flag (Lemma~\ref{lemma:circuit}). For sufficiency of the
second part we refer to Hibi and Ohsugi in \cite{0916.05046}.

(\ref{thm:edgepolys:2}) follows from work of Hibi and
Ohsugi~\cite[combining Prop.\ 1.1, Thms.\ 1.8 and 2.1]{Ohsugi2008}. In
the same paper, they also characterize which graphs $G$ give a simple
edge polytope other than a simplex. Further, they prove that all such
polytopes are smooth~\cite[Thm.\ 1.8]{Ohsugi2008}.  
Theorem~\ref{thm:empty-smooth} below provides an alternative proof of part (\ref{thm:edgepolys:2})
for the case where $G$ has no loops. 

(\ref{thm:edgepolys:3}) directly follows from Part~\ref{lemma:bipartite-edge-poly:simplex} of Lemma~\ref{lemma:edge-poly}.
\end{proof}

\medskip

Our last result in this section is from \cite{1027.52008}, where Ohsugi characterizes which graph polytopes are compressed.
We call $G$  an \emph{odd   cycle graph}
if $G$ has at least two disjoint odd cycles and satisfies the odd
cycle condition of Theorem~\ref{thm:OH:graphpolytopes}. 
Given two disjoint odd cycles $C$ and $C'$ in a graph $G$ we define
\begin{align*}
  S_{C'}(C):=\left\{v\in V(C)\;\bigg\vert\;%
    \genfrac{}{}{0pt}{}%
    {\text{$v$ is incident to a chord of C}}%
    {\text{or a bridge between $C$ and $C'$}}%
  \right\}\,.
\end{align*}
$S_{C'}(C)$ decomposes  $C$ into paths. Let  $s_{C'}(C)$ be the number
of paths of odd length in this decomposition. 

\begin{theorem}[\protect{\cite[Thm.\ 4.1]{1027.52008}}]
\label{thm:edgepolys:4} 
$P_G$ is compressed if and only if 
$s_{C'}(C)>1$ for all disjoint pairs of odd cycles $C,C'$ in $G$.
\end{theorem}

Observe that the condition in the statement implies that $G$ satisfies
the odd cycle condition. Indeed, if $G$ has an induced subgraph
consisting of two cycles $C$ and $C'$, then $s_{C'}(C)$ equals
one. The same paper also gives a complete characterization of all
graphs $G$ whose edge polytope has a regular unimodular pulling
triangulation~\cite[Thm.\ 4.4]{1027.52008}.  Moreover, it asserts that
for a simple odd cycle graph, the existence of a lifting function
defining a regular unimodular triangulation can be tested by checking
that a certain linearly defined subspace of edge space is
non-empty. The linear inequalities only depend on bridges and chords
in pairs of odd cycles~\cite[Thm.\ 3.5]{1027.52008}.

\subsection{Other Graph Polytopes}
\label{sec:graph-polytopes}

\subsubsection{Flow Polytopes} \label{sec:flow-polytopes}

Let $G=(V,A)$ be a directed graph with $n$ vertices and $m$ edges. The
incidence matrix $D_G$ of $G$ is the $(n\times m)$-matrix with a $1$
at position $(i,j)$ if edge $j$ is directed towards $i$, $-1$ if
$j$ is directed away from $i$, and $0$ otherwise.  Let $\va\in\Z^n$
be a \emph{demand vector} (i.e.\ some integer number for each vertex)
and $\vl,\vu\in (\Z\cup\{\infty\})^m$ be upper and lower bounds for the
flow on each edge.  The \emph{flow polytope} corresponding to the
digraph $D$ with demand $\va$ and bounds $\vl, \vu$ is
\begin{align}\label{eq:transport}
  F_{D,\va,\vl,\vu}:=\{\vx\in\R^m\mid  D_G\cdot \vx=\va, \;  \vl\le \vx\le
  \vu\}.
\end{align}
The dimension of $F_{D,\va,\vl,\vu}$ is at most $m-n+k$, where $k$ is
the number of connected components of the undirected graph
corresponding to $D$.  In the following we assume that
$\sum_{i=1}^na_i= 0$, as otherwise $F_{D,\va,\vl,\vu}$ is empty.
Roughly, a flow polytope is the set of all assignments of a flow to
all edges such that the upper and lower bounds are respected and at
each vertex the incoming and outgoing flow differ by the demand at
that vertex.

\emph{Transportation polytopes} are the class of flow polytopes for which $D=\vec\K_{n,m}$,
$\va=(-\vr,\vc)$, $\vl=\0$ and $\vu=\mathbf\infty$, where $\vec\K_{n,m}$
is the complete bipartite directed graph on $m$ and $n$ vertices.   Fulkerson's
integral    flow theorem    (see  e.g.\ \cite{Schrijver})
induces a path  decomposition of any  integral flow, which proves  the
following proposition.
\begin{proposition}
  Flow polytopes are lattice polytopes.
\end{proposition}

Birkhoff polytopes, denoted  $B_n$,  are another 
well known special case of transportation polytopes.
 They are the case where $\vr=\vc=\1$.  $B_n$ is
an $(n-1)$-dimensional \zo-polytope with facets $x_{ij}\ge 0$ for
$1\le i, j\le n$. This polytope can also be
defined as the convex hull of all $(n\times n)$-permutation matrices
or as the set of all $(n\times n)$-doubly stochastic matrices.

Consider the  $n=3$ case.  Geometrically, the Birkhoff polytope
$B_3$ is a direct sum of the two triangles. Viewed as
the convex hull of permutation matrices, the two triangles are given
by the odd and even permutations respectively.  These triangles are
not faces of $B_3$, and they form a circuit in $B_3$, i.e.\ a
minimally dependent set.  Hence, its toric ideal is the principal
ideal
\begin{align*}
\ideal_{B_3} =  \langle x_{123}x_{231}x_{312} - x_{132}x_{213}x_{321} \rangle.
\end{align*}
$\ideal_{B_3}$  has  two  initial  ideals,  $\langle  x_{123}x_{231}x_{312}
\rangle$ and $\langle x_{132}x_{213}x_{321} \rangle$.  This means $\ideal_{B_3}$ is
not  quadratically  generated,  and  hence,  $B_3$  does  not  have  a
quadratic triangulation. In fact,
$B_3$ and its
multiples are the only  $(3\times 3)$-transportation polytopes that fail to have a quadratic
triangulation.

\begin{theorem}[Haase, Paffenholz 07~{\cite[Thm.\ 1.5]{TransportPolys}}]\label{thm:quadTrans}
  If a $(3\times 3)$-transpor\-ta\-tion polytope $\transport$ is not a
  multiple of $B_3$ then $\transport$ has a quadratic triangulation.
\end{theorem}
Ohsugi and Hibi showed that for $k\ge 2$ all multiples $kB_3$ 
have  quadratic Gr\"obner bases~\cite{0907.3253}, but their initial
ideals are not square-free.

The proof of Theorem~\ref{thm:quadTrans} uses pulling refinements of
hyperplane subdivisions. Specifically, it looks at the set of
hyperplanes obtained by fixing the variable corresponding to one edge
in the definition of $ F_{D,\va,\vl,\vu}$ \eqref{eq:transport}. The transportation polytope is then subdivided
by the hyperplanes corresponding to all but one pair of adjacent
edges in $\vec\K_{n,m}$. The result is a finite list of lattice
isomorphism types of dicing cells, and there is a linear functional
that induces a flag pulling refinement on each cell. This gives a flag
regular unimodular triangulation of the transportation polytope.

The following proposition shows that the class of transportation
polytopes that are not multiples of some $B_n$ includes all smooth
transportation polytopes.
\begin{proposition}[Haase, Paffenholz 07~{\cite[Lemma~1]{TransportPolys}}]
  For  an  $(m\times   n)$ transportation polytope  $\transport$,  the
  following are equivalent.
  \begin{compactenum} 
  \item $X_{\transport}$ is smooth.
  \item $\transport$ is smooth.
  \item $\transport$ is simple.
  \item $\sum_{i \in  I}r_i \neq \sum_{j \in  J}c_j$  for all index sets    $I   \subset [m]$,  $J     \subset  [n]$  satisfying
    $|I|\cdot|J^c|, |I^c|\cdot|J|>1$,  where for $K\subset[n]$,  $K^c$ denotes the set $[n]\setminus
  K$.
  \end{compactenum}
 
\end{proposition}

In his MSc-Thesis, M.~Lenz~\cite{0709.3570} considered the case of smooth
$(3\times 4)$-transportation polytopes with the following result.
\begin{proposition}[Lenz 07, {\cite[Satz
    4.5.4]{0709.3570}}]\label{prop:lenz}
  Every smooth $(3\times 4)$-trans\-por\-ta\-tion polytope has a
  quadratic triangulation.
\end{proposition}

Based on experimental evidence, Diaconis and Eriksson
conjectured that the toric ideal
of the Birkhoff polytope is generated in degree
three~\cite[Conj.~7]{DiaconisEriksson}. This conjecture was later
confirmed by Yamaguchi, Ogawa, and Takemura~\cite{BirkhoffDegreeThree}. Their theorem leaves room for
speculation about degrees of Gr\"obner bases and minimal non-faces of
unimodular triangulations.

\subsubsection{Polytopes defined by characteristic vectors of subsets
  of vertices and edges} \label{sec:characteristic-polytopes}

Let $G=(V,E)$  be a finite simple  graph with vertex set  $V$ and edge
set $E$.   For any subset $W\subset  V$ the characteristic
vector $\vchi^W\in\R^V$  is given by $\chi^W_i=1$ if  $i\in W$ and
$0$  otherwise.   For a  subset  $\calA\subseteq  2^V$  
the\emph{      characteristic      polytope}  is    $P_\calA:=\conv(\vchi_A\mid
A\in\calA)$. We address such polytopes for various special choices of
$\calA$.

Given a graph $G=(V,E)$  \emph{vertex cover} is a  subset $C\subseteq E$ such that each $v\in
V$  is incident  to at  least one  $e\in C$.   A vertex  cover  $C$ is
\emph{minimal}  if no proper  subset of  $C$ is  a vertex  cover.  Let
$\calVC$ denote the set of all minimal vertex covers.
\begin{theorem}[Herzog, Hibi, Ohsugi '09 \cite{Herzog2009}]
  Let $G$ be  a bipartite graph.  If all minimal  vertex covers of $G$
  have  the  same  size,  then  $P_\calVC$ has  a  quadratic triangulation.
\end{theorem}

A \emph{stable set}  in $G$ is a subset $S$ of  the vertices such that
no two  vertices in  $S$ are connected  by an  edge.  
   For the set $\calS$ of all stable
sets  in   $G$, the  polytope  $P_\calS$  is   commonly  called  the
\emph{stable set polytope} of $G$.
A \emph{clique} in  a finite simple graph is  an induced subgraph that
is complete (i.e.\ every pair of vertices is  connected by an  edge).  The
\emph{chromatic number}  of a graph  is the minimal number  of colors
needed to  assign a color to  each vertex so that  no two
adjacent vertices have the same color.  Clearly, the chromatic number
of a  graph is at least the  cardinality of the maximal  clique in the
graph.   When the  chromatic
number of  every subgraph in a  finite  simple  graph equals the cardinality  of its maximal
clique, the graph is called  \emph{perfect}.

\begin{theorem}
The stable set polytope of a perfect graph is compressed.
\end{theorem}

\begin{proof}
Let $P$ be the stable set polytope of a graph $G=(V,E)$.
It is obvious that the following inequalities are feasible on $P$:
\[
\begin{cases}
x_v\ge 0 & \forall v\in V\\
\sum_{v\in C} x_v \le 1 & \forall \ \text{clique } C.
\end{cases}
\]

Chv\'atal~\cite[Thm.\ 3.1]{Chvatal75} proved that these inequalities actually define $P$ if (and only if) $G$ is perfect. Now, since clearly at every vertex of $P$ we have
\[
\begin{cases}
x_v\le 1 & \forall v\in V\\
\sum_{v\in  C} x_v \ge 0 & \forall \ \text{clique } C,
\end{cases}
\]
$P$ has width one with respect to every facet. In particular, by Theorem~\ref{thm:paco},
$P$ is compressed.
\end{proof}

As is somehow apparent in the above proof, for perfect graphs there is a duality between cliques (face inequalities) and stable sets (vertices). Since cliques in a graph $G$ correspond to stable sets in the complement graph $\overline G$, \cite[Thm.\ 3.1]{Chvatal75} is in fact a polyhedral proof of Lov\'asz theorem that the complement of a perfect graph is perfect.

One example of perfect graph is the comparability graph of any poset $(X,\preccurlyeq)$. Its cliques and stable sets are, respectively, the chains and antichains of $X$. That is, the corresponding stable set polytope has as vertices the characteristic vectors of antichains and as defining inequalities $\sum_{i\in  C} x_i \le 1$ for each chain $C$ (together with $x_i\ge 0$ for every $i$). This polytope was studied by Stanley~\cite{StanleyOrderPolytope}, who constructed a particular quadratic triangulation of it that piecewise linearly bijects to the dicing triangulation of the order polytope, mentioned in Section~\ref{sec:type-A-facets}. As a corollary, the order polytope and the chain polytope of $X$ have the same Ehrhart polynomial (in particular, the same volume, equal to the number of linear extensions of the poset).

\begin{remark}
Order polytopes and chain polytopes have been generalized to \emph{double posets} in~\cite{1606.04938}. Here, a double poset is a triple $(X,\preccurlyeq_1, \preccurlyeq_2)$ where $\preccurlyeq_1$ and $\preccurlyeq_2$ are two partial orders on $X$. The two orders, or the double poset, are said to be \emph{compatible} if they have at least one common linear extension. The Cayley difference of two lattice polytopes $P_1, P_2\subset \R^d$ is the lattice polytope $\conv( P_1\times \{ 0 \} \cup (-P_{2})\times \{ 1
\} )$ (compare with the definitions of Cayley sum in Sections~\ref{sec:fiberproducts}
and~\ref{sec:kmwreduction}). With these preliminaries, the \emph{double order polytope} (resp. \emph{double chain polytope} of $(X,\preccurlyeq_1, \preccurlyeq_2)$ is the Cayley difference of the order polytopes (resp. of the chain polytopes) of $(X,\preccurlyeq_1)$ and $(X,\preccurlyeq_2)$. 

Corollary~4.1 and Theorem 4.3 in \cite{1606.04938} show that if 
$(X,\preccurlyeq_1, \preccurlyeq_2)$ is a compatible double poset then its double order and double chain polytopes have quadratic triangulations. Moreover, such triangulations can be constructed so that there is a piecewise linear map between them, as in the original case studied by Stanley. Corollary 4.7  in \cite{1606.04938} gives a summation formula for the volume of these polytopes.
\end{remark}

Going back to characteristic vectors related to graphs, 
an alternate common definition takes the set of
all edges instead of the vertices as a base set, yielding
vectors $\vchi_W\in \R^E$.  Then for any subset $U\subset V$, we can
associate the subset $C_U\in E$ of all edges incident to exactly one
node from $U$, and define the \emph{cut polytope} $\cut(G)\subset\R^E$ of $G$ as
 the characteristic polytope given by the set $\calA:=\{C_U\mid U\subset
V\}$.
\begin{theorem}[Sullivant 2004, \cite{SullivantCompressed}; see also \cite{Sturmfels2007}]
  Let $G=(V,E)$ be a finite simple undirected graph. The cut polytope
  $\cut(G)$ of $G$ is compressed if and only if
  \begin{enumerate}
  \item $G$ has no $K_5$-minor, and
  \item all induced cycles in $G$ have length at most four.
  \end{enumerate}
\end{theorem}

In this context, $G$ contains a
$K_5$-minor if we can find five distinct vertices together with a set
of pairwise internally vertex disjoint paths that connects each pair
of the five vertices.

In discrete optimization one is often interested in whether the
polytope  associated to a combinatorial optimization
problem is integrally closed, as this provides a way to solve integer
linear optimization problems. In this setting, integrally closed
polytopes are said to have the \emph{integer decomposition
  property}. Examples include the $s$-$t$-connector polytope $P_G$ of
a directed graph $G=(V,A)$ (Trotter, see \cite[Thm.~13.8]{MR1956924}),
the base polytope, the spanning set polytope  of a matroid and the
independent set polytope of a matroid~\cite[Cor.~42.1e]{MR1956924},
and the up and down hull of the perfect matching polytope
\cite[Cor.~20.9c, 20.11b]{MR1956924}.  Note that this is not true for
matching polytopes in general, as illustrated by the Petersen graph. 

Various polytopes associated to combinatorial optimization problems
are facet unimodular (i.e.~the matrix of facet normals is totally
unimodular). These include
$b$-transhipment polytopes~ \cite[Section 11.4]{MR1956924}, bipartite matching and bipartite perfect
matching polytopes~\cite[Section 18.1]{MR1956924}, vertex cover~\cite[Section 18.4]{MR1956924}, edge cover and stable set polytopes of bipartite graphs~\cite[Section 19.5]{MR1956924}. Hence, by
Theorem~\ref{thm:hyperplanes} each of these classes have  regular unimodular
triangulations.

\subsection{Lecture hall polytopes}
\label{sec:lecture-hall}

Euler's classic result that there are as many partitions of an integer
$n$ into odd parts as there are partitions into distinct parts can be
regarded as a ``limit'' ($d \to \infty$) of the following theorem by
Bousquet-M\'elou and Eriksson~\cite{MR1607531}: For every $n,d\in \N$ the
number of so called $d$-lecture hall partitions of $n$ is equal to the
number of partitions of $n$ into an odd and less than $2d$ number of parts.
Here, a
\emph{$d$-lecture hall partition} is a partition $\lambda \in \Z^d$
satisfying the inequalities
\[
  0 \le \frac{\lambda_1}{1} \le \frac{\lambda_2}{2} \le \ldots \le
  \frac{\lambda_d}{d} \,.
\]
As Savage says in her survey \cite{MR3534075}, ``Over the past twenty
years, lecture hall partitions have emerged as fundamental structures
in combinatorics, number theory, algebra, and geometry, leading to new
generalizations and interpretations of classical theorems and new
results''.

Already Bousquet-M\'elou and Eriksson point out that the study of
lecture hall partitions falls naturally within the theory of lattice
points in cones~\cite[Sect.~5]{MR1606188}.
For $d \ge 1$, the lecture hall simplex is defined as
\[
\lecturehallsimplex{d+1} \ := \ \conv \left[
\begin{array}{cccccc}
0 & 0 & \cdots & \cdots & 0 & 1 \\
0 & \vdots &  & \dddots & 2 & 2 \\
0 & \vdots & \dddots & \dddots & & \vdots \\[.6mm]
0 & 0 & d-1 & \cdots & \cdots & d-1 \\
0 & d & d & \cdots & \cdots & d 
\end{array}
\right]
\subset \R^d \,.
\]
It has the inequality description
\[
  \lecturehallsimplex{d+1} \ = \left\{ \vx \in \R^d \ \Big| \ 0 \le
    \frac{x_1}{1} \le \frac{x_2}{2} \le \ldots \le \frac{x_d}{d} \le 1
  \right\} \,.
\]
We use $d+1$ rather than $d$ as index since the cone over this simplex is
unimodularly equivalent to the more familiar $d+1$st lecture hall cone.
The triangulation in the following statement was communicated by C.~Haase to the authors of 
\cite{lecturehall-triangulation}, where it appears in detail. 
Beck, Braun, K\"oppe, Savage, and Zafeirakopoulos
then ask~\cite[Conj.~6.1]{lecturehall-triangulation} whether the
triangulation can be chosen in such a way that it elucidates the
desirable enumerative properties of $\lecturehallsimplex{d+1}$.

\begin{theorem}[\negthickspace\negmedspace\protect{\cite[Thm.\ 4.2]{lecturehall-triangulation}}]
\label{thm:lecturehall} 
The lecture hall simplex $\lecturehallsimplex{d+1}$ has a quadratic
triangulation.
\end{theorem}

\begin{proof}
We proceed by induction on $d$. For $d=1$, $\lecturehallsimplex{2}$ is
a unit interval.

For $d \ge 2$, it is natural to intersect $\lecturehallsimplex{d+1}$ with the
hyperplane $x_d-x_{d-1} = 1$:
\[
  P \ := \ \left\{ \vx \in \lecturehallsimplex{d+1} \ | \ x_d-x_{d-1}
    = 1
  \right\} \ = \ \conv \left[
    \begin{array}{cccccc}
      0 & \cdots & \cdots & 0 & 1 \\
      \vdots &  & \dddots & 2 & 2 \\
      \vdots & \dddots & \dddots & & \vdots \\[.6mm]
      0 & d-1 & \cdots & \cdots & d-1 \\
      1 & d & \cdots & \cdots & d 
    \end{array}
  \right]
\]
to obtain a simplex $P$ which is unimodularly equivalent to
$\lecturehallsimplex{d}$ and thus has a quadratic triangulation $\T$
by induction. The hyperplane splits $\lecturehallsimplex{d+1}$ into
two simplices which can be triangulated in a compatible fashion.

\noindent
\underline{$x_d - x_{d-1} \le 1$:}
This polytope below $P$ is a lattice pyramid over $P$ with apex the
origin. Coning off $\T$, we obtain a quadratic triangulation which is
compatible with any weights inducing $\T$ on $P$.

\noindent
\underline{$x_d - x_{d-1} \ge 1$:}
This polytope above $P$ is (equivalent to) a chimney polytope over $P$
(compare Sect.~\ref{sec:chimney}). In fact, let $\pi \colon \R^d \to
\R^{d-1}$ be the projection that forgets the last coordinate, and let $P' := \pi(P) \cong P$.
Then
\[
  \left\{ \vx \in \lecturehallsimplex{d+1} \ | \ x_d-x_{d-1}
    \ge 1 \right\}
  \ = \
  \left\{ (\vy,y_d) \in P' \times \R \ | \ y_{d-1}+1 \le y_d \le d
  \right\} \,.
\]
Thus,  we can lift $\T$ to a quadratic triangulation of the polytope above
$P$.
\end{proof}

More generally, for any sequence $\s = (s_i)_{i=1}^d$, Savage and
Schuster~\cite{SavageSchuster} define the $\s$-lecture hall simplex
$\slecturehall{d+1}$ to be
\begin{equation}
  \label{eq:sLectureHall}
  \slecturehall{d+1} \ := \left\{ \vx \in \R^d \ \Big| \ 0 \le
    \frac{x_1}{s_1} \le \frac{x_2}{s_2} \le \ldots \le \frac{x_d}{s_d}
    \le 1
  \right\} \,.
\end{equation}
Equivalently, 
\[
\slecturehall{d+1} \ := \ \conv \left[
\begin{array}{ccccc}
0 &  \cdots & 0 & s_1 \\
 \vdots & \dddots  & \dddots & \vdots \\[.6mm]
0 & s_d  & \cdots & s_d 
\end{array}
\right]
\subset \R^d \,.
\]
It would be desirable to understand for which $\s$ this simplex has a
unimodular triangulation.

The simplices $\slecturehall{d+1}$ appear also in the following works: (1)
The fundamental parallelepiped of the simplex $\slecturehall{d+1}$ has been studied by Liu and
Stanley~\cite{MR3245894}.
(2) Consider a finite partial order $(X,\preccurlyeq)$ as
in Section~\ref{sec:type-A-facets}, and take as additional input 
a function $\s \in \Z_{>0}^X$ giving positive integer weights to
its elements. Then Br\"and\'en and Leander~\cite{BrandenLeander}
define an $\s$-order polytope $O(\preccurlyeq,\s)$ as the image of the
order polytope $O(\preccurlyeq)$ under the coordinate-wise scaling $x
\mapsto (s_1x_1,\ldots,s_nx_n)$.
Via this map, the canonical triangulation of the usual order polytope triangulates 
$O(\preccurlyeq,\s)$ into the simplices
$\slecturehall[\sigma(\s)]{n+1}$ where the permutation $\sigma$ runs
over all linear extensions of $\preccurlyeq$.

\subsection{Smooth Polytopes}  \label{sec:Smooth}

We now turn our attention to two classes of smooth polytopes,  those
with no interior points and those with one interior lattice
point,  satisfying a special condition.

\subsubsection{Empty Polytopes} \label{sec:empty}  

First we consider  lattice polytopes whose only lattice  points are their
vertices.   The  following theorem  establishes  strong
restrictions on the combinatorics of such polytope if they are smooth.
\begin{theorem}  \label{thm:empty-smooth}  Every  smooth polytope  $P$
  such  that  $P\cap\Z^d=\vertices{P}$  is  lattice equivalent  to  a
  product of unimodular simplices.
\end{theorem}
Together with Proposition~\ref{prop:products}, this yields
the following corollary.
\begin{corollary}
  Any smooth lattice polytope  $P$ satisfying $P\cap\Z^d=\vertices{P}$ has a
  quadratic triangulation.
\end{corollary}
The proof of the theorem is  very much inspired by 
Kaibel and Wolff's proof of a slightly different result.
\begin{theorem}[Kaibel, Wolff '00~\cite{KaibelWolff00}]
  \label{thm:KaibelWolff}
  Any simple \zo-po\-ly\-to\-pe is the product of some \zo-simplices.
\end{theorem}
Note that in the vertex-edge graph of a simple polytope every vertex has
$d$ neighbors, and for any choice of neighbors  $N$ of a vertex $v\in
P$ in  the graph of $P$, there is a unique face $F$ of $P$ such that
$\{v\}\cup N \subseteq F$ and $N$ 
is the  set of neighbors of  $v$ in the  graph of $F$. This
face is denoted  $F(v,N)$.  The following  lemma is needed for our proof of
Theorem~\ref{thm:empty-smooth}.    The  crucial  observation   in
both  proofs is that every  face of $P$ is again smooth and
all its lattice points are vertices. In particular, every two-face is
either a standard triangle or a unit square.
\begin{lemma}
  If    $P$    is   a    smooth    lattice    polytope   such    that
  $P\cap\Z^d=\vertices{P}$,  $v\in\vertices{V}$, and $N$ is a set
  of neighbors of  $v$ in the graph of  $P$ such that no  two vertices
  in $N$  are  adjacent  in  the   graph  of  $P$, then  $F(v,N)$  is
  a combinatorial cube.
\end{lemma}
\begin{proof}
  Up  to a unimodular  transformation we  can assume  that $\vv$  is the
  origin $\0$  and the adjacent  vertices are $\ve_1,  \ldots, \ve_d$.

  Let $E_I:=\{\ve_i\mid i\in I\}$  for some $I\subset\{1,\ldots, r\}$ be
  a set of non-adjacent vertices.   We use induction on the coordinate
  sum  $\vv_I:=\sum_{i\in  I}\ve_i$.   If   $|I|=2$,  then  by  the  above
  observation  the vertices  in $E_I$  span a  square at  $\0$. Hence,
  $\vv_I\in P$.   If $|I|\ge 3$, then  by induction all  partial sums of
  the elements of $E_I$ are contained  in $P$. Hence, they span a face
  that  differs from  the cube  by at  most the  vertex $\vv_I$.  But in
  dimensions three and higher, the cube minus a vertex is not a smooth polytope,
  so $\vv_I\in P$.
\end{proof}

\begin{proof}[Proof of Theorem~\ref{thm:empty-smooth}.]
  We can again assume that $\0$  is a vertex with neighbors $\ve_i$ ($1
  \le i  \le d$). So the  incident facet inequalities are  $x_i \ge 0$
  ($1 \le  i \le d$).   Since $P$  is a simple  polytope, any $k$  of the
  $\ve_i$ define a $k$-dimensional face.
  
  The first step is to  show that ``being adjacent'' is an equivalence
  relation  among  the  vertices  $\ve_i$.  That is, if for  
  pairwise  distinct $1\le  i,j,k\le  d$,  $\ve_i$ is  adjacent  to  $\ve_j$,  and $\ve_j$  is
  adjacent  to  $\ve_k$, then  all  three  are  contained in  a  common
  two-dimensional face.  Suppose $\ve_i$ and $\ve_k$ were not adjacent, then
  $\ve_i-\ve_j+\ve_k$ would  be a  forth vertex in  this face.   But this
  violates the inequality $x_j\ge 0$. So, $\ve_i$ and $\ve_k$ are
  adjacent, and the three vertices are in a common two dimensional face.

  If $C_0,  \ldots, C_r$  are the equivalence  classes of this relation,
   then for $1\le j\le r$,  $C_j\cup  \{\0\}$ spans a unimodular
  simplex  $\Delta_j$.    Letting  $Q:=  \Delta_1   \times  \cdots  \times
  \Delta_r$ be  the product of these  simplices, we aim  to show that
  $P=Q$.

 By the previous  lemma all vertices of $Q$ are  contained in $P$. It
  remains to show  that all  edges of $Q$  are edges of $P$.   Before 
  proving this  we show  that this will  finish the  proof. The
  graph of $P$ is connected, so if  there was a vertex $v$ of $P$
  that was not  a vertex of $Q$,  then at least one of  the vertices of
  $P$ that was  in $Q$ would have a neighbor in  $P$ that was not a
  neighbor in $Q$, so $P$ would not be simple.

  We now show that all edges of $Q$ are edges of $P$. By the  previous lemma, if $E_I$  is a set  of pairwise non-adjacent
  neighbors of $\0$,  then all edges of  the cube $F(\0,E_I)$ are
  in  $P$.  It remains  to  argue that  if $\ve_1$  and $\ve_2$  are
  adjacent  and  $\{1\}   \sqcup  I$  indexes  pair-wise  non-adjacent
  vertices  of $P$,  then  $\ve_1+\ve_I$ and  $\ve_2+\ve_I$ are  adjacent.
  Assume they are not  adjacent.  The faces $F(\0,\{\ve_1\}\cup E_I)$
  and $F(\0,\{\ve_2\}\cup E_I)$ are cubes.
  So as neighbors of the
  vertex $v_I$,  vertices  $\ve_1+\vv_I$, $\ve_2+\vv_I$, and
  $\vv_I-\ve_i$ are  pair-wise non-adjacent, for
  $i\in  I$.  Hence, they  span a  cube at
  $\vv_I$, contradicting the adjacency of $\ve_1$ and $\ve_2$.
\end{proof}

\subsubsection{Reflexive Polytopes}
\label{subsubsec:reflexive}

In this section we report on a class of polytopes where a
computational approach implementing pull-back and push-forward
subdivisions has been quite successful.

A lattice polytope $P$ is called \emph{reflexive}, if it contains a
unique interior lattice point and all facets are lattice distance one
from this point.  Without lose of generality it can be assumed that the interior point is the
origin.  The \emph{polar} of a polytope $P\subseteq\R^d $ with
$\0\in\interior(P)$ is
\begin{align*}
  P^\vee :=\{\vu\in\R^d \mid \langle \vx, \vu\rangle \ge -1\; \forall
  \, \vx\in P\}\,.
\end{align*}
If $P$ is reflexive then $P^\vee$ is again a lattice polytope. $P$ is
\emph{smooth} if $P$ is simple and the primitive generators of every
vertex cone span the lattice.

\begin{theorem}
 For $d\le 3$, every reflexive $d$-dimensional lattice polytope has a
  regular unimodular triangulation. 
\end{theorem}
\begin{proof}
  The origin is the unique interior lattice point of $P$. Any
  pulling triangulation in which the origin is pulled first will be
  regular and unimodular. For $d=3$
  we use the fact that any full triangulation of a polygon is
  unimodular.
\end{proof}

  There are $5, 18, 124, 866, 7622$, $72256$, $749892$, and
$8229721$ \emph{smooth} reflexive polytopes in dimensions two, three, four, five, six, seven, eight, and nine
 respectively (up to lattice
equivalence). Those of dimension up to eight were first computed by
{\O}bro~\cite{OebroSmoothFano} and a full list of explicit
representatives can be obtained from the \texttt{polymake database}
(\href{http://www.polymake.org/doku.php/data}{www.polymake.org/doku.php/data});
the classification has been extended to dimension nine by Lorenz and
the third author, see
\href{http://polymake.org/polytopes/paffenholz/www/fano.html}{polymake.org/polytopes/paffenholz/www/fano.html}.

We used two
approaches to establish regular unimodular triangulations of the
polars of these polytopes up to dimension seven (checking dimensions eight
and nine is in progress):
\begin{enumerate}
\item We checked whether the facet normals define a unimodular system
  (then the existence of a triangulation follows from
  Theorem~\ref{thm:hyperplanes}).
\item We searched for  a sequence of projections along
  coordinate directions satisfying the conditions for push-forward
  and pull-back subdivisions given in Section \ref{sec:chimney}. If
  such a sequence projects the polytope down to
  dimension two, the polytope has a regular
  unimodular triangulation (as all two-dimensional lattice polytopes
  have regular unimodular triangulations).  If there is such a sequence
  projecting the polytope to
  dimension one, we know that the polytope has a
  \emph{quadratic} triangulation.

  For those polytopes for which a two-dimensional, but no one-dimensional projections were found, individual inspection confirmed
  a quadratic triangulation for some of them.
\end{enumerate}
Both checks were done with software system
\polymake~\cite{Joswig2009,polymake-sw} using an extension for
projections of polytopal subdivisions\cite{subdivProj-sw}.  The
detailed results are listed in
Table~\ref{tab:smoothreflexive}, and are summarized in the next
theorem.
\begin{table}[t]
  \centering
  \begin{tabular}{rrrrrrrrrrr}
    dim.\ & number of&  RUT&  quadratic& facet uni- & projects to\\
    &polytopes&&triangulation&modular&dimension one\\
    \toprule  
    2&      5&      5&      5&      5&      5\\
    3&     18&     18&     18&     16&     18\\
    4&    124&    124&    124&     96&    124\\
    5&    866&    866&    866&    554&    866\\
    6&   7622&   7622&   $\ge$7620&   4097&   $\ge$7620\\
    7&  72256&  72256&  $\ge$72240&  31881&  $\ge$72240\\
  \end{tabular}
  \caption{Number of smooth reflexive polytopes of dimension$\le 7$; all of them possess regular unimodular triangulations, 
  and all except perhaps 18 possess quadratic ones.
  }
  \label{tab:smoothreflexive}
\end{table}
\begin{theorem}[Haase, Paffenholz~'09~\cite{HPchimneys}]\label{thm:smoothreflexive}\hfill
  \begin{compactitem}
  \item All smooth reflexive $d$-polytopes for $d\le 8$ are integrally
    closed.%
  \item All smooth reflexive polytopes in dimension six have a regular
    unimodular triangulation, and all but at most two
    have a quadratic triangulation.
  \item All smooth reflexive polytopes in dimension seven have a regular
    unimodular triangulation, and all but at most $16$
    have a quadratic triangulation.
\end{compactitem}
\end{theorem}
The remaining two smooth Fano polytopes in dimension six and the $16$
smooth Fano polytopes in dimension seven may still have a quadratic
triangulation, but we were not able to construct such a triangulation
with our approach.  We checked whether the polytopes are integrally
closed using the \polymake-interface to
\texttt{Normaliz}~\cite{normaliz2}.  Previously it was shown by
Piechnik that all smooth reflexive $d$-polytopes have a regular
unimodular triangulation for $d\le 4$. Computations in dimensions eight
and nine are currently work in progress. Data vor these computation can be found at \href{http://polymake.org/polytopes/paffenholz/www/rut.html}{\texttt{polymake.org/polytopes/paffenholz/www/rut.html}}.

\subsection{The Gr\"obner fan and the toric Hilbert scheme}

\label{sec:Hilbert}

\subsubsection{The Gr\"obner fan and the secondary fan}

Here we examine the relation between Gr\"obner bases and subdivisions, mentioned in section~\ref{sec:toric-groebner-bases}.

As in section~\ref{sec:toric-groebner-bases},%
\footnote{Everything we say in this section can be extended to the more general case where $\configuration$ is any finite and homogeneous set of lattice vectors in $\Z^{d+1}$}
we let $\configuration := (P \times
\{1\}) \cap \Z^{d+1}$ be the homogenized set of lattice points in a polytope
$P\in \R^d$,
 take the polynomial ring $S := \mathbbm{k}[x_{\va} \! : \! \va \in \configuration]$
with one variable for each lattice point in $P$, and consider the toric ideal of $P$, i.e., the binomial ideal generated by linear dependences among the lattice points:
\[
\ideal_P = \left\langle \ \vx^\vm-\vx^\vn \ : \ \vm, \vn \in \Z_{\ge
    0}^\configuration \ , \ 
    \sum_{\va \in \configuration} m_{\va} \va
  = \sum_{\va \in \configuration} n_{\va} \va \ \right\rangle.
\]
Recall, each
choice of weights $\vomega \in \R^\configuration$
induces a regular subdivision $\T_\vomega$ of $P$, on the lattice
polytope side, and 
an initial ideal $\lead_\vomega \ideal_P := \langle \lead_\vomega f \ : \ f \in I
\rangle$, on the toric algebra side. 

However, if $\vomega$ is not \emph{generic}, the subdivision
$\T_\vomega$ may not be a triangulation, and the ideal $\lead_\vomega \ideal_P$
may not be a monomial. This means a polynomial $f$ may have several
monomials of highest weight with respect to $\vomega$, and all those
monomials form the leading part of $f$, which we also denote $\lead f$.
For example, if $\vomega_{\va}$ is the same constant for every $\va$ then:
\begin{itemize}
\item $\T_v\vomega$ is the \emph{trivial subdivision} ($P$ itself is its only full-dimensional cell) and 
\item $\lead_\vomega \ideal_P =\ideal_P$, since $\ideal_P$ is
  homogeneous and for a homogeneous $f$, and $\lead f = f$.
\end{itemize}

We state without proof two properties relating the initial ideals
$\lead_\vomega \ideal_P$ and the corresponding regular subdivisions
$\T_\vomega$.
 (See~\cite[Chapter 10]{SturmfelsGBCP} for details.)
\begin{theorem}[Sturmfels 96 \cite{SturmfelsGBCP}]
If  $\lead_{\vomega_1} \ideal_P = \lead_\vomega (\lead_{\vomega_2} \ideal_P)$ for some $\vomega_1$, $\vomega_2$ and $\vomega$ (that is, if $
\lead_{\vomega_1} \ideal_P $ is an initial ideal of $\lead_{\vomega_2} \ideal_P $) then the subdivision $\T_{\vomega_1}$ refines the subdivision $\T_{\vomega_2}$.
\end{theorem}

\begin{corollary}
If $\lead_{\vomega_1} \ideal_P =\lead_{\vomega_2} \ideal_P$ for different  $\vomega_1$ and $\vomega_2$, then $\T_{\vomega_1} = \T_{\vomega_2}$. This follows from the previous property by letting $\vomega=\vomega_1$ (and then switching the roles of $\vomega_1$ and $\vomega_2$).
\end{corollary}

These properties mean that there is an order-preserving map from the poset of all initial ideals of $\ideal_P$ to the poset of all regular subdivisions of $\configuration$, where the latter are partially ordered by refinement and the former are ordered by $\ideal_1 
< \ideal_2 $ if $\ideal_1 $ is an initial ideal of $\ideal_2 $. This
was proved by Sturmfels in~\cite{SturmfelsGroebnerBasesTohoku}.

Put another way, the regular subdivision and the initial ideal
construction provide two stratifications of the  vector space
$\R^\configuration$. The first is based sets of $\vomega$'s which give
the same regular subdivisions and the second is based on which
$\vomega$'s give the same initial ideals.
Both stratifications are 
complete rational polyhedral fans, and they are called , respectively,
the \emph{secondary fan} and the \emph{Gr\"obner fan} of
$\configuration$.

\begin{theorem}[Sturmfels 91,~\cite{SturmfelsGroebnerBasesTohoku}]
\label{thm:groebner-fan}
The Gr\"obner fan of $\configuration$ refines the secondary fan of
$\configuration$.
\end{theorem}

Observe that for the secondary fan to be well defined, the
subdivisions of $\configuration$ need to be considered as subdivisions
of
point configurations, as when we defined weak pulling in Section~\ref{sec:weakstrong}.

Theorem~\ref{thm:groebner-fan} implies that  $\T_\omega$ can be
clearly determined from $\lead_{\omega} \ideal_P$.

\subsubsection{The toric Hilbert scheme}
We now look at a property that is shared by all initial ideals of
$\ideal_P$ and the toric ideal $\ideal_P$ itself.

As detailed in 
\cite[Chapter 10]{SturmfelsGBCP}, 
$\configuration$ defines a $d$-dimensional
multi-grading on the polynomial ring
$\mathbbm{k}[x_{\va} \! : \! \va \in \configuration]$, assigning multi-degree $\va _i$
to the variable $x_i$.
Ideals $\ideal\subset\mathbbm{k}[x_{\va} \! : \! \va \in \configuration]$ that are
homogeneous with respect to this grading have 
well-defined Hilbert functions
\[
\begin{tabular}{ccc}
$\Z_{\ge 0}\configuration$ & $\longrightarrow$ & $\Z_{\ge 0}$ \cr
$\vb$            &  $\longmapsto $   & $\dim_{\mathbbm{k}}\ideal_\vb$ \cr
\end{tabular}
\]
where $\Z_{\ge 0}$ is the set of non-negative integers, $\Z_{\ge 0}\configuration$ is the
semigroup of non-negative integer combinations of $\configuration$, and for each
$\vb\in\Z_{\ge 0}\configuration$, $\ideal_\vb$ is the degree $\vb$ part of $\ideal$.

The most natural $\configuration$-homogeneous ideal is the toric ideal $\ideal_P$, generated by the binomials 
\[
\{  \ \vx^\vm-\vx^\vn \ : \ \vm, \vn \in \Z_{\ge
    0}^\configuration \ , \ 
    \sum_{\va \in \configuration} m_{\va} \va
  = \sum_{\va \in \configuration} n_{\va} \va \ 
\},
\]
because every $\vb\in\Z_{\ge 0}\configuration$, $(\ideal_P)_\vb$ has
codimension one in $(\mathbbm{k}[x_{\va} \! : \! \va \in \configuration])_\vb$. This characterizes the
Hilbert function of $\ideal_\vb$.

An $\configuration$-homogeneous
ideal $\ideal \subset \mathbbm{k}[x_{\va} \! : \! \va \in \configuration]$ is called \emph{$\configuration$-graded}
if it has the same Hilbert function as the toric ideal $\ideal_P$.
$\configuration$-graded ideals include all the
initial ideals of $\ideal_P$, but can include other ideals as well. 
The \emph{toric Hilbert scheme}, as introduced by Peeva and Stillman
\cite{PeevaStillman}, is
the set of all $\configuration$-graded ideals with a suitable algebraic
structure defined by some determinental equations. An equivalent
 description via
binomial equations appeared in \cite[\pS6]{Sturmfels-agraded}.
(See also \cite{MaclaganThomas,StillmanSturmfelsThomas}.)

Surprisingly, the $\configuration$-graded ideals that are not
initial are still related to subdivisions of $\configuration$.
Sturmfels \cite[Theorem 10.10]{SturmfelsGBCP} proved that 
the order-preserving map implied in Theorem~\ref{thm:groebner-fan} extends to an
order preserving map from the poset of all $\configuration$-graded
ideals, where the partial ordered is given by
by ``toric deformation'' (a generalization of the property of being an initial ideal), to the poset
of all subdivisions, still ordered by refinement.

That is to say, every $\configuration$-graded ideal 
 $I$ has a canonically associated
polyhedral subdivision $\T_I$ of $\configuration$. 
If $I$ is monomial, then $\T_I$ is a triangulation, whose simplices are spanned by the
standard monomials in $\mathbbm{k}[ x_{\va} \! : \! \va \in \configuration]/\operatorname{Rad}(I)$. 

Santos~\cite{SantosNonconnectedHilbertScheme} used this map to show the existence
of non-connected toric Hilbert schemes, elaborating on work of
Maclagan and Thomas~\cite{MaclaganThomas}. 
Here are the main ideas.
There is a natural and well-known adjacency relation between
triangulations of the same configuration. It can be defined as an
operation that takes out certain simplices and inserts others, but it
is equivalent to the following~\cite[Section~2.4]{deLoeraRambauSantos}: two triangulations $\T_1$ and $\T_2$
of $\configuration$ are related by a \emph{geometric bistellar flip}
(or just \emph{flip}, for short) if there is a polyhedral subdivision
$\T$ of $\configuration$ whose only refinements are $\T_1$ and
$\T_2$. Maclagan and Thomas defined an analogous adjacency relation
between $\configuration$-graded monomial ideals ({\em
  mono-$\configuration$-GIs} for short), which they also called
flip. Their relation has the following properties.

\begin{proposition}[MacLagan \& Thomas 02, \cite{MaclaganThomas}]
\begin{compactenum}
\item A toric Hilbert scheme is connected if and only if the graph of
  mono-$\configuration$-GIs is connected.
\item Triangulations 
of $\configuration$ corresponding to neighboring mono-$\configuration$-GIs either
  coincide or differ by a geometric bistellar flip.
\end{compactenum}
\end{proposition}

Lattice point configurations with non-connected graphs of flips are
rare, but they were shown to exist in 2000~\cite{SantosNoFlips}. Still,
this does not necessarily imply the corresponding toric Hilbert
scheme is non-connected, because the Sturmfels map going from
$\configuration$-graded ideals to subdivisions of $\configuration$ is
in general not surjective \cite[Example 10.13]{SturmfelsGBCP},
\cite{PeevaStillman}.
However, Maclagan and Thomas also  observed, based on \cite[Theorem
10.14]{SturmfelsGBCP}, that the image of the map always 
contains all the unimodular triangulations of $\configuration$. 

\begin{corollary}
\label{coro:mactho}
The toric Hilbert scheme of $\configuration$ has at least as many
connected components as
there are connected components in the graph of triangulations of $\configuration$ that
contain unimodular triangulations. 
\end{corollary}

Santos' non-connected toric Hilbert scheme is based on the
construction of a polytope $P$ for which the corresponding
configuration $\configuration$ has unimodular triangulations in
different components of the graph of flips. Let $Q\in \R^4$ be the
$24$-cell. The $24$-cell is one of the six regular four-dimensional polytopes, and can be realized as the convex hull of
the root system of type $\rootD_4$.  We
consider it as a lattice polytope in the lattice it
generates,
$D_4:=\{(a_1,a_2,a_3,a_4)\in \Z^4 : \sum a_i \in 2\Z\}$. Its only
lattice points are the origin and the $24$ roots $\pm \ve_i \pm
\ve_j$, with $(i,j)\in {\binom{[4]}2}$. $Q$ has 24-facets, all
regular octahedra, and 96 two-faces, all regular triangles.

\begin{theorem}[Santos 05, \cite{SantosNonconnectedHilbertScheme}, see
  also~\protect{\cite[Section~7.3]{deLoeraRambauSantos}}]
\label{thm:nonconnected:hilbert}
If $P=Q\times [0,1]$ is a lattice polytope in $D_4\times \Z$, such that
$\configuration$ has 50 points: the 48 vertices of $P$ plus the
centers of $Q\times\{0\}$ and $Q\times\{1\}$, then:
\begin{compactenum}
\item The graph of triangulations of $\configuration$ has at least 13
  connected components, each containing at least $3^{48}$ unimodular
  triangulations.
\item The toric Hilbert scheme of $\configuration$ has at least
  13 connected components each containing at least $3^{48}$ monomial
  ideals, and they all have dimension at least $96$.
\end{compactenum}
\end{theorem}

\begin{proof}
The proof uses the idea of \emph{staircase refinements} introduced in Section~\ref{sec:joins-and-products}.
We sketch it here, and refer to~\cite[Section 7.3]{deLoeraRambauSantos} or \cite{SantosNonconnectedHilbertScheme} for details. The crucial tool is an orientation of the
graph of the two-cell with the following properties:
\begin{itemize}
\item It is locally acyclic, but no edge can be reversed in it without creating a cycle in one of the $96$ triangular faces of $D_4$
\item In each of the $24$ octahedral facets of $Q$, one (and only one) of the four-cycles of edges is a directed cycle.
\end{itemize}

\begin{figure}
\begin{tikzpicture}[y=-1cm,scale=.53]

\draw (16.66875,3.6449) -- (15.08125,4.5974);
\draw (16.98625,6.1849) -- (15.39875,7.1374);
\draw (19.20875,7.29615) -- (16.51,5.39115);
\draw (16.66875,3.6449) -- (16.01682,9.2075);
\draw (15.38182,3.81) -- (16.98625,6.1849);
\draw (16.01682,9.2075) -- (15.39875,7.1374);
\draw (16.98625,6.1849) -- (19.20875,7.29615);
\draw[arrows=-stealth'] (15.39875,7.1374) -- (16.1925,6.66115);
\draw[arrows=-stealth'] (16.01682,9.2075) -- (17.62125,8.24865);
\draw[arrows=-stealth'] (15.38182,3.81) -- (16.98625,2.85115);
\draw[arrows=-stealth'] (15.08125,4.5974) -- (15.875,4.12115);
\draw[arrows=-stealth'] (16.99048,6.18067) -- (16.8275,4.9149);
\draw[arrows=-stealth'] (15.3924,7.09083) -- (15.24423,5.9055);
\draw[arrows=-stealth'] (15.38182,3.81) -- (15.69932,6.35);
\draw[arrows=-stealth'] (17.24448,3.12208) -- (17.4625,2.93582);
\draw[arrows=-stealth'] (17.69957,6.55532) -- (17.85832,6.63575);
\draw[arrows=-stealth'] (15.76282,8.27617) -- (15.66757,8.04757);
\draw[arrows=-stealth'] (15.17015,4.33917) -- (15.25482,4.15925);
\draw[arrows=-stealth'] (17.3355,3.90525) -- (17.4625,3.7465);
\draw[arrows=-stealth'] (16.54175,4.71382) -- (16.5735,4.49157);
\draw[arrows=-stealth'] (17.8054,6.2865) -- (17.653,6.15738);
\draw[arrows=-stealth'] (16.65182,5.72982) -- (16.77882,5.90338);
\draw[arrows=-stealth'] (16.27082,7.27075) -- (16.27082,7.06332);
\draw[arrows=-stealth'] (15.97025,6.20607) -- (16.08032,6.03038);
\draw[arrows=-stealth'] (16.03375,4.74557) -- (16.1544,4.953);
\draw[arrows=-stealth'] (15.7099,4.953) -- (15.875,5.03132);
\draw (15.09607,4.60375) -- (16.4465,5.34882);
\draw (19.20875,7.29615) -- (16.01682,9.2075);
\draw (18.57375,1.89865) -- (16.66875,3.6449) -- (16.98625,6.1849) -- (19.20875,7.29615) -- (18.57375,1.89865) -- (15.38182,3.81) -- (15.08125,4.5974) -- (15.39875,7.1374) -- (16.01682,9.2075) -- (15.38182,3.81);
\draw (15.39875,7.1374) -- (18.57375,1.89865);
\draw (16.66875,3.6449) -- (18.57375,1.89865);
\draw[arrows=-stealth'] (18.57798,1.93675) -- (18.89125,4.43865);
\draw (21.9075,3.4925) -- (22.225,6.0325) -- (20.6375,6.985);
\draw (24.765,6.35) -- (22.225,6.0325);
\draw (23.1775,7.3025) -- (24.765,6.35) -- (24.4475,3.81);
\draw (20.32,4.445) -- (20.6375,6.985) -- (23.1775,7.3025) -- (22.86,4.7625) -- (20.32,4.445) -- (21.9075,3.4925) -- (24.4475,3.81) -- (22.86,4.7625);
\draw (20.6375,6.985) -- (24.4475,3.81);
\draw (24.765,6.35) -- (20.32,4.445);
\draw (21.9075,3.4925) -- (23.1775,7.3025);
\draw (22.86,4.7625) -- (22.225,6.0325);
\draw[arrows=-stealth'] (20.6375,6.985) -- (20.47875,5.715);
\draw[arrows=-stealth'] (24.765,6.35) -- (24.60625,5.08);
\draw[arrows=-stealth'] (22.225,6.0325) -- (22.06625,4.7625);
\draw[arrows=-stealth'] (23.1775,7.3025) -- (23.01875,6.0325);
\draw[arrows=-stealth'] (20.32,4.445) -- (21.59,4.60375);
\draw[arrows=-stealth'] (22.86,4.7625) -- (23.65375,4.28625);
\draw[arrows=-stealth'] (24.4475,3.81) -- (23.1775,3.65125);
\draw[arrows=-stealth'] (21.9075,3.4925) -- (21.11375,3.96875);
\draw[arrows=-stealth'] (20.6375,6.985) -- (21.9075,7.14375);
\draw[arrows=-stealth'] (23.1775,7.3025) -- (23.97125,6.82625);
\draw[arrows=-stealth'] (24.765,6.35) -- (23.495,6.19125);
\draw[arrows=-stealth'] (22.225,6.0325) -- (21.43125,6.50875);
\draw[arrows=-stealth'] (22.33718,5.79543) -- (22.44725,5.60493);
\draw[arrows=-stealth'] (22.68643,5.82718) -- (22.63775,5.715);
\draw[arrows=-stealth'] (22.87693,5.55625) -- (22.71818,5.49275);
\draw[arrows=-stealth'] (21.88422,5.93725) -- (22.08953,5.77003);
\draw[arrows=-stealth'] (22.6695,5.30225) -- (23.20925,4.83447);
\draw[arrows=-stealth'] (22.65468,5.11175) -- (22.7965,4.92125);
\draw[arrows=-stealth'] (22.44725,5.09482) -- (22.38375,4.8895);
\draw[arrows=-stealth'] (22.0345,5.17525) -- (21.717,5.04825);
\draw (15.39875,11.303) -- (16.03375,16.383) -- (21.11375,17.018) -- (20.47875,11.938);
\draw (18.57375,9.398) -- (15.39875,11.303) -- (20.47875,11.938) -- (23.65375,10.033) -- cycle;
\draw (21.11375,17.018) -- (24.28875,15.113) -- (23.65375,10.033);
\draw[arrows=-stealth'] (20.47875,11.938) -- (21.11375,10.668);
\draw[arrows=-stealth'] (15.39875,11.303) -- (14.2875,10.82675);
\draw[arrows=-stealth'] (23.65375,10.033) -- (24.60625,9.23925);
\draw[arrows=-stealth'] (18.57375,9.398) -- (18.25625,8.4455);
\draw (18.57375,9.398) -- (19.20875,14.478) -- (16.03375,16.383);
\draw (19.20875,14.478) -- (24.28875,15.113);
\draw[arrows=-stealth'] (21.336,17.78) -- (21.2725,17.49425);
\draw[arrows=-stealth'] (25.4,15.58925) -- (24.60625,15.27175);
\draw[arrows=-stealth'] (14.9225,17.3355) -- (15.08125,17.17675);
\draw[arrows=-stealth'] (18.89125,15.113) -- (19.05,14.7955);
\draw[arrows=-stealth'] (16.03375,16.383) -- (18.57375,16.7005);
\draw[arrows=-stealth'] (21.11375,17.018) -- (22.70125,16.0655);
\draw[arrows=-stealth'] (24.28875,15.113) -- (21.74875,14.7955);
\draw[arrows=-stealth'] (19.20875,14.478) -- (17.62125,15.4305);
\draw[arrows=-stealth'] (15.39875,11.303) -- (17.93875,11.6205);
\draw[arrows=-stealth'] (20.47875,11.938) -- (22.06625,10.9855);
\draw[arrows=-stealth'] (23.65375,10.033) -- (21.11375,9.7155);
\draw[arrows=-stealth'] (18.57375,9.398) -- (16.98625,10.3505);
\draw[arrows=-stealth'] (15.39875,11.303) -- (15.71625,13.843);
\draw[arrows=-stealth'] (18.57375,9.398) -- (19.05,13.208);
\draw[arrows=-stealth'] (23.65375,10.19175) -- (23.97125,12.573);
\draw[arrows=-stealth'] (20.47875,11.938) -- (20.79625,14.478);
\draw (21.11375,17.018) -- (21.2725,17.49425);
\draw (16.03375,16.383) -- (15.08125,17.17675);
\draw (24.28875,15.113) -- (25.4,15.58925);
\draw (18.25625,16.39782) -- (19.20875,14.478);
\draw (11.43,15.08125) -- (13.6525,16.1925);
\draw (11.1125,12.54125) -- (13.0175,11.1125);
\draw[arrows=-stealth'] (11.70517,12.26608) -- (11.8745,12.10733);
\draw[arrows=-stealth'] (12.28725,15.51517) -- (12.04383,15.38817);
\draw (11.43,15.08125) -- (13.6525,16.1925);
\draw (8.89,14.76375) -- (8.5725,15.5575);
\draw[arrows=-stealth'] (11.43,15.08125) -- (10.16,14.9225);
\draw[arrows=-stealth'] (13.6525,16.1925) -- (11.1125,15.875);
\draw (7.9375,10.48808) -- (8.5725,12.22375);
\draw[arrows=-stealth'] (8.6995,15.22942) -- (8.78417,15.02833);
\draw (13.6525,16.1925) -- (13.0175,11.1125) -- (7.9375,10.48808) -- (8.5725,15.5575) -- cycle;
\draw (11.1125,12.54125) -- (8.5725,12.22375) -- (8.89,14.76375) -- (11.43,15.08125) -- cycle;
\draw[arrows=-stealth'] (11.1125,12.54125) -- (9.8425,12.3825);
\draw[arrows=-stealth'] (13.0175,11.1125) -- (10.4775,10.80558);
\draw (7.9375,10.48808) -- (11.43,15.08125);
\draw (8.5725,15.5575) -- (11.1125,12.54125);
\draw (13.0175,11.1125) -- (10.23408,13.55725) -- (8.89,14.76375);
\draw[arrows=-stealth'] (8.5725,12.22375) -- (10.25525,13.53608) -- (13.6525,16.1925);
\draw[arrows=-stealth'] (11.01725,14.53092) -- (10.84792,14.30867);
\draw[arrows=-stealth'] (11.03842,14.17108) -- (10.84792,14.00175);
\draw[arrows=-stealth'] (10.90083,13.00692) -- (10.7315,13.1445);
\draw[arrows=-stealth'] (10.68917,13.0175) -- (10.55158,13.19742);
\draw[arrows=-stealth'] (9.94833,13.32442) -- (9.76842,13.17625);
\draw[arrows=-stealth'] (9.9695,13.13392) -- (9.78958,12.954);
\draw[arrows=-stealth'] (9.76842,13.97) -- (9.5885,14.12875);
\draw[arrows=-stealth'] (9.82133,14.097) -- (9.68375,14.26633);
\draw[arrows=-stealth'] (8.43492,11.8745) -- (8.33967,11.60992);
\draw[arrows=-stealth'] (8.89,14.76375) -- (8.73125,13.54667);
\draw[arrows=-stealth'] (7.9375,10.4775) -- (8.255,13.0175);
\draw[arrows=-stealth'] (11.43,15.08125) -- (11.27125,13.81125);
\draw[arrows=-stealth'] (13.0175,11.1125) -- (13.335,13.6525);
\draw (6.985,9.525) -- (6.35,7.46125);
\draw (3.81,7.14375) -- (1.905,8.89);
\draw (7.9375,6.50875) -- (10.16,7.62);
\draw (5.3975,6.19125) -- (5.08,6.985);
\draw (10.16,7.62) -- (5.08,6.985) -- (1.905,8.89) -- (6.985,9.525) -- cycle;
\draw (7.9375,6.50875) -- (5.3975,6.19125) -- (3.81,7.14375) -- (6.35,7.46125) -- cycle;
\draw (10.16,7.62) -- (3.81,7.14375);
\draw (7.9375,6.50875) -- (1.905,8.89);
\draw (6.35,7.46125) -- (5.08,6.985);
\draw[arrows=-stealth'] (7.9375,6.50875) -- (6.6675,6.35);
\draw[arrows=-stealth'] (5.3975,6.19125) -- (4.60375,6.6675);
\draw[arrows=-stealth'] (3.81,7.14375) -- (5.08,7.3025);
\draw[arrows=-stealth'] (6.35,7.46125) -- (7.14375,6.985);
\draw[arrows=-stealth'] (10.16,7.62) -- (7.62,7.3025);
\draw[arrows=-stealth'] (5.08,6.985) -- (3.52425,7.90575);
\draw[arrows=-stealth'] (1.905,8.89) -- (4.445,9.2075);
\draw[arrows=-stealth'] (6.985,9.525) -- (8.5725,8.5725);
\draw[arrows=-stealth'] (6.80508,6.95325) -- (7.20725,6.81567);
\draw[arrows=-stealth'] (5.70442,6.84742) -- (5.57742,6.56167);
\draw (5.3975,6.19125) -- (6.985,9.525);
\draw[arrows=-stealth'] (6.08542,7.366) -- (6.25475,7.41892);
\draw[arrows=-stealth'] (5.40808,7.27075) -- (5.16467,7.239);
\draw[arrows=-stealth'] (6.31825,8.10683) -- (6.20183,7.89517);
\draw[arrows=-stealth'] (7.03792,7.39775) -- (6.80508,7.37658);
\draw[arrows=-stealth'] (5.36575,7.06967) -- (5.57742,7.16492);
\draw[arrows=-stealth'] (5.1435,6.80508) -- (5.24933,6.55108);
\draw[arrows=-stealth'] (3.19617,7.71525) -- (3.34433,7.54592);
\draw[arrows=-stealth'] (6.70983,8.5725) -- (6.62517,8.27617);
\draw[arrows=-stealth'] (8.98525,7.02733) -- (8.81592,6.96383);
\draw[arrows=-stealth'] (3.96875,8.05392) -- (4.1275,8.01158);
\draw (5.08,2.2225) -- (5.715,4.28625);
\draw (8.255,4.60375) -- (10.16,2.8575);
\draw (4.1275,5.23875) -- (1.905,4.1275);
\draw (6.6675,5.55625) -- (6.985,4.7625);
\draw (1.905,4.1275) -- (6.985,4.7625) -- (10.16,2.8575) -- (5.08,2.2225) -- cycle;
\draw (4.1275,5.23875) -- (6.6675,5.55625) -- (8.255,4.60375) -- (5.715,4.28625) -- cycle;
\draw (1.905,4.1275) -- (8.255,4.60375);
\draw (4.1275,5.23875) -- (10.16,2.8575);
\draw (5.715,4.28625) -- (6.985,4.7625);
\draw[arrows=-stealth'] (4.1275,5.23875) -- (5.3975,5.3975);
\draw[arrows=-stealth'] (6.6675,5.55625) -- (7.46125,5.08);
\draw[arrows=-stealth'] (8.255,4.60375) -- (6.985,4.445);
\draw[arrows=-stealth'] (5.715,4.28625) -- (4.92125,4.7625);
\draw[arrows=-stealth'] (1.905,4.1275) -- (4.445,4.445);
\draw[arrows=-stealth'] (6.985,4.7625) -- (8.5725,3.81);
\draw[arrows=-stealth'] (10.16,2.8575) -- (7.62,2.54);
\draw[arrows=-stealth'] (5.08,2.2225) -- (3.4925,3.175);
\draw[arrows=-stealth'] (5.25992,4.79425) -- (5.38692,4.74133);
\draw[arrows=-stealth'] (6.44525,5.06942) -- (6.36058,4.94242);
\draw (6.6675,5.55625) -- (5.08,2.2225);
\draw[arrows=-stealth'] (5.85258,4.34975) -- (5.99017,4.39208);
\draw[arrows=-stealth'] (6.97442,4.51908) -- (6.64633,4.47675);
\draw[arrows=-stealth'] (7.5565,3.86292) -- (7.9375,3.73592);
\draw[arrows=-stealth'] (5.74675,3.64067) -- (5.67267,3.48192);
\draw[arrows=-stealth'] (5.02708,4.34975) -- (4.80483,4.32858);
\draw[arrows=-stealth'] (6.48758,4.572) -- (6.65692,4.62492);
\draw[arrows=-stealth'] (6.80508,5.21758) -- (6.858,5.08);
\draw[arrows=-stealth'] (8.86883,4.03225) -- (9.05933,3.85233);
\draw[arrows=-stealth'] (5.51392,3.58775) -- (5.461,3.39725);
\draw[arrows=-stealth'] (3.37608,4.85775) -- (3.24908,4.78367);
\draw (11.27125,9.2075) -- (13.17625,7.46125) -- (12.85875,4.92125) -- (10.63625,3.81) -- (11.27125,9.2075) -- (14.44625,7.3025) -- (14.76375,6.50875) -- (14.44625,3.96875) -- (13.81125,1.905) -- (14.44625,7.3025);
\draw (13.17625,7.46125) -- (14.76375,6.50875);
\draw (12.85875,4.92125) -- (14.44625,3.96875);
\draw (10.63625,3.81) -- (13.81125,1.905) -- (13.81125,2.06375);
\draw (10.63625,3.81) -- (13.335,5.715);
\draw (13.17625,7.46125) -- (13.81125,1.905);
\draw (14.44625,3.96875) -- (11.27125,9.2075);
\draw (14.44625,7.3025) -- (12.85875,4.92125);
\draw (13.81125,1.905) -- (14.44625,3.96875);
\draw (12.85875,4.92125) -- (10.63625,3.81);
\draw[arrows=-stealth'] (14.44625,3.96875) -- (13.6525,4.445);
\draw[arrows=-stealth'] (13.81125,1.905) -- (12.22375,2.8575);
\draw (13.17625,7.46125) -- (11.27125,9.2075);
\draw[arrows=-stealth'] (14.44625,7.3025) -- (12.85875,8.255);
\draw[arrows=-stealth'] (14.76375,6.50875) -- (13.716,7.14375);
\draw[arrows=-stealth'] (13.17625,7.46125) -- (13.0175,6.19125);
\draw[arrows=-stealth'] (10.63625,3.81) -- (10.95375,6.6675);
\draw[arrows=-stealth'] (14.76375,6.50875) -- (14.605,5.3975);
\draw[arrows=-stealth'] (13.81125,1.905) -- (14.12875,4.7625);
\draw[arrows=-stealth'] (12.24068,8.3439) -- (12.3825,8.17033);
\draw[arrows=-stealth'] (12.14543,4.55083) -- (11.98668,4.4704);
\draw[arrows=-stealth'] (14.25575,3.28083) -- (14.17743,3.05858);
\draw[arrows=-stealth'] (14.55843,6.99558) -- (14.6685,6.7564);
\draw[arrows=-stealth'] (13.22917,5.94783) -- (13.10217,6.1595);
\draw[arrows=-stealth'] (13.28208,6.46642) -- (13.2715,6.61458);
\draw[arrows=-stealth'] (13.18683,5.60917) -- (13.02808,5.49275);
\draw[arrows=-stealth'] (13.19318,5.37633) -- (13.06618,5.20277);
\draw[arrows=-stealth'] (13.49375,4.70958) -- (13.45142,5.15408);
\draw[arrows=-stealth'] (13.81125,5.02708) -- (13.62075,5.334);
\draw[arrows=-stealth'] (13.81125,6.36058) -- (13.68425,6.13833);
\draw[arrows=-stealth'] (13.91708,6.04308) -- (13.68425,5.91608);
\draw (14.76375,6.5024) -- (13.3985,5.75733);
\draw (4.445,12.85875) -- (2.2225,11.7475);
\draw (4.7625,15.39875) -- (2.8575,16.8275);
\draw[arrows=-stealth'] (4.7625,15.39875) -- (4.60375,14.12875);
\draw[arrows=-stealth'] (2.2225,11.7475) -- (2.54,14.2875);
\draw[arrows=-stealth'] (3.82693,16.2814) -- (3.96875,16.10783);
\draw[arrows=-stealth'] (3.73168,12.48833) -- (3.57293,12.4079);
\draw (4.445,12.85875) -- (2.2225,11.7475);
\draw (6.985,13.17625) -- (7.3025,12.3825);
\draw[arrows=-stealth'] (4.445,12.85875) -- (5.715,13.0175);
\draw[arrows=-stealth'] (2.2225,11.7475) -- (4.7625,12.065);
\draw (7.9375,17.4625) -- (7.3025,15.71625);
\draw[arrows=-stealth'] (7.3025,12.3825) -- (7.62,14.9225);
\draw[arrows=-stealth'] (7.493,16.40417) -- (7.57132,16.62642);
\draw[arrows=-stealth'] (7.10142,12.90108) -- (7.16492,12.68942);
\draw (2.2225,11.7475) -- (2.8575,16.8275) -- (7.9375,17.4625) -- (7.3025,12.3825) -- cycle;
\draw (4.7625,15.39875) -- (7.3025,15.71625) -- (6.985,13.17625) -- (4.445,12.85875) -- cycle;
\draw[arrows=-stealth'] (7.3025,15.71625) -- (7.14375,14.44625);
\draw[arrows=-stealth'] (4.80483,15.40933) -- (6.04308,15.56808);
\draw[arrows=-stealth'] (2.8575,16.8275) -- (5.3975,17.145);
\draw[arrows=-stealth'] (4.8895,13.82183) -- (5.02708,13.95942);
\draw[arrows=-stealth'] (5.03767,13.67367) -- (5.16467,13.83242);
\draw[arrows=-stealth'] (4.87892,15.01775) -- (5.04825,14.89075);
\draw[arrows=-stealth'] (5.04825,15.07067) -- (5.17525,14.91192);
\draw[arrows=-stealth'] (5.969,14.859) -- (6.08542,15.01775);
\draw[arrows=-stealth'] (6.0325,14.72142) -- (6.223,14.88017);
\draw[arrows=-stealth'] (5.97958,14.097) -- (6.1595,13.91708);
\draw[arrows=-stealth'] (5.969,13.92767) -- (6.10658,13.7795);
\draw (4.7625,15.39875) -- (7.3025,12.3825);
\draw (2.2225,11.7475) -- (7.3025,15.71625);
\draw (4.445,12.85875) -- (7.9375,17.4625);
\draw (6.985,13.17625) -- (2.8575,16.8275);

\end{tikzpicture}%
 \caption{A locally acyclic orientation of the graph of a 24-cell.}{\label{fig:24cell_lao}}
\end{figure}
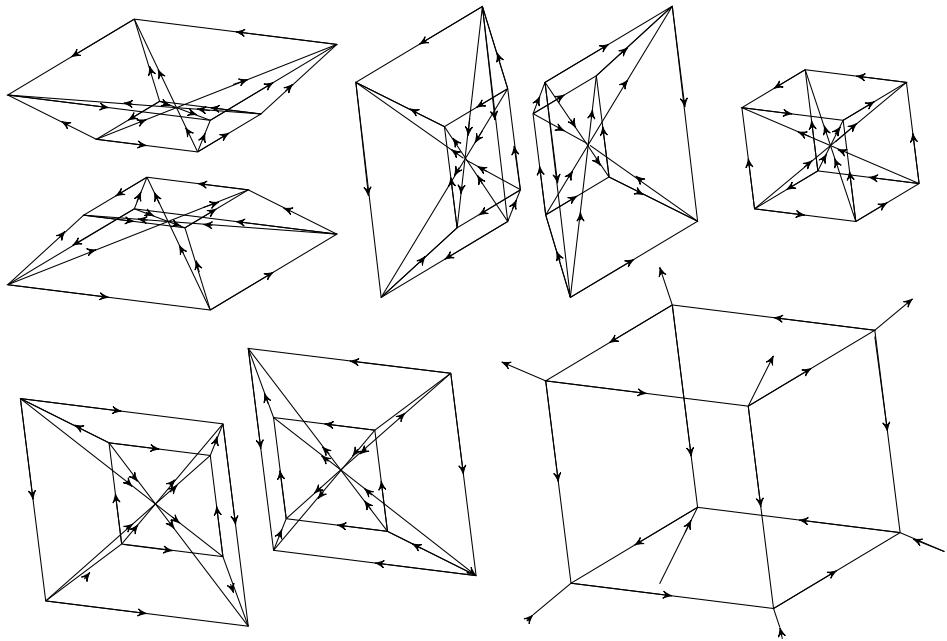

Figure~\ref{fig:24cell_lao} depicts this orientation, in several
pieces. To understand the figure, observe that that $24$ vertices of
$Q$ are the union of the $16$ vertices of a four-cube plus the eight
points $\pm2e_i$. Moreover, the graph of $Q$ is the graph of the
four-cube plus the edges joining each of the eight three-cube facets of
it to the corresponding $e_i$ or $-e_i$. The figure shows
stereographic projections of the eight three-cube facets of the four-cube, each with the orientation of its edges. The projection makes two of the facets regular cubes, so that one of the extra points is sent to infinity. We leave it to the reader to check that the orientation displayed has the two claimed properties (thanks to symmetry this task is not as hard as it might seem).

In order to  triangulate $P$ we first triangulate each octahedral facet
of $Q$ in the unique way that does not create a local cycle. That is,
 in each octahedron use the diagonal that is orthogonal to the
quadrilateral cycle of that octahedron. Also, by the properties of the
triangulation, this new diagonal can only be given one of its two
orientations without creating a local cycle. Cone this triangulation
of the boundary of $Q$ to the origin, and orient every new edge away
from the origin.
 This produces a triangulation $\T$ of $Q$ with a
locally acyclic orientation.
 Further, this triangulation and orientation 
 is unimodular, and it has no 
  reversible edge. That is, reversal of any individual edge 
 creates a directed cycle. 
For the $96$ original edges of the
$24$-cell, this is part of what we claimed before. It is also true and
easy to check for the new $24$ diagonals of the octahedra and $24$ edges from the center.

We can now triangulate $P$ using the staircase refinement of $\T$ and the edge $[0,1]$ (we give to $[0,1]$ either of its two orientations). 
It turns out that the fact that no edge can be reversed in the locally
acyclic orientation implies that this triangulation cannot be
connected by flips to any triangulation constructed in the same way
using a different initial orientation of the edges of the
$24$-cell. This proves that the graph of triangulations contains
unimodular triangulations in different components. There are at least
$13$ such components because symmetries of the $24$-cell produce $12$
different ways of constructing the initial locally acyclic
orientation, plus another component that will contain all the
\emph{globally acyclic} orientations. The  $3^{48}$  triangulations in
each component and the $96$ in the dimension of the toric Hilbert scheme are 
by-products of the many symmetries of the construction.
\end{proof}

\section{Dilations and the KMW Theorem} \label{sec:kmw}

Dilating a lattice polytope by a positive integer is a natural
operation.
One of the first theorems about unimodular triangulations was proved
in the early days of toric geometry by~Knudsen, Mumford, and Waterman%
\footnote{We call this the Knudsen-Mumford-Waterman Theorem, or KMW, since
Knudsen  says: ``One of the key steps is due to Alan Waterman. 
The rest is a truly joint effort by Mumford and
me''~\cite[p.~109]{KKMS3}.}~\cite{KKMS3},
who were interested in semi-stable reduction of families over
curves.
\begin{theorem}[KMW theorem, \cite{KKMS3}]\label{thm:kmw}
  Given a polytope $P$, there is a  $c \in \Z_{>0}$ such
  that the dilation $c \cdot P$ admits a regular unimodular
  triangulation.
\end{theorem}

We say that $c$ is a \emph{KMW-number} of $P$ if $cP$ has a unimodular
triangulation.  This KMW-theorem raised several questions that are still open, including:
\begin{itemize}
\item What is the minimum $c(P)$ for a given polytope $P$?
  Is there a $c(d)$ that is a KMW-number for every polytope of
  dimension $d$? 
\item What is the structure of the set of KMW-numbers of a given $P$?
  Is it a monoid? Theorem~\ref{thm:unimodular-dilations} 
  implies it is closed under taking multiples of an element. It is not
  clear whether it is closed under taking sums.
\item To our knowledge, Example~\ref{kmw:eg:c+1} below is the first
  class of polytopes $P$ and integers $c$ in the literature such that
  $c$ is a KMW-number for $P$ but $c+1$ is not.
\end{itemize}

The following two theorems by Bruns,
Gubeladze, and Trung; and Eisenbud,
Reeves, and Totaro, respectively, show that we cannot expect algebraic
obstructions to these questions.
\begin{theorem}[\cite{BGT}]
  $c \cdot P$ is integrally closed for $c \ge d-1$, and Koszul for $c \ge
  d$.
\end{theorem}

\begin{theorem}[\cite{EisenbudReevesTotaro}] \label{thm:ert}
  After a linear change of coordinates, $\ideal_{cP}$ has a quadratic
  Gr\"obner basis for $c \ge \operatorname{reg}(\ideal_P)/2$.
\end{theorem}

In this respect if we relax the requirement from a unimodular triangulation to a \emph{unimodular cover}
there is indeed a constant $c_d$ depending only on the dimension such that any dilation
by a factor of at least $c_d$ has a unimodular cover. This result is from Bruns and Gubeladze 
but an improvement by von Thaden shows that $c_d$ can be bounded by a polynomial in $d$.

\begin{theorem}[\cite{BGcover,vonThadenThesis}]
For each fixed dimension $d$ there is a $c_d\in O(d^6)$ such that
 $c \cdot P$ has a unimodular cover for every $c\ge c_d$.
\end{theorem}

The second half of this section is dedicated to a proof of
Theorem~\ref{thm:kmw}.
Our proof differs from the original one (and the reworking of it
in~\cite{BG}) in two ways.
First, a more careful application of the elementary reduction step
enables us to avoid using ``local lattices''~\cite{BG} or ``rational
structures''~\cite{KKMS3} and yields a cleaner proof of regularity.%
\footnote{
Bruns and Gubeladze~\cite{BG} omit regularity, and the proof of
regularity in~\cite{KKMS3} (Theorem 2.22, pp. 147--151) is quite
intricate.} 
Secondly, we set up a book-keeping method that allows us
to obtain the bound stated in Theorem~\ref{thm:KMW-effective} below.
Determining a bound based upon the original proof would be difficult
and involve a tower of exponentials whose minimal length would be the
maximum volume among the simplices in the triangulation of $P$ being
used for the construction (see Remark~\ref{rem:tower-of-exponentials}).
In our proof the bound is ``merely'' double exponential.

\begin{theorem}[Effective KMW Theorem]
\label{thm:KMW-effective}
If $P$ is a lattice polytope of dimension $d$ and (lattice) volume $\vol(P)$, then the dilation 
\[
(d+1)!^{\vol(P)! {(d+1)^{(d+1)^2\vol(P)}}} P
\]
has a regular unimodular triangulation.

More precisely, if $P$ has a triangulation $\T$ into $N$ $d$-simplices, of volumes $V_1,\dots,V_N$, then the dilation 
\[
(d+1)!^{\sum_{i=1}^N  V_i! \left( (d+1)!^{d+1}\right)^{V_i-1}} \T
\]
has a regular unimodular refinement.

\end{theorem}

Before going into the KMW Theorem we deal with two questions that are,
in a sense, special cases of it. In Section~\ref{sec:kmw3d} we review
what is known about KMW-numbers of three-dimensional lattice
polytopes, summing up results from~\cite{KantorSarkaria}
and~\cite{SantosZiegler}.
In Section~\ref{sec:kmwhypersimplex} we  give two proofs that if a
polytope $P$ has a unimodular triangulation every dilation $cP$
($c\in\N$) of it has one, and that regularity and flagness of the
triangulation can also be preserved in the process. This is equivalent
to saying (as noted above) that the set of KMW-numbers of any polytope
$P$ is closed under taking products. 

The \emph{canonical triangulations} of
ordered simplices introduced in Section~\ref{sec:kmwhypersimplex} are also instrumental 
for the proof of Theorem~\ref{thm:KMW-effective}, which  occupies Sections~\ref{sec:kmwreduction}--\ref{sec:KMW-effective}. Section~\ref{sec:kmwreduction} deals with the case of a lattice $d$-simplex $P$ and shows that $cP$ can be triangulated into simplices of volume
\emph{strictly smaller} than $P$, for some $c\in\{1,\dots,d+1\}$. In
Section~\ref{sec:KMWfirstproof} this reduction is applied iteratively
to a triangulation of a general polytope $P$
to yield Theorem~\ref{thm:kmw}. However, the effective version stated
as Theorem~\ref{thm:KMW-effective} requires a careful analysis of how to
control the number (and type) of new simplices obtained in each
iterative step. This is done in our final Section~\ref{sec:KMW-effective}.

\subsection{KMW numbers in dimension three} \label{sec:kmw3d}
If $P$ is a segment or a lattice polygon, then $P$ has a regular
unimodular triangulation, so for the one and two-dimensional cases, $P$ has KMW-number $c(P)=1$.

For three-dimensional polytopes the following is known:
\begin{compactenum}
\item One and two are not KMW-numbers of every three-polytope. In fact, there are non-unimodular simplices whose second dilation does not have a unimodular triangulation (Ziegler 1997, unpublished).
\item It is not known whether three and five are KMW-numbers of  every three-polytope.
\item Every other integer is a KMW-number of every three-polytope (Kantor and  Sarkaria~\cite{KantorSarkaria} for the integer four, Santos and Ziegler~\cite{SantosZiegler} for every other integer).
\end{compactenum}

These results all emanate from the fact that  every full triangulation of $P$ gives a subdivision of $cP$ into dilations of empty simplices. 
This means that in order to prove that $cP$ can be unimodularly triangulated for every $P$, we can restrict our attention to the case where $P$ is an empty simplex, as long as we manage to make sure that the triangulations of the individual dilated simplices  agree on their common faces.
So, let us look at what empty simplices in dimension three look like.

According to White's Theorem~\cite{White,Scarf,MS}, every empty tetrahedron in $\R^3$ has width one with respect to some lattice direction (not necessarily with respect to a facet normal). 
That is, it fits between two adjacent lattice hyperplanes. Consequently, it is either unimodular (if it has three vertices in one of the two lattice hyperplanes) or lattice equivalent to
\[
S(p,q) := \conv \left[
  \begin{smallmatrix}
    0&0&1&p\\
    0&0&0&q\\
    0&1&0&1
  \end{smallmatrix}
  \right]
  \]
for some integers $0 \le p \le q$ with $gcd(p,q)=1$. Observe that
$S(p,q) \sim S(q-p,q)\sim S(p^{-1},q) \sim S({q-p}^{-1},q)$, where
inverses are taken modulo $q$. It can be shown that these are the only
cases that give equivalent simplices. The parameter $q$ is the volume of $S(p,q)$, while $p$ carries information of more arithmetic nature.

In the $c$-th dilation $cS(p,q)$, all the lattice simplices lie in either the bottom edge, the top edge, or the $c-1$ intermediate horizontal lattice planes. The latter can be tiled by lattice translations of the parallelogram $C(p,q):= 2S(p,q)\cap \{z=1\}= \conv\{(0,0,1), (1,0,1), (p,q,1), (p+1,q,1)\}$.  So, in order to understand the structure of lattice points in $cS(p,q)$ we only need to look at those in $C(p,q)$. These are its four vertices plus $q-1$ interior points with the following properties: (a) there is exactly one of them in each of the $c-1$ lattice lines of direction $(1,0,0)$ and intersecting  $C(p,q)$ and (b) there is exactly one in each of the $c-1$ lattice lines of direction $(0,p,q)$. This induces two distinct orderings of these interior $q-1$ points which we call the $Y$-order and the $X$-order. 

All the results about unimodular triangulations of a dilated tetrahedron $S(p,q)$ use, in one way or another, the following decomposition of $c S(p,q)$. We only sketch the descriptions, relying on Figure~\ref{fig:KMW3d} for some of the meaning (more details can be found in~\cite{SantosZiegler}).
\begin{enumerate}
\item[(a)] First slice $c S(p,q)$ into $c$ layers by the $c-1$ horizontal planes at integer heights (Figure~\ref{fig:KMW3d}.a).
\item[(b)]  Then divide each layer into a number of triangular prisms  in one of two ways: the $k$-th layer can be divided into either $2k-1$ prisms with axis in the $X$-direction or $2(c-k)-1$ prisms with axis in the $Y$-direction (Figure~\ref{fig:KMW3d}.b). These prisms have a rectangular facet in one of the horizontal lattice planes and a single edge in the other. We call this edge the \emph{tip} of the prism.  
\item[(c)]  Each such prism can be unimodularly triangulated using joins of the primitive segments in the tip and monotone lattice paths of length $q$ in the opposite facet (Figure~\ref{fig:KMW3d}.c). These paths have to be chosen compatible to one another (they can all be the same one) and apart of these joins the triangulation of the prism only contains pyramids, with vertices at the tip, over unimodular triangles in the opposite facet.
\end{enumerate}

\begin{figure}
\begin{tikzpicture}[y=-1cm,scale=.5]

\definecolor{penColor}{rgb}{0,0,0.6902}
\draw[very thick,join=round,penColor] (14.605,0.9525) -- (16.51,2.8575) -- (20.32,2.8575) -- (18.89125,4.60375) -- (18.415,2.8575) -- (16.98625,4.60375) -- (18.89125,4.60375);
\draw[very thick,join=round,penColor] (14.605,2.2225) -- (16.98625,4.60375);
\draw[very thick,join=round,penColor] (16.51,2.8575) -- (16.98625,4.60375) -- (14.605,2.2225) -- (14.605,0.9525) -- (18.415,0.9525) -- (20.32,2.8575) -- (18.415,2.8575) -- (16.51,0.9525);
\draw[semithick,join=round,penColor] (15.08125,1.42875) -- (18.89125,1.42875) -- (19.3675,1.905) -- (15.5575,1.905) -- (16.03375,2.38125) -- (19.84375,2.38125);
\draw[very thick,join=round,penColor] (11.1125,6.35) -- (14.9225,6.35) -- (18.7325,10.16) -- (14.9225,10.16) -- (11.1125,6.35) -- (11.1125,9.2075) -- (15.875,13.97) -- (18.7325,10.16);
\draw[very thick,join=round,penColor] (15.875,13.97) -- (14.9225,10.16);
\draw[semithick,join=round,penColor] (13.97,9.2075) -- (13.81125,8.41375) -- (14.9225,8.89) -- (15.5575,8.73125) -- (16.1925,8.5725) -- (16.98625,9.04875) -- (16.8275,8.255);
\draw[semithick,join=round,penColor] (11.1125,9.2075) -- (13.97,9.2075) -- (12.065,10.16);
\draw[semithick,join=round,penColor] (13.0175,11.1125) -- (13.97,9.2075) -- (13.97,12.065);
\draw[semithick,join=round,penColor] (14.9225,13.0175) -- (13.97,9.2075);
\draw[semithick,join=round,penColor] (15.875,13.97) -- (13.97,9.2075);
\draw[semithick,join=round,penColor] (13.97,9.2075) -- (12.065,10.16);
\draw[very thick,join=round,penColor] (6.35,7.9375) -- (2.54,7.9375) -- (0.635,6.0325);
\draw[very thick,join=round,penColor] (0.635,7.3025) -- (3.01625,9.68375);
\draw[very thick,join=round,penColor] (3.4925,11.43) -- (0.635,8.5725) -- (0.635,0.9525) -- (12.065,0.9525) -- (3.4925,11.43) -- (0.635,0.9525);
\draw[very thick,join=round,penColor] (2.06375,6.19125) -- (7.77875,6.19125) -- (9.2075,4.445) -- (1.5875,4.445) -- (1.11125,2.69875) -- (10.63625,2.69875);
\draw[very thick,join=round,penColor] (2.06375,6.19125) -- (0.635,4.7625) -- (0.635,3.4925) -- (1.5875,4.445) -- (1.11125,2.69875) -- (0.635,2.2225);
\draw[very thick,join=round,penColor] (3.01625,9.68375) -- (4.92125,9.68375);
\definecolor{textColor}{rgb}{0,0,0.6902}
\path (3,12.54125) node[%
anchor=base west] {%
(a)};
\path (17,5.55625) node[%
anchor=base west] {%
(b)};
\path (12,13.17625) node[%
anchor=base west] {%
(c)};

\filldraw[semithick,black] (11.1125,9.2075) circle (0.15875cm);
\filldraw[semithick,black] (12.065,7.3025) circle (0.15875cm);
\filldraw[semithick,black] (15.875,13.97) circle (0.15875cm);
\filldraw[semithick,black] (14.9225,10.16) circle (0.15875cm);
\filldraw[semithick,black] (18.7325,10.16) circle (0.15875cm);
\filldraw[semithick,black] (17.78,9.2075) circle (0.15875cm);
\filldraw[semithick,black] (16.8275,8.255) circle (0.15875cm);
\filldraw[semithick,black] (15.875,7.3025) circle (0.15875cm);
\filldraw[semithick,black] (14.9225,6.35) circle (0.15875cm);
\filldraw[semithick,black] (11.1125,6.35) circle (0.15875cm);
\filldraw[semithick,black] (13.0175,8.255) circle (0.15875cm);
\filldraw[semithick,black] (13.97,9.2075) circle (0.15875cm);
\filldraw[semithick,black] (12.065,10.16) circle (0.15875cm);
\filldraw[semithick,black] (13.0175,11.1125) circle (0.15875cm);
\filldraw[semithick,black] (13.97,12.065) circle (0.15875cm);
\filldraw[semithick,black] (14.9225,13.0175) circle (0.15875cm);
\filldraw[semithick,black] (15.08125,7.14375) circle (0.15875cm);
\filldraw[semithick,black] (11.90625,6.50875) circle (0.15875cm);
\filldraw[semithick,black] (13.6525,6.82625) circle (0.15875cm);
\filldraw[semithick,black] (13.0175,6.985) circle (0.15875cm);
\filldraw[semithick,black] (14.33407,6.71407) circle (0.15875cm);
\filldraw[semithick,black] (16.03375,8.09625) circle (0.15875cm);
\filldraw[semithick,black] (12.85875,7.46125) circle (0.15875cm);
\filldraw[semithick,black] (14.605,7.77875) circle (0.15875cm);
\filldraw[semithick,black] (13.97,7.9375) circle (0.15875cm);
\filldraw[semithick,black] (15.28657,7.66657) circle (0.15875cm);
\filldraw[semithick,black] (16.98625,9.04875) circle (0.15875cm);
\filldraw[semithick,black] (13.81125,8.41375) circle (0.15875cm);
\filldraw[semithick,black] (15.5575,8.73125) circle (0.15875cm);
\filldraw[semithick,black] (14.9225,8.89) circle (0.15875cm);
\filldraw[semithick,black] (16.23907,8.61907) circle (0.15875cm);
\filldraw[semithick,black] (17.93875,10.00125) circle (0.15875cm);
\filldraw[semithick,black] (14.76375,9.36625) circle (0.15875cm);
\filldraw[semithick,black] (16.51,9.68375) circle (0.15875cm);
\filldraw[semithick,black] (15.875,9.8425) circle (0.15875cm);
\filldraw[semithick,black] (17.19157,9.57157) circle (0.15875cm);

\end{tikzpicture}%
 \caption{The three steps in the decomposition of a dilated three-dimensional empty simplex.}
\label{fig:KMW3d}
\end{figure}
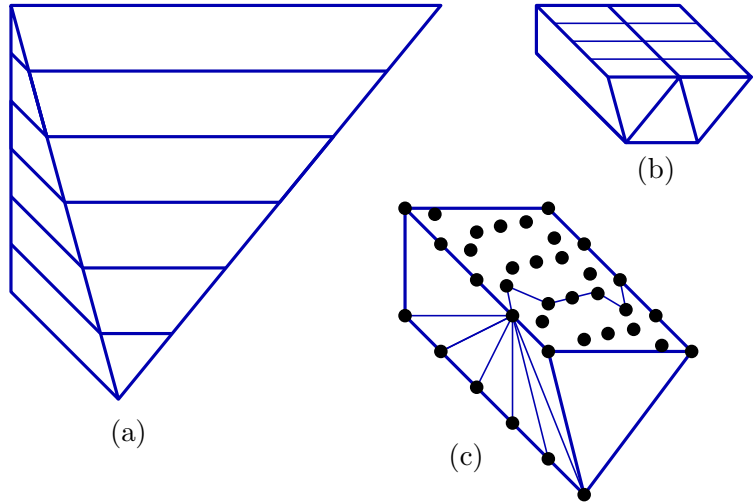

With this technique we see that:

\begin{corollary}[\protect{\cite[Cor. 4.2]{SantosZiegler}}]
For every lattice three-polytope $P$ and every dilation factor $c\ge 2$, $cP$ can be dissected into unimodular simplices with disjoint interiors.
\end{corollary}

\begin{proof}
Consider any full triangulation $\T$ of $P$, and its dilation $c\T$ into dilated empty simplices. Divide every dilated simplex in layers, then prisms, then simplices, as explained above.
\end{proof}

So, the only thing to take care of is how to make all these dissections agree on the common boundary facets of adjacent dilated simplices, adjacent layers within a dilated simplex, and adjacent prisms within a layer.
One way to try to address this issue is to triangulate each dilated simplex $c\Delta$ so that its boundary receives the \emph{standard} triangulation, the one obtained slicing the boundary by all the lattice lines parallel to edges of $c\Delta$ (see Figure~\ref{fig:quasistandard}). 
With this in mind, the following statements on dilated simplices guarantee we can find such desired triangulations of the adjacent components as described above.

\begin{theorem}
\label{thm:dilated3simplices}
\label{thm:kkms3d}
If $\Delta(p,q)$ is an empty three-simplex, then:
\begin{compactenum}
\item $2\Delta(p,q)$ has a unimodular triangulation if and only if $p=\pm 1 \pmod q$. In this case, the triangulation can be made unimodular and with standard boundary~\cite[Cor.~3.5, Prop.~3.6]{SantosZiegler}
\item For every $c\ge 4$, $c\Delta$ has a unimodular triangulation~(\cite[Cor.~4.5]{SantosZiegler}).
\item For every $c$ that can be written as a sum of composite numbers (that is, every $c\ge 4$, other than $c=5,7,11$), $c\Delta$ has a unimodular triangulation with standard boundary~(\cite{KantorSarkaria} for $c=4$, \cite[Thm.~5.1]{SantosZiegler} for the rest).
\item For every $c\ge 7,$ $c\Delta$ has a unimodular triangulation in which  the boundary is triangulated with the ``quasi-standard'' triangulation obtained from the standard one by flipping the diagonals incident to edges of $c\Delta$ and at lattice distance three from the vertices (see Figure~\ref{fig:quasistandard}~\cite{SantosZiegler}).
\end{compactenum}
\end{theorem}

\begin{figure}[htb]
\begin{tikzpicture}[y=-1cm, scale=.9]

\draw[black] (2,8) -- (-3,8) -- (-0.5,3.5) -- cycle;
\draw[black] (1,8) -- (1.5,7.1) -- (-2.5,7.1) -- (-2,8) -- (0,4.4) -- (-1,4.4);
\draw[black] (-2,6.2) -- (-1,8) -- (0.5,5.3) -- (-1.5,5.3) -- (0,8) -- (1,6.2) -- cycle;
\draw[black] (-1,4.4) -- (1,8);
\draw[black] (6.5,2.7) -- (3,9) -- (10,9) -- cycle;
\draw[black] (3.5,8.1) -- (9.5,8.1) -- (9,9) -- (6,3.6) -- (7,3.6) -- (4,9) -- cycle;
\draw[black] (4,7.2) -- (9,7.2);
\draw[black] (8,9) -- (5.5,4.5);
\draw[black] (7.5,4.5) -- (5,9);
\draw[black] (4,7.2) -- (5,9) -- (6.5,8.1) -- (8,9) -- (9,7.2) -- (7.5,6.3) -- (7.5,4.5) -- (5.5,4.5) -- (5.5,6.3) -- cycle;
\draw[black] (7.5,6.3) -- (5.5,6.3) -- (6.5,8.1) -- cycle;
\draw[black] (5,7.2) -- (5.5,8.1);
\draw[black] (6.5,8.1) -- (7,9);
\draw[black] (8,7.2) -- (7.5,8.1);
\draw[black] (5,5.4) -- (5.5,6.3) -- (4.5,6.3);
\draw[black] (8.5,6.3) -- (7.5,6.3) -- (8,5.4);
\draw[black] (6,9) -- (6.5,8.1);
\draw[black] (6,5.4) -- (7,5.4);

\end{tikzpicture}%
 \caption{The standard triangulation of a dilated triangle (left) and
  the quasi-standard one used for the case $c=7$.}
\label{fig:quasistandard}
\end{figure}
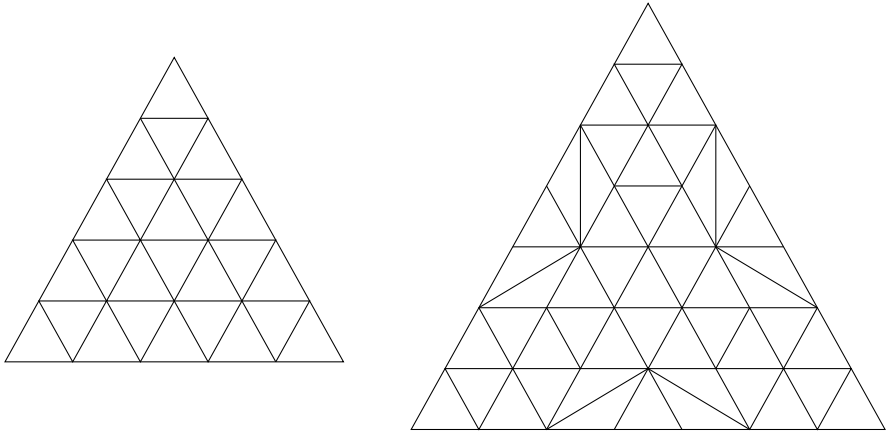

Hibi, Higashitani, and Yoshida~\cite{hibi2017existence} generalize part (1) of Theorem~\ref{thm:dilated3simplices} as follows: it was shown by Batyrev and Hofscheier in \cite{batyrev2010generalization} that empty $(2k-1)$-simplices with unimodular facets are Cayley sums of $k$ unit segments if and only if their $k$-th dilation does not have interior lattice points. Hibi, Higashitani and Yoshida prove that, among those, the $k$-th dilation has a unimodular triangulation if and only if  they are of the following form:
\[
\conv\{O,e_1,\dots,e_{2k-2},v\},
\]
with $v=e_1+\dots+ e_{k-1} -e_{k}-\dots - e_{2k-2} + q e_{2k-1}$ for some $q\in \N$.

\subsection[Unimodular refinements]{Canonical triangulation of a dilated simplex} 
\label{sec:kmwhypersimplex}

In this section we show that given  a unimodular triangulation $\T$ of
a polytope $P$ another unimodular triangulation of $cP$ can be
derived, preserving regularity
(Theorem~\ref{thm:unimodular-dilations}) and flagness
(Corollary~\ref{cor:quadratic-dilations}). The first result can by now
be considered ``folklore'' (it is sometimes attributed to Santos,
1996, unpublished). The second result is new.

The proofs of both theorems start by noting that dilating every
simplex in $\T$ yields a triangulation $c\T$ of $cP$ into dilations of
unimodular simplices, and then use the canonical subdivisions of
dilated unimodular simplices, as discussed in
section~\ref{sec:hyperplanes}. That is, the dilated standard simplex
$c\Delta^d$ is realized as $\{ \vx \in \R_{\ge 0}^{d+1} \ : \ \sum x_i
= c \}$, and sliced along the hyperplanes parallel to the facets: $x_i
= k$ for $i=1, \ldots, d+1$ and $k = 1, \ldots c-1$. The cells of this
subdivision are hypersimplices.
The canonical weights $\vomega_\vm := \sum m_i^2$ show that this
subdivision is regular, and the weights restrict to faces as well.

The naturality of this subdivision and its weights allows one to patch
simplices together. That is, if we subdivide each of the cells of
$c\T$ canonically, they agree along their boundaries, and we get a
subdivision of $P$ which is regular if $\T$ was regular.

\begin{theorem}
\label{thm:unimodular-dilations}
If $P$ has a (regular) unimodular triangulation $\T$ then, for every
positive integer $c$, its dilation $cP$ has one too.
\end{theorem}

\begin{proof}
  As in Theorem \ref{thm:hyperplanes}, all cells (hypersimplices) of
  the canonical subdivision of $c\T$ are compressed.  Hence, pulling all the
  lattice points of $cP$ in any order yields a unimodular
  triangulation $\T'$. This triangulation is a
  regular refinement.
Hence, regularity of $\T$ implies regularity of $\T'$
(cf.~\cite[Lemmas Lemma 2.3.16 and  4.3.12]{deLoeraRambauSantos}, 
or see more details in the proof of Corollary~\ref{cor:quadratic-dilations}).
\end{proof}

To guarantee flagness, instead of using pulling refinements we dice
with respect to a larger set of hyperplanes, namely those based on the
root system of type $\rootA$. Remember that if $P$ is a lattice
polytope of type $\rootA$ then dicing $P$ according to all lattice
hyperplanes with normals in the root system gives a quadratic
triangulation of $P$ (Theorem~\ref{thm:rootA}). 

To apply this to an arbitrary simplex $\Delta$ we first need to show a
way to map a $d$-simplex to a simplex of type $\rootA$. This is
canonical, once an ordering of $\Delta$'s vertices  is prescribed.
Thus we introduce the concept of an ordered simplex, and the canonical
triangulation of its dilation. These concepts will be crucial in the
proof of the KMW Theorem (and were used in the previous proofs of it,
see~\cite[Section 3.A]{BG}).

\begin{definition}
\label{def:ordered_simplex}
\label{def:typeA_canonical_triangulation} 
An \emph{ordered simplex} is a simplex,
$\Delta=\conv\{\va_0,\dots,\va_d\}$ that comes with an ordering of
its vertices, $\va_0,\dots,\va_d$.
Let $\Delta$ be an ordered $d$-simplex and consider the affine isomorphism sending
$\Delta$ to the standard simplex $\Delta_\rootA^d$ of type $\rootA$
via
\[
\va_i\mapsto \sum_{j=1}^{i+1} \ve_j \in \R^{d+1}.
\]
The $\rootA$-\emph{canonical triangulation} of $c\Delta$ 
is the pull back under this map of the type $\rootA$ dicing
triangulation of $c \Delta_\rootA^d$.
\end{definition}

Observe that since the $\rootA$-dicing is unimodular, all the
simplices in the canonical triangulation of $c\Delta$ have the same
volume as $\Delta$.

The following lemma makes  what the canonical triangulation is more
explicit (as a dicing triangulation) by expressing it in barycentric
coordinates with respect to $\Delta$. Recall that the barycentric
coordinates of a point $\vx\in \Delta$ with respect to the vertices
$\va_0,\dots,\va_d$ of a simplex $\Delta$ are the unique vector
$(x_0,\dots,x_d)$ with $\sum x_i=1$ such that
\[
\vx =\sum_{i=0}^d x_i \va_i.
\]

In particular, for each subset $I\subseteq\{0,\dots,d\}$ the
hyperplanes defined by setting $\sum_{i\in I} x_i$ equal to a constant
are parallel to the following two complementary faces of $\Delta$:
\[
\conv\{\va_i: i\in I\},\qquad \text{and}\qquad \conv\{\va_i: i\not\in
I\}.
\]

\begin{lemma}
\label{lemma:Acanonical}
Let $\Delta=\conv\{\vv_0,\dots,\vv_d\}$ be an ordered simplex and $c$
a positive integer.
Then the canonical triangulation of $\Delta$ is the dicing
triangulation with respect to the families of lattice hyperplanes
defined by making any sum $x_i+\cdots+x_j$ of consecutive barycentric
coordinates constant.
\end{lemma}

\begin{proof}
Since an affine isomorphism from one simplex to another preserves
bary\-cen\-tric coordinates, there is no loss of generality in
assuming that $\Delta$ is the standard simplex
$\Delta_\rootA^d$. 
We only need to show that the hyperplanes orthogonal to the roots are
indeed defined by making a sum of consecutive barycentric coordinates
constant. For this, observe that if $(x_0,\dots,x_d)$ are the
barycentric coordinates of $\vx$ with respect to $\Delta_\rootA^d$
then

\[
\vx=\sum_{j=1}^d x_j\left(\sum_{i=1}^{j+1} \ve_i \right) =
\sum_{i=1}^{d+1} \ve_i  \left(\sum_{j=i-1}^d  x_j\right),
\]
so that
\begin{align*}
\langle \vx, \ve_j-\ve_i\rangle&=\sum_{k=i-1}^{j-2}  x_k  , &\forall
1\le i<j\le d+1.
\end{align*}
\end{proof}

With this we can conclude the following.

\begin{theorem}
\label{thm:canonical-dilation}
Let $\T$ be a lattice triangulation of a polytope $P$, let $c$ be a
positive integer, and let $c\T$ denote the dilation of $\T$, that is,
the following triangulation of $cP$: 
\[
c\T:=\{c\Delta : \Delta\in \T\}\,.
\]

Consider a total order in the vertices of $\T$, so that every
$\Delta\in \T$ is regarded as an ordered simplex. Call $\T'$ the
triangulation of $cP$ obtained refining each $c\Delta$ to its
canonical triangulation. Then: 
\begin{enumerate}
\item $\T'$ is indeed a triangulation. That is, common faces of
  different simplices in $c\T$ get the same refinement.
\item If $\T$ his unimodular, then $\T'$ is unimodular.
\item If $\T$ is regular, then $\T'$ is regular.
\item If $\T$ is flag, then $\T'$ is flag.
\end{enumerate}
\end{theorem}

\begin{proof}
The canonical triangulations  agree on  common faces of simplices of
$c\T$ since they only depend on the ordering of the vertices in those
faces. Therefore, this procedure indeed gives a lattice triangulation
$\T'$ of $cP$. Since the canonical triangulation preserves volumes, it
preserves unimodularity. We will show that regularity and flagness are
also preserved.

For regularity we use the idea of regular refinements~\cite[Lemma
2.3.16]{deLoeraRambauSantos}. Assume
that $\T_0$ is a regular subdivision of a point configuration $\configuration$, given by a weight vector $\vomega_0$. If  each cell
of $\T_0$ is refined using a common weight vector  $\vomega \in
\R^\configuration$, then the refinements of all cells agree on their
intersections, so that we get a subdivision of
$\configuration$. Moreover, this subdivision is regular, and could
have been obtained directly using the weight vector $\vomega_0 +
\epsilon \vomega$, for a sufficiently small $\epsilon>0$.

We apply this procedure $\binom{n+1}{2}$ times to the
subdivision $c\T$, where $n$ is the number of vertices of
$\T$, as follows. Let $\va_1,\dots,\va_n$ be the ordered sequence
of vertices of $\T$. For each interval $[i,j]\subseteq[1,n]$
(including the case $i=j$), let $\phi_{i,j}$ be the function that is
zero on $c\va_k$ if $k\in[i,j]$, one on $c\va_k$  if $k\notin[i,j]$,
and linear on each simplex of $c\T$. Consider the weight function
$\vomega_{i,j}\in \R^{cP\cap \Z^d}$ that restricts $\phi_{i,j}^2$ 
to all lattice points in $cP$. Clearly, the regular refinement of
$c\T$ via this weight function dices each simplex $c(\conv\{\va_i:i\in
S\})$ by the lattice hyperplanes parallel to $\conv\{\va_i:i\in
S\cap[i,j]\}$ and $\conv\{\va_i:i\in S\setminus[i,j]\}$. In
each simplex of $c\T$ this is one of the  hyperplanes we want to use
for the canonical triangulation of that simplex. Performing these
regular refinements on $c\T$ for each such weight function
$\vomega_{i,j}$, in any order, gives a dicing of each simplex by more
and more families of hyperplanes, eventually all hyperplanes in 
Lemma~\ref{lemma:Acanonical}. Hence, the resulting triangulation is
$\T'$, proving that $\T'$ is regular.

Flagness of $\T'$ is based solely on the fact that $\T$ is flag and
that the way $\T'$'s  refines each simplex of $c\T$ is also flag (by
Theorem~\ref{thm:rootA}). Indeed, let $N \subset cP \cap \Z^d$ and suppose
that every pair of points in $N$ form an edge in $\T'$. 
To each point $\vn \in N$ associate its  carrier simplex
$cF(\vn) \in c\T$, and we call $S$ the union of all vertices of all the
carriers $F(\vn)$. Observe that for $\vn_1,\vn_2 \in N$ we have that 
$cF(\vn_1)\cup cF(\vn_2)$ equals the minimal face of $c\T$ containing the edge
$\{\vn_1,\vn_2\}\in \T'$ (here we are using the fact that $\T'$ refines $c\T$). In particular, 
every two points in $S$ form an edge of $\T$ because either they lie
in the same $F(\vn)$, which is a simplex, 
or they lie in $F(\vn_1)$ and $F(\vn_2)$ for two points
of $N$. Since $\T$ is flag, $S$ is
a face of $\T$, which implies that $N$ is contained in the convex hull of the face $cS\in c\T$.
Since $\T'$ refines $cS$ in a flag manner and $N$ is a clique in $\T'$, $N$ is a face in $\T'$.
\end{proof}

\begin{corollary}
\label{cor:quadratic-dilations}
If $P$ has a quadratic triangulation $\T$, then so does its dilation
$cP$ for every positive integer $c$. 
\end{corollary}

With the canonical triangulations at our disposal, we can now discuss
the example advertised at the beginning of the section. This example
is inspired by~\cite{MustataPayne} and appears
in~\cite{CoxHaaseHibiHigashitani}.

\begin{example} \label{kmw:eg:c+1}
  Let $h$ and $k$ be positive integers and set $d:=hk-1$. In
  $\R^{d+1}$ consider the vector $v:=\frac1k(1,\ldots,1)$, the
  lattice $\Lambda := \Z^{d+1} + \Z v$ and the $d$-polytope
  $\Delta:=\conv\{e_0,\ldots,e_d\}$. Then $h\Delta$ has a quadratic
  triangulation joining $v$ to the canonical triangulation of
  $h\partial\Delta$.

  We claim that certain higher multiples of $\Delta$ are not normal
  and thus do not admit unimodular triangulations. To see this,
  consider the homomorphism $\phi \colon \Lambda \to \Z/k$ given by $x
  \mapsto \overline{kx_0}$. If $a$ and $b$ are positive integers with
  $b<h$, then the $\Lambda$-points in $(ah+b)\Delta$ map to $\{\bar 0,
  \bar 1, \ldots, \bar a\}$. Therefore, we can only obtain $\{\bar 0,
  \bar 1, \ldots, \overline{2a}\}$ from the sum of two such points. On
  the other hand, if $2b \ge h$, then $2(ah+b)\Delta$ contains a
  $\Lambda$-point with image $\overline{2a+1}$, witnessing
  non-normality if $2a+1<k$. The parameters $k=4$, $h=2$ yield a
  seven-simplex so that $2\Delta$ has a quadratic
  triangulation but $3\Delta$ is not normal.
\end{example}

\begin{example}
\label{kmw:eg:c1+c2}
Concerning the IDP property, 
Laso\'n and Micha{\l}ek~\cite[Sect. 3.1--3.3]{LasonMichalek} show the following: 
let $T(k)$ be the simplex in $\R^{2k-1}$ obtained as the convex hull
of the origin, the first $2k-2$ coordinate unit vectors $e_1,\dots,
e_{2k-2}$, and the point $v:=e_1+\dots+e_{k-1}
-e_{k}-\dots-e_{2k-2}+(k+2)e_{2k-1}$ in $\R^{2k-1}$. (Equivalently,
$T(k)$ is obtained from a unimodular $(k,k)$ circuit in $\R^{2k-2}$ by
lifting one point to height $k+2$ and the rest to height $0$). Let
$P(k)$ be the Minkowski sum of $T(k)$ and the segment
$[O,e_{2k-1}]$. Then,
$c P(k)$ is IDP if and only if $c\ge k$ or $c$ does not divide $k$. 
For example, $2P(25)$ and $3P(25)$ are IDP, but $5P(25)$ is not, for the $49$-dimensional polytope $P(25)$.
\end{example}

\subsection{Reducing the volume of simplices in the dilation} \label{sec:kmwreduction}

The canonical triangulation divides $c\Delta$ into simplices of the
same volume as $\Delta$. We now want to improve this, and triangulate
$c\Delta$ into simplices of volume strictly smaller than
$\Delta$ (assuming $\Delta$ is not unimodular).
This cannot be done for every $c$, but we prove here that it can
always be done for some $c\le d$ and all its multiples.

For this we introduce the concept of box points.

We start with a seemingly unrelated lemma about the Cayley sum
of two dilated polytopes of type $\rootA$.
Remember that if $P_1$ and $P_2$ are lattice polytopes in a lattice
$\Lambda\subset\R^d$ then the \emph{Cayley sum} of $P_1$ and $P_2$ is
the  lattice polytope $\conv( P_1\times \{ 0 \} \cup P_{2}\times \{ 1
\} )$ in the lattice $\Lambda \times \Z \subset \R^{d+1}$.

As usual,  the canonical triangulation of a polytope of type $\rootA$ is the
one obtained by slicing by all the lattice hyperplanes normal to the
roots. Remember that this triangulation is regular and unimodular
(Theorem~\ref{thm:rootA}).

\begin{lemma}
\label{lemma:cayley}
If $P_1$ and $P_2$ are lattice polytopes of type $\rootA$, with
canonical triangulations $\T_1$ and $\T_{2}$, then:
\begin{compactenum}
\item There is a regular triangulation of the Cayley sum $\conv(
  P_1\times \{ 0 \} \cup P_{2}\times \{ 1 \} )$ that restricts to
  $\T_1$ and $\T_2$ on $P_1$ and $P_2$.
\item Any such triangulation (regular or not) is unimodular.
\end{compactenum}
\end{lemma}

\begin{proof}
The canonical triangulation of a type-$\rootA$ polytope is regular,
so let $\vomega_1$ and $ \vomega_2$ be weight vectors producing $\T_1$
and $\T_2$ in $P_1$ and $P_2$, respectively.
Since $P_1\times \{ 0 \}$ and  $P_{2}\times \{ 1 \} $ are faces of the
Cayley sum, using the same weight vectors (that is, using $\vomega_1$
in $P_1\times \{ 0 \}$ and $ \vomega_2$ in $P_{2}\times \{ 1 \}$)
yields a regular subdivision that restricts to $\T_1$ and $\T_2$ on
$P_1$ and $P_2$.
If that subdivision is not a triangulation, we do pulling refinements,
which preserve regularity.

For part (2), observe that every cell in such a triangulation $\T$
will be a join of a face $F_1$ of $\T_1$ and a face $F_2$ of
$\T_2$. It is convenient to think of the root system $\rootA_d$ as
consisting of the standard basis vectors $\pm \ve_i$ plus the
differences $\pm (\ve_i-\ve_j)$. This is done by forgetting the
coordinate $d+1$ in the usual definition of $\rootA_{d}$ as embedded
in $\R^{d+1}$. In this description, the $\rootA_d$ dicing of $\R^d$
refines the tiling by translations of the unit cube, triangulating every
 lattice unit cube by (translations of) the order triangulation of $[0,1]^d$.

In particular, if $k$ and $l$ are the dimensions of $F_1$ and $F_2$
there is no loss of generality in assuming that 
\[
F_1=\conv\{\va_0,\va_1\dots, \va_k\}\times \{0\}, \qquad
F_2=\conv\{\vb_0,\vb_1\dots, \vb_l\}\times \{1\},
\]
where the $\va_i$'s and the $\vb_j$'s are sequences of 0/1 vectors in
$\R^d$ with increasing supports. Moreover, by subtracting $\va_0$ from
every $\va_i$ and $\vb_0$ from every $\vb_j$ we can assume
$\va_0=\vb_0=0$.  Observe that the $\va_i$ and $\vb_i$ are in
transversal spaces.  Then, the volume of the join of $F_1$ and $F_2$
equals the volume of $F:=\conv\{0=\va_0=\vb_0,\va_1\dots, \va_k,
\vb_1\dots, \vb_l\}\subset \R^d$. We are going to show that this
volume is always one (or zero). Without loss of generality we assume $F$ to be
  full-dimensional. If it is not, we extend the sequence of $\va_i$'s
  until it is. 

We use induction on the dimension $d$. In particular, we assume
  that no $\va_i$ and  $\va_{i+1}$ differ in a single
  coordinate, because otherwise we can project along that coordinate
  and get the result by inductive hypothesis. The same happens if some
  $\vb_j$ and $\vb_{j+1}$ differ by a single coordinate.
 That implies that both $k$ and $l$ are at most $d/2$, hence they
  are both equal to $d/2$ because $k+l=d$. But then we must have
  $\va_k=\vb_l=(1,\dots,1)$, which implies that the volume of $F$ is
  zero (except, of course, in the base case for the induction, which is
  $k=l=0$ and produces volume one).\qedhere
\end{proof}

It is worth mentioning that Lemma~\ref{lemma:cayley} is not true for
the Cayley sum of three or more type-$\rootA$ polytopes, even in
dimension three. Indeed, the edges parallel to the vectors $(1,1,0)$,
$(1,0,1)$ and $(0,1,1)$ are type-$\rootA$ polytopes and their Cayley
sum, which equals their join, is a non-unimodular simplex.

\medskip

Our method for refining dilations will be based on combining
Lemma~\ref{lemma:cayley} with the concept of an $\rootA$-canonical
triangulation of a dilated ordered simplex 
(Definition~\ref{def:typeA_canonical_triangulation}). Let us explain
how.

Let $\Delta=\conv(\vb_0,\dots,\vb_d)$ be a non-unimodular lattice simplex, of
volume $V$ with respect to a lattice $\Lambda$. 
Let $L_\Delta$ be the lattice spanned by the vertices of $\Delta$. We take
  $\vec L_\Delta:=L_\Delta$ to denote the linear lattice parallel to $L_\Delta$,
  so $\Lambda: \vec L_\Delta = V$.%
\footnote{%
  To simplify the exposition, in the rest of the paper we assume
  (w.~l.~o.~g.)  that $\Lambda$  contains the origin, so that all the
  dilations $c\Delta$ are lattice polytopes with respect to the same lattice $\Lambda$.
  We do not require, however, the affine lattice $L_\Delta$ to contain the origin. 
}
We use barycentric coordinates with respect to $\Delta$, and dilated
barycentric coordinates for a dilation $c\Delta$. In the latter,
points in $c\Delta$ are written as convex combinations of $c\vv_0,
\dots, c\vv_d$.

A \emph{box point} for $\Delta$ is simply an element of the quotient 
$\Lambda/\vec L_\Delta$, where $\vec L_\Delta$ is the 
linear lattice parallel to $L_\Delta$.
We normally represent a box point in the $\Delta$-barycentric
coordinates of any of its representatives, as
$\m=(m_0,\dots,m_{d})$. Quotienting by $\vec L_\Delta$ then means that
we are interested only in $(\{m_0\},\dots, \{m_d\})$, where
$\{x\}:=x-\lfloor x \rfloor$ denotes the fractional part of a real
number $x$.  The \emph{height} $h(\m)$ of a box point $\m$ is the number
$\sum_i \{m_i\}$. It is an integer between one and $V-1$
and coincides with the smallest integer $h$ for which $h\Delta$
contains a point of $\m + \Lambda$.

These concepts extend to lower dimensional faces. For a face $F$ 
 $L_F$ denotes the lattice spanned by the vertices of $F$ and 
$\Lambda_F$ is the intersection of $\Lambda$ with the linear space parallel
to $F$. A box point for $F$ is then an element of $\Lambda_F/\vec L_F$, and is
represented by an $\m$ in barycentric coordinates with respect to
$F$. Observe that if $F$ is a face of $\Delta$ then there is a natural
inclusion $\Lambda_F/\vec L_F \le \Lambda/(\vec L_\Delta)$ so that a box point for
$F$ is also a box point for $\Delta$. Coordinates for $\m$ are extended
from $F$ to $\Delta$ by putting zeroes in the coordinates of the
vertices in $\Delta\setminus F$. Conversely, a box point of $\Delta$
is a box point for $F$ if and only if those particular coordinates are zero.

We call \emph{carrier} of a box point $\m$ of $\Delta$ the (unique) minimal face 
$F$ of $\Delta$ for which $\m$ is a box point. A  box point of $\Delta$ is 
\emph{non-degenerate} if its carrier is $\Delta$ itself. That is, if none of its
$\Delta$-barycentric coordinates is integral. 

\medskip

Given a box point $\m$ of a simplex $\Delta$ and a certain positive integer $c$,
we aim to triangulate $c\Delta$ into simplices of volume strictly
smaller than that of $\Delta$, using only the points of $L_\Delta \cup
(\m + L_\Delta )$. Of course, this is impossible if $c<h(\m)$. Although it can be done
for every $c\ge h(\m)$

We offer
two versions of this procedure. The first works for every $c\ge c_0$
but is less symmetric and  complicates the proof of
regularity. The second one is somehow simpler but requires $c$ to be a
multiple of $(d+1)!$ (or, rather, a multiple of the number of non-zero
coordinates in $\m$. However, since we will later apply the procedure to
several box points at the same time, having $c$ be a multiple of $(d+1)!$
is convenient).

The structure of $c\Delta \cap (\m+L_\Delta)$ is straightforward: it
is a translation of the points of $L_\Delta$ in $(c-c_0)\Delta$ (see
Figure~\ref{fig:KMW-lattice}).

\begin{figure}
\begin{tikzpicture}[y=-1cm,scale=.5]

\draw[join=round,black] (8.89,16.34913) -- (21.59,16.34913) -- (15.24,5.23875) -- cycle;
\draw[join=round,dashed,black] (11.43,16.34913) -- (16.51,7.46125) -- (13.97,7.46125) -- (19.05,16.34913) -- (20.32,14.12663) -- (10.16,14.12663) -- (11.43,16.34913) -- (13.97,16.34913) -- (17.78,9.68375) -- (12.7,9.68375) -- (16.51,16.34913) -- (19.05,11.90413) -- (11.43,11.90413) -- (13.97,16.34913);

\path[draw=black,fill=black!50] (15.5575,8.41375) circle (0.15875cm);
\path[draw=black,fill=black!50] (14.2875,10.63625) circle (0.15875cm);
\path[draw=black,fill=black!50] (13.0175,12.85875) circle (0.15875cm);
\path[draw=black,fill=black!50] (11.7475,15.08125) circle (0.15875cm);
\path[draw=black,fill=black!50] (14.2875,15.08125) circle (0.15875cm);
\path[draw=black,fill=black!50] (16.8275,15.08125) circle (0.15875cm);
\path[draw=black,fill=black!50] (19.3675,15.08125) circle (0.15875cm);
\path[draw=black,fill=black!50] (18.0975,12.85875) circle (0.15875cm);
\path[draw=black,fill=black!50] (16.8275,10.63625) circle (0.15875cm);
\path[draw=black,fill=black!50] (15.5575,12.85875) circle (0.15875cm);

\filldraw[black] (10.16,14.12875) circle (0.15875cm);
\filldraw[black] (8.89,16.35125) circle (0.15875cm);
\filldraw[black] (11.43,16.35125) circle (0.15875cm);
\filldraw[black] (12.7,14.12875) circle (0.15875cm);
\filldraw[black] (15.24,14.12875) circle (0.15875cm);
\filldraw[black] (13.97,16.35125) circle (0.15875cm);
\filldraw[black] (16.51,16.35125) circle (0.15875cm);
\filldraw[black] (17.78,14.12875) circle (0.15875cm);
\filldraw[black] (20.32,14.12875) circle (0.15875cm);
\filldraw[black] (19.05,16.35125) circle (0.15875cm);
\filldraw[black] (21.59,16.35125) circle (0.15875cm);
\filldraw[black] (15.24,5.23875) circle (0.15875cm);
\filldraw[black] (13.97,7.46125) circle (0.15875cm);
\filldraw[black] (16.51,7.46125) circle (0.15875cm);
\filldraw[black] (12.7,9.68375) circle (0.15875cm);
\filldraw[black] (15.24,9.68375) circle (0.15875cm);
\filldraw[black] (17.78,9.68375) circle (0.15875cm);
\filldraw[black] (16.51,11.90625) circle (0.15875cm);
\filldraw[black] (13.97,11.90625) circle (0.15875cm);
\filldraw[black] (11.43,11.90625) circle (0.15875cm);
\filldraw[black] (19.05,11.90625) circle (0.15875cm);

\end{tikzpicture}%
 \caption{The points in $c\Delta \cap (\m+L_\Delta)$. Here $c=5$ and $c_0=2$}
\label{fig:KMW-lattice}
\end{figure}
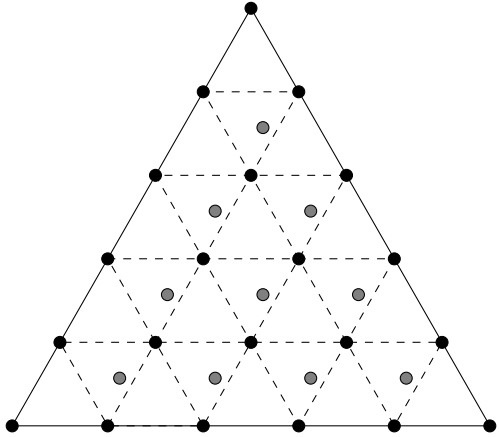

Assuming $\Delta$ is an ordered simplex allows us to speak of the
canonical triangulation of $c\Delta$, and of $cF$ for every face $F$
of $\Delta$.

Let $F$ be a face in a triangulation $\T$ and let $\m$ be a box point
of $F$. We say that $\m$ is \emph{non-degenerate} for $F$ if it is not
a box point of a proper face of $F$. Equivalently, $\m$ is
non-degenerate if all the $F$-coordinates of $\m$ are
fractional. Then, the full-dimensional simplices of $\T$ for which
$\m$ is a box point are precisely the star of $F$.  Recall that the
(closed) star of a face $F$ in a simplicial complex $\T$ is the set
$\overline\Star(F;\T)$ of all simplices of $\T$ that contain $F$,
together with all their faces. Equivalently:
\[
\overline\Star(F;\T):= \{ F'\in \T : F\cup F' \in \T\}.
\]
By $\partial \Star (F;\T)$ we mean the faces of $\overline\Star(F;\T)$
that do not contain $F$. If $\T$ is a triangulation of a polytope and
$F$ is not contained in the boundary of $P$ then $\partial \Star
(F;\T)$ coincides with the topological boundary of
$\overline\Star(F;\T)$.

As before, a global ordering is assumed on the vertices of $\T$, so
that we can speak of the canonical triangulation of each dilated face
in $c\T$.

\begin{lemma} \label{lemma:kmw:pull}
If $\T$ is a lattice triangulation on an ordered set of vertices
lying in a lattice $\Lambda$, and $F=\{\vv_0,\dots,\vv_k\}$ is a
non-unimodular face with a non-degenerate box point $\m =
(m_0,\dots,m_k) \in L_F\setminus \Lambda_F$, then for every integer
$c\in (k+1)\N$,  $c \cdot \overline{\Star} (F;\T)$ has a  refinement
$\T_\m$ such that: 
\begin{enumerate}
\item The volume of every full-dimensional simplex $\Delta'$ in
  $\T_\m$ is strictly less than the volume of simplex $\Delta$ for
  which $\Delta'\subset c\Delta$. \label{lemma:kmw:pull-i}
\item $\T_\m$ induces the canonical triangulation on the boundary $c
  \cdot \partial \Star (F;\T)$. \label{lemma:kmw:pull-ii} 
\item $\T_\m$ is a regular refinement of $\T$, so if $\T$ is regular
  then $\T_\m$ is regular. In particular, any choice of weights
  inducing the canonical triangulation on $\Lambda \cap c \cdot
  \partial \Star (F;\T)$ has an extension to $\Lambda \cap c \cdot
  \overline \Star (F;\T)$ that induces $\T_\m$ as a regular refinement
  of $\T$.
  \label{lemma:kmw:pull-iii}
\end{enumerate}
\end{lemma}

\begin{proof}
  The idea is to first subdivide each $c\Delta$, $\Delta\in \T$ by
  concentric copies of smaller and smaller dilations of $\Delta$ and
  then use Lemma~\ref{lemma:cayley} to refine those subdivisions. But
  let us explain it in a way that demonstrates the regularity
  properties.

  We start by describing the structure of lattice points in $c\cdot \overline \Star (F;\T)$.
   (Note: these are not all the lattice points; we are
  not claiming our triangulation $\T_\m$ to be full.)  Observe that
  requiring $c$ to be a multiple of $k$ implies that the barycenter of
  $cF$ is in $L_F$. Around it, we have concentric copies $k\partial
  \Star (F;\T)$, $2k \partial \Star (F;\T)$, \dots, $c\cdot \partial
  \Star (F;\T)$, which contain all the lattice points of each
  $c\Delta$ that lie in $L_\Delta$, for each $\Delta\in
  \overline\Star(F;\T)$. Between each pair of consecutive copies of multiples of $k\partial
  \Star (F;\T)$,
   we have translated copies
  of $(k-c_0)\partial \Star (F;\T)$, $(2k-c_0) \partial \Star (F;\T)$,
  \dots, $(c-c_0)\partial \Star (F;\T)$, and those complexes contain
  the lattice points in the distinct $\m + L_\Delta$. The latter are
  not concentric to the barycenter, but concentric to one another and
  displaced by the vector $\m$. For each two consecutive dilated
  copies of $\partial \Star(F;\T)$, corresponding faces form a Cayley
  polytope and these Cayley polytopes form a polyhedral subdivision of
  $c\cdot \overline\Star(F;\T)$ (the innermost copy $(c-c_0)\overline\Star
  (F;\T)$ is stellarly subdivided from the origin into pyramids, which
  are nothing but degenerate cases of Cayley polytopes).

We claim that this polyhedral subdivision $\S$ is the regular
refinement of $c\cdot\overline\Star(F;\T)$ obtained by the following choice
of weights. Choose numbers $\alpha_0=0 <\alpha_1 < \alpha_2< \dots <
\alpha_{2c/k}$, each much bigger than the previous one, and give
height $\alpha_i$ to the points in the $i$-th concentric copy.

We then refine $\S$ into $\T_\m$ via any second choice of weights
$\vomega$. Any such refinement will satisfy the first claimed
property. If $\vomega$ is chosen to induce the canonical triangulation
on $c\cdot\partial \Star(F;\T)$, then $\T_\m$ also satisfies properties (2)
and (3). For (3), observe that our first choice of weights (the
$\alpha$'s) were constant on $c\cdot\partial \Star(F;\T)$. Hence, the
restriction of $\T_\m$ to $c\cdot\partial \Star(F;\T)$ is only governed by
$\vomega$.
\end{proof}

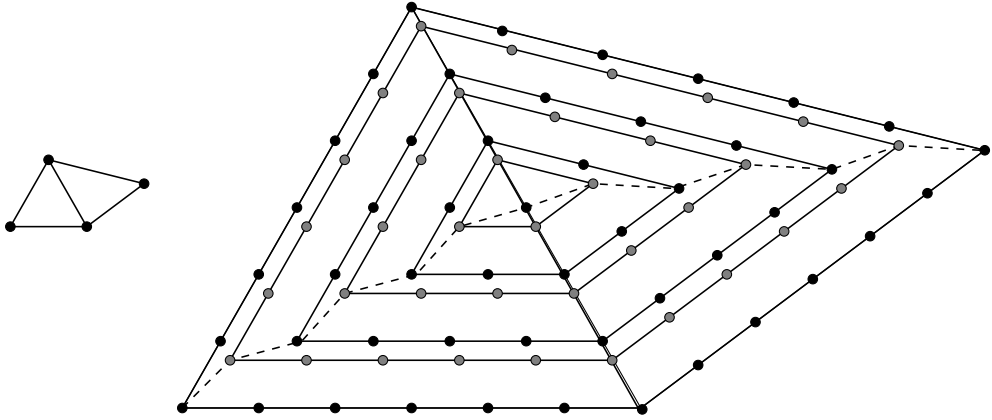
\begin{figure}
\begin{tikzpicture}[y=-1cm,scale=.4]

\draw[join=round,black] (8.89,16.34913) -- (24.13,16.35125) -- (16.51,3.01625) -- cycle;
\draw[join=round,black] (24.13,16.35125) -- (35.56,7.77875) -- (16.51,3.01625);

\draw[semithick,join=round,black] (16.8275,3.65125) -- (10.4775,14.76375) -- (23.1775,14.76375) -- (32.7025,7.62) -- cycle;
\draw[semithick,join=round,black] (16.51,3.01625) -- (8.89,16.35125) -- (24.13,16.35125) -- (35.56,7.77875) -- cycle;
\draw[semithick,join=round,black] (17.78,5.23875) -- (12.7,14.12875) -- (22.86,14.12875) -- (30.48,8.41375) -- cycle;
\draw[semithick,join=round,black] (18.0975,5.87375) -- (14.2875,12.54125) -- (21.9075,12.54125) -- (27.6225,8.255) -- cycle;
\draw[semithick,join=round,black] (19.3675,8.09625) -- (18.0975,10.31875) -- (20.6375,10.31875) -- (22.5425,8.89) -- cycle;
\draw[semithick,join=round,black] (19.05,7.46125) -- (16.51,11.90625) -- (21.59,11.90625) -- (25.4,9.04875) -- cycle;
\draw[semithick,join=round,black] (24.13,16.51) -- (16.51,3.01625);
\draw[semithick,join=round,black] (4.445,8.09625) -- (3.175,10.31875) -- (5.715,10.31875) -- (7.62,8.89) -- cycle;
\draw[semithick,join=round,black] (4.445,8.09625) -- (5.715,10.31875);
\draw[semithick,join=round,dashed,black] (8.89,16.35125) -- (10.4775,14.76375) -- (12.85875,14.12875) -- (14.2875,12.54125) -- (16.66875,11.90625) -- (18.0975,10.31875) -- (20.32,9.68375) -- (22.5425,8.89) -- (25.24125,9.04875) -- (27.6225,8.255) -- (30.48,8.41375) -- (32.7025,7.62) -- (35.40125,7.77875);

\path[draw=black,fill=black!50] (10.4775,14.76375) circle (0.15875cm);
\path[draw=black,fill=black!50] (13.0175,14.76375) circle (0.15875cm);
\path[draw=black,fill=black!50] (15.5575,14.76375) circle (0.15875cm);
\path[draw=black,fill=black!50] (18.0975,14.76375) circle (0.15875cm);
\path[draw=black,fill=black!50] (20.6375,14.76375) circle (0.15875cm);
\path[draw=black,fill=black!50] (15.5575,5.87375) circle (0.15875cm);
\path[draw=black,fill=black!50] (14.2875,8.09625) circle (0.15875cm);
\path[draw=black,fill=black!50] (16.8275,8.09625) circle (0.15875cm);
\path[draw=black,fill=black!50] (13.0175,10.31875) circle (0.15875cm);
\path[draw=black,fill=black!50] (15.5575,10.31875) circle (0.15875cm);
\path[draw=black,fill=black!50] (18.0975,10.31875) circle (0.15875cm);
\path[draw=black,fill=black!50] (16.8275,12.54125) circle (0.15875cm);
\path[draw=black,fill=black!50] (14.2875,12.54125) circle (0.15875cm);
\path[draw=black,fill=black!50] (11.7475,12.54125) circle (0.15875cm);
\path[draw=black,fill=black!50] (19.3675,12.54125) circle (0.15875cm);
\path[draw=black,fill=black!50] (23.1775,14.76375) circle (0.15875cm);
\path[draw=black,fill=black!50] (21.9075,12.54125) circle (0.15875cm);
\path[draw=black,fill=black!50] (20.6375,10.31875) circle (0.15875cm);
\path[draw=black,fill=black!50] (19.3675,8.09625) circle (0.15875cm);
\path[draw=black,fill=black!50] (18.0975,5.87375) circle (0.15875cm);
\path[draw=black,fill=black!50] (16.8275,3.65125) circle (0.15875cm);
\path[draw=black,fill=black!50] (25.0825,13.335) circle (0.15875cm);
\path[draw=black,fill=black!50] (23.8125,11.1125) circle (0.15875cm);
\path[draw=black,fill=black!50] (22.5425,8.89) circle (0.15875cm);
\path[draw=black,fill=black!50] (19.84375,4.445) circle (0.15875cm);
\path[draw=black,fill=black!50] (23.1775,5.23875) circle (0.15875cm);
\path[draw=black,fill=black!50] (24.4475,7.46125) circle (0.15875cm);
\path[draw=black,fill=black!50] (25.7175,9.68375) circle (0.15875cm);
\path[draw=black,fill=black!50] (26.9875,11.90625) circle (0.15875cm);
\path[draw=black,fill=black!50] (28.8925,10.4775) circle (0.15875cm);
\path[draw=black,fill=black!50] (27.6225,8.255) circle (0.15875cm);
\path[draw=black,fill=black!50] (26.3525,6.0325) circle (0.15875cm);
\path[draw=black,fill=black!50] (29.5275,6.82625) circle (0.15875cm);
\path[draw=black,fill=black!50] (30.7975,9.04875) circle (0.15875cm);
\path[draw=black,fill=black!50] (32.7025,7.62) circle (0.15875cm);
\path[draw=black,fill=black!50] (21.2725,6.6675) circle (0.15875cm);

\filldraw[black] (10.16,14.12875) circle (0.15875cm);
\filldraw[black] (8.89,16.35125) circle (0.15875cm);
\filldraw[black] (11.43,16.35125) circle (0.15875cm);
\filldraw[black] (12.7,14.12875) circle (0.15875cm);
\filldraw[black] (15.24,14.12875) circle (0.15875cm);
\filldraw[black] (13.97,16.35125) circle (0.15875cm);
\filldraw[black] (16.51,16.35125) circle (0.15875cm);
\filldraw[black] (17.78,14.12875) circle (0.15875cm);
\filldraw[black] (20.32,14.12875) circle (0.15875cm);
\filldraw[black] (19.05,16.35125) circle (0.15875cm);
\filldraw[black] (21.59,16.35125) circle (0.15875cm);
\filldraw[black] (15.24,5.23875) circle (0.15875cm);
\filldraw[black] (13.97,7.46125) circle (0.15875cm);
\filldraw[black] (16.51,7.46125) circle (0.15875cm);
\filldraw[black] (12.7,9.68375) circle (0.15875cm);
\filldraw[black] (15.24,9.68375) circle (0.15875cm);
\filldraw[black] (17.78,9.68375) circle (0.15875cm);
\filldraw[black] (16.51,11.90625) circle (0.15875cm);
\filldraw[black] (13.97,11.90625) circle (0.15875cm);
\filldraw[black] (11.43,11.90625) circle (0.15875cm);
\filldraw[black] (19.05,11.90625) circle (0.15875cm);
\filldraw[black] (24.17657,16.39782) circle (0.15875cm);
\filldraw[black] (22.86,14.12875) circle (0.15875cm);
\filldraw[black] (21.59,11.90625) circle (0.15875cm);
\filldraw[black] (20.32,9.68375) circle (0.15875cm);
\filldraw[black] (19.05,7.46125) circle (0.15875cm);
\filldraw[black] (17.78,5.23875) circle (0.15875cm);
\filldraw[black] (16.51,3.01625) circle (0.15875cm);
\filldraw[black] (26.035,14.9225) circle (0.15875cm);
\filldraw[black] (24.765,12.7) circle (0.15875cm);
\filldraw[black] (23.495,10.4775) circle (0.15875cm);
\filldraw[black] (22.225,8.255) circle (0.15875cm);
\filldraw[black] (19.52625,3.81) circle (0.15875cm);
\filldraw[black] (22.86,4.60375) circle (0.15875cm);
\filldraw[black] (24.13,6.82625) circle (0.15875cm);
\filldraw[black] (25.4,9.04875) circle (0.15875cm);
\filldraw[black] (26.67,11.27125) circle (0.15875cm);
\filldraw[black] (27.94,13.49375) circle (0.15875cm);
\filldraw[black] (29.845,12.065) circle (0.15875cm);
\filldraw[black] (28.575,9.8425) circle (0.15875cm);
\filldraw[black] (27.305,7.62) circle (0.15875cm);
\filldraw[black] (26.035,5.3975) circle (0.15875cm);
\filldraw[black] (29.21,6.19125) circle (0.15875cm);
\filldraw[black] (30.48,8.41375) circle (0.15875cm);
\filldraw[black] (31.75,10.63625) circle (0.15875cm);
\filldraw[black] (33.655,9.2075) circle (0.15875cm);
\filldraw[black] (32.385,6.985) circle (0.15875cm);
\filldraw[black] (35.56,7.77875) circle (0.15875cm);
\filldraw[black] (20.955,6.0325) circle (0.15875cm);
\filldraw[black] (3.175,10.31875) circle (0.15875cm);
\filldraw[black] (4.445,8.09625) circle (0.15875cm);
\filldraw[black] (5.715,10.31875) circle (0.15875cm);
\filldraw[black] (7.62,8.89) circle (0.15875cm);

\end{tikzpicture}%
 \caption{$\overline \Star (F;\T)$ (left) and the subdivision of
  $c\cdot\overline \Star (F;\T)$ into Cayley polytopes via concentric
  dilated copies of $\partial \Star (F;\T)$ (right)}
\label{fig:KMW-concentric}
\end{figure}

Let us look at how much the number of simplices
in a triangulation grows with each iteration of Lemma~\ref{lemma:kmw:pull}. One crude bound is that the number of simplices produced in the refinement of each individual dilated simplex
$c\Delta$ is at most the total number of simplices in the
refinement.

\begin{lemma}
\label{lemma:number_of_simplices}
In the triangulations constructed in Lemma~\ref{lemma:kmw:pull} and
Corollary~\ref{coro:simultaneous}, each dilated simplex $c\Delta$,
$\Delta\in \T$, is refined into at most $(d+1) c^d$ full-dimensional
simplices. 
\end{lemma}

\begin{proof}
Observe that the combinatorial type, hence the number of simplices, of
$\T_m$ depends only on two parameters (apart of $d$ and $c$): how many
of the $\{m_i\}$'s are not zero (call this number $d_0$), and the
value of $c_0= \sum_i \{m_i\}$. In particular, to compute the number
of simplices there is no loss of generality in assuming that all the
non-zero $\{m_i\}$'s are equal to one another, hence equal to
$c_0/d_0$. In this case, all the simplices in $\T_\m$ have volume
exactly $c_0/d_0$ times the volume of $\Delta$, so the number of them
needed to fill $c\Delta$ equals 
\[
\frac{d_0 c^d}{c_0} \le (d+1) c^d.
\]
In the last inequality we use that $d_0\le d+1$ and $c_0\ge 1$.
\end{proof}

\subsection{A proof of the KMW Theorem} \label{sec:KMWfirstproof}

Lemma~\ref{lemma:kmw:pull} can be applied simultaneously to several
faces $F_1,\dots, F_N$ each with a non-degenerate box point
$\m_1,\dots,\m_N$, as long as the stars of the $F_i$'s intersect only
on their boundaries. Equivalently, as long as there is not an
$i,j$ pair for which $\conv(F_i\cup F_j)$ is a simplex, or as long as 
the open stars of the $F_i$'s are disjoint. Here the \emph{open
  star} of a face in a simplicial complex is defined as:
\[
\Star(F;\T) = \{F'\in \T: F\subset F'\}.
\]
(Observe that open stars are not simplicial complexes.) 
 
In the following statement, we assume that $c$ is a multiple of $(d+1)!$ to
guarantee that it is a multiple of $\dim(F_i) +1$ for every $F_i$, as required in order to apply
Lemma~\ref{lemma:kmw:pull}.
  
\begin{corollary}
  \label{coro:simultaneous}
Let $\T$ be a lattice triangulation (with an ordering of its
vertices). Suppose there is a family $F_1,\dots,F_N$ of faces of $\T$
with respective non-degenerate box points $\m_1,\dots, \m_N$ such that
$\conv(F_i \cup F_j) \not\in \T$, for every pair $i,j\in [N]$.

Then, for every integer $c\in (d+1)! \N$, the dilation $c\T$ can be
refined into a triangulation $\T'$ with the following properties: 
\begin{enumerate}
\item $\T'$ is the canonical refinement of $c\T$ away from  $\cup_i
  \overline\Star(F_i;\T)$. In particular, for every full-dimensional
  simplex $\Delta'\in \T'$ not in $\cup_i \overline\Star(F_i;\T)$,
  $L_{\Delta'}=L_\Delta$, where $\Delta\in \T$ is such that
  $\Delta'\subset c\Delta$.
\item For each full-dimensional simplex $\Delta'$ of $\T'$ contained
  in $\cup_i \overline\Star(F_i;\T)$,  $\vol(\Delta') < \vol(\Delta)$
  for $\Delta\in \T$ is such that $\Delta'\subset c\Delta$.
\item $\T'$ is a regular refinement of $\T$.
\end{enumerate}
\end{corollary}
 
\begin{proof}
The condition $\conv(F_i \cup F_j) \not\in \T$ is equivalent to saying
that the closed stars $\overline \Star (F_i;\T)$ and $\overline \Star
(F_j;\T)$ intersect only on their boundaries. So,
Lemma~\ref{lemma:kmw:pull} can be applied simultaneously with respect
to all box points and stars, since it gives the canonical
triangulation on the boundary.
 
We know the components match up to form a regular refinement since
refining every simplex of $c\T$ to its canonical triangulation is
clearly a regular refinement.  By part \eqref{lemma:kmw:pull-iii} of
Lemma~\ref{lemma:kmw:pull}, if we keep the weights that give the
canonical triangulations outside the stars of the $F_i$'s (including
their boundaries) we can extend that to give our desired refinement
inside the stars.
\end{proof}
 
With this we can conclude the following statement, which implies
Theorem~\ref{thm:kmw} by induction on $V$.
 
\begin{theorem}
  \label{thm:kmw-inductive}
Let $\T$ be a triangulation of $P$, and $V$ be the maximal
volume among its simplices. If $\Delta_1,\dots,\Delta_N$ is the
full-dimensional simplices of volume $V$, then applying
Corollary~\ref{coro:simultaneous} $N$ times with any $c\in (d+1)!\N$ 
yields  a regular refinement of $c^N \T$ into a triangulation with
maximal volume strictly less than $V$.
\end{theorem}
 
\begin{proof}
Our goal is to apply apply Corollary~\ref{coro:simultaneous} once (at most) for each of
the $N$ simplices of volume $V$. The only difficulty is that with each
application the number of simplices of volume $V$ grows. 
Indeed,  for each such
simplex $\Delta$ not lying in the star of one of the $F_i$'s that we
use in a particular step, $c\Delta$ is refined into its canonical triangulation,
which consists of a (huge) number of
simplices of the same volume $V$. Our proof takes advantage of the fact that
all these simplices form a canonical triangulation, to refine (dilations of) all of
them simultaneously into simplices of volume strictly smaller than $V$.

Suppose that a triangulation $\T$ contains the canonical
triangulation $\T_{k\Delta}$ of a dilated simplex $k\Delta$, for some
$k$. Choose a box point $\m_0$ of one of the simplices $\Delta_0 \in
\T_{k\Delta}$. Since all the simplices of the canonical triangulation
define the same lattice, $\m_0$ can also be considered a box point in
every other simplex of $\T_{k\Delta}$. So, let $F_1,\dots, F_M$ ($M\le
k^d$) be the list of minimal faces of $\T_{k\Delta}$ for which $\m_0$
is a box point. 
We claim that the stars of the $F_i$'s are disjoint and that they
cover $\T_{k\Delta}$. (The second claim is obvious; the first is also
not hard, and a proof of it in a more general context is in
Lemma~\ref{lemma:independent box points} below.) Hence, we can get rid
off all the simplices of $\T_{k\Delta}$ in a single step.

So, our algorithm for refining $c^N\T$ in $N$ steps is:
  In the first step, we choose a box point $\m_1$ for $\Delta_1$,
  and apply Corollary~\ref{coro:simultaneous} to that box point alone,
  to get a triangulation $\T_1$ in which all the simplices of volume
  $V$ belong to the canonical triangulations of $c\Delta_2$,
  $c\Delta_3, \dots, c\Delta_N$.
After the $i$-th step, we have a triangulation $\T_i$ in which
  all the simplices of volume $V$ belong to the canonical
  triangulations of $c^i\Delta_{i+1}$, \dots, $c^i\Delta_N$. We choose
  a box point $\m_{i+1}$ for one simplex in $c^i\Delta_{i+1}$ and
  apply Corollary~\ref{coro:simultaneous} simultaneously to the stars
  of all the minimal faces of $c^i\Delta_{i+1}$ having $\m_{i+1}$ as a
  box point.
\end{proof} 

\begin{remark}
\label{rem:tower-of-exponentials}
As pointed out by Bruns and Gubeladze~\cite[Remark 3.20]{BG}, ``One
could try to give an effective upper bound for the number c in
Theorem~\ref{thm:kmw} by tracing its proof''. However, neither Knudsen
et al.~nor Bruns and Gubeladze do this.
Part of the reason is that the bound obtained ``by tracing the proof''
would certainly be a tower of exponentials of length at least the
initial maximal volume $V$.

This is because applying the result for a given initial maximal volume
$V$ yields a triangulation with maximal volume at most $V-1$, but we
do not have easy control on the new number $N'$ of simplices of volume
$V-1$. Using Lemma~\ref{lemma:number_of_simplices} for this would give
$N'=c^{Nd} \vol(P)$. Then, in the second step we then get that our bound
for the number of cells of volume $V-2$ is $c^{c^{Nd} \vol(P)d}\vol
(P)$, in the third step we get $c^{c^{c^{Nd} \vol(P)d}\vol
  (P)d}\vol(P)$, and so on.
\end{remark}

Observe that if all simplices of maximal volume in $\T$ have non-degenerate box points whose 
carriers have disjoint stars then we only need to apply 
Corollary~\ref{coro:simultaneous} once to get a triangulation with smaller maximal volume, leading to a bound of $c$ instead of $c^N$ in Theorem~\ref{thm:kmw-inductive}. If, moreover, the disjoint-stars property could be preserved in the iteration, we would get an effective KMW Theorem with a singly exponential factor of type $(d+1)!V$, rather than the double exponential of our Theorem~\ref{thm:KMW-effective}.

Unfortunately, the hypothesis that stars are disjoint is not always
satisfied and seems not easy to be preserved. However, it automatically holds when $V$ is a prime.

\begin{lemma}
\label{lemma:prime}
If $\Delta$ is a lattice simplex of volume $V$ with respect to a
certain lattice $\Lambda$, then:
\begin{enumerate}
\item $\Delta$ has a unique minimal face $F_0$ of volume $V$.
\item  If $V$ is prime, then every (non-zero) box point of $F_0$ is
  non-degenerate.
\end{enumerate}
\end{lemma}

\begin{proof}
Let $F_1$ and $F_2$ be two faces of volume $V$ of $\Delta$. $F_1$
and $F_2$ cannot be disjoint, because the volume of a join is at
least the product of the volumes. So, let $F_3=F_1 \cap F_2$. Then
either $F_3$ also has volume $V$ or $\Delta$ has volume greater than
$V$. More precisely, let $\Lambda_i$, $i=1,2,3$, be the sublattice
of $\Lambda$ in the affine span of $F_i$. Then, the volume of
$\Delta$ in the sublattice $\Lambda_1 + \Lambda_2$ times the volume
of $F_3$ in $\Lambda_3$ equals the volume of $F_1$ in $\Lambda_1$
times that of $F_2$ in $\Lambda_2$. This proves part (1).

For part (2), observe that $V$ being prime implies that $F_0$ is  the
intersection of all the faces of $\Delta$ that are not unimodular. In
particular, all proper faces of $F_0$ are unimodular and every box
point of $F_0$ is non-degenerate. 
\end{proof}

\begin{corollary}
\label{coro:prime}
If $\T$ is a lattice simplicial complex in which no simplex has volume
larger than $V$, then the minimal faces of volume $V$ in $\T$ have
disjoint stars. Further, if $V$ is prime, these minimal faces of
volume $V$ have non-degenerate box points.
\qed
\end{corollary}

\begin{corollary} \label{coro:kmw:prime}
Let $\T$ be a lattice triangulation with all its faces of volume at
most $V$, for a prime $V$.
Then,  for every $c\in (d+1)!\N$, the dilation $c\T$ has a regular
refinement in which every cell has volume less than $V$.
\end{corollary}

\begin{proof}
Let $F_1,\dots, F_N$ be the list of minimal faces of volume $V$ in
$\T$. By Corollary~\ref{coro:prime}, they have disjoint stars and
possess non-degenerate box points $\m_1,\dots,\m_N$. 
We then satisfy the conditions of Corollary~\ref{coro:simultaneous},
with the added feature that every simplex of volume $V$ is in one of
the stars of the $F_i$'s. 
\end{proof}

\subsection{An effective version of the KMW-Theorem} \label{sec:KMW-effective}

Our proof of Theorem~\ref{thm:kmw-inductive} (from which Theorem~\ref{thm:kmw} follows)
already departs from the previous proofs in that we use the ``canonical triangulation'' structure
of the pieces after each refinement in order to control the number of iterations needed
to decrease the maximum volume of simplices. This avoids us the use of ``local lattices''~\cite{BG} or ``rational
structures''~\cite{KKMS3}, and it gives a cleaner proof of regularity. 

Here we push this idea further and relate by an ``$\rootA$-structure'' simplices that do not come
from the canonical triangulation of the same dilated simplex.
As usual, we let $L_\Delta$ denote the (affine) lattice generated by the vertices
of a simplex $\Delta$, and say that two full-dimensional ordered lattice simplices
$\Delta=\conv\{\va_0,\dots,\va_d\}$ and
$\Delta'=\conv\{\vb_0,\dots,\vb_d\}$  are \emph{$\rootA$-equivalent}
if $L_\Delta=L_{\Delta'}$ and
\[
\{\va_{i} - \va_{i-1} : i \in [d]\}
=
\{\vb_{i} - \vb_{i-1} : i \in [d]\}.
\]

Lemma~\ref{lemma:rootA-equivalent} offers ties between the canonical
triangulation of $c \Delta$ and $\rootA$-equivalence.

\begin{lemma}
\label{lemma:rootA-equivalent}
\begin{enumerate}
\item All the simplices in the canonical triangulation of $c\Delta$ are $\rootA$-equivalent to $\Delta$.
\item If two simplices $\Delta$ and $\Delta'$ are $\rootA$-equivalent then
the $\rootA$-dicing defined by $\Delta$ and by $\Delta'$ are the same, modulo a translation in $\Lambda$. \qed
\end{enumerate}
\end{lemma}

The second part of Lemma ~\ref{lemma:rootA-equivalent} allows us to
understand a box point $\m$ for a simplex $\Delta$ as a box point for
any $\rootA$-equivalent simplex $\Delta'$. Specifically, let $\m'$ be
the translation of $\m$ obtained by sending the $\rootA$-dicing of
$\Delta$ to that of $\Delta'$. This translation is unique, modulo
the linear lattice $\vec L_\Delta$. 
The key property of $\rootA$-equivalence that we need is the
following:  

\begin{lemma}
\label{lemma:independent box points}
Let $\Delta$ and $\Delta'$ be two $\rootA$-equivalent full-dimensional
simplices in a triangulation $\T$, and $\m$ be a box point for
$\Delta$ and $\m'$ the corresponding box point for $\Delta'$. If $F$
and $F'$ be the carrier faces of $\m$
and $\m'$ in $\Delta$ and $\Delta'$ (that is, $F$ and $F'$ are the faces for
which $\m$ and $\m'$ are  non-degenerate box points), then either
$F=F'$ or $\conv(F\cup F')$ is not a simplex in $\T$. 
\end{lemma}

\begin{proof}
Although we defined box points as equivalent classes, let us now 
think of $\m$ and $\m'$ as representatives for their classes. That is,
$\m$ and $\m'$ are two lattice points in the ambient lattice $\Lambda$.

Let $F_0$ be the minimal face of the $\rootA$-dicing of $\Delta$
containing $\m$. $F_0$ might not be equal to $F$, since $F_0$ may not be
a face of $\Delta$, but without loss of generality (by changing the representative for
$\m$ if needed) we can assume that $\aff(F_0)\subset
\aff(F)$. (This is because the flat spanned by $F$ contains representatives for $\m$,
and the $\rootA$-dicing induced by $\Delta$
is a hyperplane arrangement refining the flat spanned by every face of $\Delta$).
For the same reason, we assume that
the minimal face $F'_0$, of the $\rootA$-dicing of $\Delta'$
containing $\m'$, satisfies $\aff(F'_0)\subset \aff(F')$. 

Suppose now that $\conv(F\cup F')$ is a simplex in $\T$. Since
$\aff(F_0)$ and $\aff(F'_0)$ are parallel flats contained in respective 
faces of $\conv(F\cup F')$, we must have $\aff(F_0)=\aff(F_0')$, so
$F\cap F'$ is not empty. In particular, the $\rootA$-dicings of
$\Delta$ and $\Delta'$ are not only translations of one another, but
actually equal. Hence, $\m=\m'$, modulo the
(linear) lattice $\vec L_\Delta= \vec L_\Delta'$. 

By Lemma~\ref{lemma:prime}, there is a unique face of $\conv(F\cup
F')$ for which $\m=\m'$ is a non-degenerate box point. That face must be
$F=F'$. 
\end{proof}

Hence, if we choose an $\rootA$-class of simplices of a triangulation
$\T$, Corollary~\ref{coro:simultaneous} can be applied simultaneously
to faces whose stars contain all the simplices in that class. 

\begin{corollary}
\label{coro:KMW-effective}
Let $\T$ be a triangulation with a total order in its vertices. If $\Delta_1, \dots ,\Delta_N$ are
representatives for all the $\rootA$-equivalence classes of
full-dimensional simplices in $\T$, and $V_1,\dots,V_N$ are their respective
volumes, then the total number of $\rootA$-equivalence classes of simplices
used to obtain a unimodular triangulation by iterative applications of
Theorem~\ref{thm:kmw-inductive} is bounded above by: 
\[
\sum_{i=1}^N V_i! \left( (d+1)! c^d\right)^{V_i-1}
\]
\end{corollary}

\begin{proof}
The statement follows from the claim that
the number of $\rootA$-equivalence classes that arise in the refinements
of all the dilations of simplices of a particular $\rootA$-equivalence class of volume $V$ is bounded
above by 
\[
V! \left( (d+1)! c^d\right)^{V-1}.
\]

We prove this claim via induction on $V$, starting with the case
$V=1$. In this case the number of classes is one. 
Now consider the case where $V > 1$, and take a particular $\Delta$ in
a certain $\rootA$-equivalence class. Once its box point $\m$ is chosen, 
$c\Delta$ gets refined, by Lemma~\ref{lemma:number_of_simplices},
into at most $(d+1) c^d$ simplices of volume at most $V-1$.
By inductive hypothesis each of them will produce at most
\[
(V-1)! \left( (d+1)! c^d\right)^{V-2}
\]
$\rootA$-classes when further and further refined. So, the total
number of new classes produced by refining $\Delta$  is bounded by
\begin{eqnarray*}
&& (d+1) c^d (V-1)! \left( (d+1)! c^d\right)^{V-2} \\
 &=&\frac{V! \left( (d+1)! c^d\right)^{V-1}}{Vd!}.
\end{eqnarray*}

Now, it is not true that all choices of $\Delta$ in the same
$\rootA$-class will produce the same $\rootA$-classes when refined. 
On the one hand, we must take into account that there are $d!$
translation classes within each $\rootA$-class. On the other hand,
simplices $\Delta$ and $\Delta'$ that are translations of one another
may get different refinements, because the refinement depends on the choice of the box point
$\m$. But these two (translation class and box point) are the only things that 
can make the dilations of two $\rootA$-equivalent simplices
have canonical refinements that are not $\rootA$-equivalent to one another.

Since there are $V-1$ choices for $\m$, the total number of
classes produced from an individual class of volume $V$ (including the
initial class itself) is bounded above by 
\begin{eqnarray*}
&&1 + d!(V-1)\frac{V! \left( (d+1)! c^d\right)^{V-1}}{Vd!}\\
&<&V! \left( (d+1)! c^d\right)^{V-1}.
\end{eqnarray*}
\end{proof}

This implies the following, which in turn proves
Theorem~\ref{thm:KMW-effective}: 

\begin{theorem}
\label{thm:KMW-effective2}
Let $\T$ be a triangulation of a lattice polytope $P$. In the
notation of Corollary~\ref{coro:KMW-effective},
for every $c\in (d+1)!\N$ and for every  $M\ge \sum_{i=1}^N V_i!
\left( (d+1)! c^d\right)^{V_i-1}$, the dilation $c^M \T$ has a regular
unimodular refinement.
\end{theorem}

\begin{proof}
Apply Corollary~\ref{coro:simultaneous} to each $\rootA$-class of
simplices in $\T$ (and the new ones created in the process) starting
with those of higher volume. At each iteration, consider box points
and faces whose stars completely cover one $\rootA$-class of maximal
volume, which can be done by Lemma~\ref{lemma:independent box
  points}. Since that class will not appear in future steps (because only classes of
  strictly smaller volume can appear), the total
number of iterations is bounded above by the total number of
$\rootA$-classes that can arise in the process, which is bounded by $M
$ according to Corollary~\ref{coro:KMW-effective}. 
\end{proof}

\begin{remark}
This proof does not claim that all the simplices of a particular class
that appear anywhere in the process are refined in the same step. Some may
get refined ``before their turn'' if  they are in the stars of the
faces $F_i$ used to refine other classes. But this does not invalidate
the proof. We devote one iteration to addressing whatever is left from
each particular class, so the number of iterations is at most the
number of classes.
\end{remark}

\newcommand{\arxiv}[1]{\href{http://arxiv.org/abs/#1}{\texttt{arxiv:}#1}}

\bibliographystyle{amsplain}
\def\cprime{$'$}
\providecommand{\bysame}{\leavevmode\hbox to3em{\hrulefill}\thinspace}
\providecommand{\MR}{\relax\ifhmode\unskip\space\fi MR }
\providecommand{\MRhref}[2]{%
  \href{http://www.ams.org/mathscinet-getitem?mr=#1}{#2}
}
\providecommand{\href}[2]{#2}

\bibliographystyletodo{plain}

\end{document}